\documentclass[1p]{elsarticle}
\usepackage[T1]{fontenc}
\usepackage{amsmath,amssymb,latexsym,mathrsfs}
\usepackage{mathabx} 
\usepackage{xcolor}
\usepackage[all,color,cmtip]{xy}
\usepackage{rotating}
\usepackage{amsthm}
\usepackage{mathtools}
\usepackage{enumitem}
\usepackage{titlesec}
\usepackage{tikz}
\biboptions{sort&compress}
\usepackage{blkarray}
\usepackage{hyperref} 

\usepackage{csquotes}
\usepackage{xcolor}
\usepackage{bbm}
\usepackage{mathtools}

\makeatletter
\def\ps@pprintTitle{%
  \let\@oddhead\@empty
  \let\@evenhead\@empty
  \def\@oddfoot{\reset@font\hfil\thepage\hfil}
  \let\@evenfoot\@oddfoot
}
\makeatother

\titleformat{\section}
  {\normalfont\fontsize{12}{15}\bfseries}{\thesection}{1em}{}

\titleformat{\subsection}
  {\normalfont\fontsize{11}{15}\bfseries}{\thesubsection.}{0.5em}{}

\def\stringmod#1{M_{#1}}

\def\a{\alpha}
\def\kk{K}
\def\La{\Lambda}
\def\ideal{R}
\def\Car{C}

\newcommand{\perm}{\mathcal{P}}
\newcommand{\forb}{\mathcal{F}}
\newcommand{\rad}{\mathrm{rad}}
\newcommand{\inc}{\mathrm{inc}}
\newcommand{\Ker}{\mathrm{ker} }
\newcommand{\Rnk}{\mathrm{rk}}
\newcommand{\tr}{{\rm{tr}}}
\newcommand{\Dyn}{\mathrm{Dyn}}
\newcommand{\R}{\mathbb{R}}
\newcommand{\Z}{\mathbb{Z}}

\newcommand{\M}{\mathbb{M}}
\newcommand{\MM}{\mathbb{M}}

\newcommand{\A}{\mathbb{A}}
\def\AA{\mathbb{A}}

\newcommand{\CC}{\mathbb{C}}
\newcommand{\C}{\mathcal{C}}
\newcommand{\D}{\mathbb{D}}
\newcommand{\DD}{\mathbb{D}}
\newcommand{\E}{\mathbb{E}}
\newcommand{\EE}{\mathbb{E}}

\newcommand{\CRnk}{\mathrm{crk}}
\newcommand{\crk}{\mathrm{crk}}
\newcommand{\Id}{\mathbf{I}}
\newcommand{\roots}{\mathcal{R}}
\newcommand{\bas}{\mathrm{e}}
\newcommand{\sgn}{\mathrm{sgn}}

\newcommand{\bulito}{\mathmiddlescript{\bullet}}
\newcommand{\diag}{\mathrm{diag}}
\newcommand{\wt}[1]{\widetilde{#1}}

\newcommand{\edges}{\mathbbm{v}}
\newcommand{\bdg}{\mathbf{B}}
\newcommand{\bdgA}{\mathbf{A}}
\newcommand{\bdgD}{\mathbf{D}}
\newcommand{\bdgC}{\mathbf{C}}
\newcommand{\Null}{\mathrm{Null}}
\def\verts#1{V(#1)} 
\def\edgs#1{E(#1)}  
\def\posloops#1{E(#1)^{+\ell}}
\def\shposloops{E^{+\ell}}
\def\negloops#1{E(#1)^{-\ell}}
\def\extbdg#1{\dot{#1}} 
\def\extwalks{\mathfrak{w}}
\def\walks{\mathfrak{w}}
\def\walksred{\overline{\mathfrak{w}}}
\def\wlk{\omega} 
\def\signbigr{\sigma} 
\def\CR{\mathcal{R}}
\def\CRinc{\mathcal{R}^{\rm inc}}
\def\GlnZ{{\rm Gl}_n(\Z)}

\def\ZZ{\Z}
\def\cyccond{\beta}
\def\CD{\mathcal{D}}
\def\CA{\mathcal{A}}
\def\CB{\mathcal{B}}
\def\CF{\mathcal{F}}
\def\CG{\mathcal{G}}
\def\weak{\sim}
\def\Gweak{\sim_{{\rm G}}}

\def\switch{\sim_s}
\def\un{\underline{n}}
\def\S{\sigma} 
\def\T{\tau} 
\def\refl{\mathrm{r}} 
\def\onevertexbdg{\mathbf{L}}

\def\dl{{\mathrm{dl}}} 
\def\Ci{{({\scriptsize${\mathbb{C}}$}1)}}
\def\Cii{{({\scriptsize${\mathbb{C}}$}2)}}
\def\Ciii{{({\scriptsize${\mathbb{C}}$}3)}}

\def\dim{{\rm dim}}
\def\pp{{\bf p}}
\def\qq{{\bf q}}
\def\quiv{\Gamma}
\def\hI{\hat{I}}
\newcommand{\swi}{\mathrm{sw}}

\newcommand\mathmiddlescript[1]{\vcenter{\hbox{$\scriptstyle #1$}}}

\numberwithin{equation}{section}

\newtheorem{lemma}[equation]{Lemma}{\bf}{\it}
\newtheorem{proposition}[equation]{Proposition}{\bf}{\it}
\newtheorem{corollary}[equation]{Corollary}{\bf}{\it}
\newtheorem{theorem}[equation]{Theorem}{\bf}{\it}
\newtheorem{problem}[equation]{Problem}{\bf}{\it}
\newtheorem*{theoremA}{Theorem A}{\bf}{\it}
{\bf}{\it}
\newtheorem*{theoremB}{Theorem B}{\bf}{\it}
\newtheorem*{theoremC}{Theorem C}{\bf}{\it}
\theoremstyle{definition}
\newtheorem{remark}[equation]{Remark}{\bf}{\rm}
\newtheorem{alg}[equation]{Algorithm}{\bf}{\rm}
\newtheorem{definition}[equation]{Definition}{\bf}{\rm}
\newtheorem{example}[equation]{Example}{\bf}{\rm}

\theoremstyle{remark}

\begin{document}
\title{Bidirected graphs, integral quadratic forms and some Diophantine equations}

\author{Jes\'us Arturo Jim\'enez Gonz\'alez\fnref{UNAM,cora}} 

\ead{jejim@im.unam.mx}

\author{Andrzej Mr\'oz\fnref{UMK}}
\ead{amroz@mat.umk.pl}

\fntext[UNAM]{Instituto de Matem\'aticas UNAM. 
Ciudad~Universitaria C.P.~04510, Mexico City.}
\fntext[UMK]{Faculty of Mathematics and Computer Science, Nicolaus Copernicus University, ul. Chopina 12/18, 87-100 Toru\'n, Poland}

\fntext[cora]{Corresponding author}
\journal{\ }

\begin{abstract}
Bidirected graphs are multigraphs where every edge has an independent direction at each end. In the paper, with an arbitrary bidirected graph we associate a non-negative integral quadratic form (called the incidence form of the graph), and determine all forms that appear in this way in two main results: first, among non-negative connected unit forms, precisely those of Dynkin type $\mathbb{A}$ or $\mathbb{D}$ are incidence forms; second, we give simple conditions on the coefficients of a non-negative connected non-unitary form to be an incidence form. We say that those non-unitary forms have Dynkin type $\mathbb{C}$, and justify such nomenclature by generalizing known classifications and properties of non-negative quadratic forms of Dynkin types $\mathbb{A}$ and $\mathbb{D}$ to the introduced type~$\mathbb{C}$. We also show that the graphical framework of an incidence form is an useful tool to visualize its arithmetical properties, to prove new facts and to perform efficient computations for integral quadratic forms and related problems in number theory, algebra and graph theory. For instance, in a third main result we relate the walks of a bidirected graph with the $0,1,2$-roots of the associated incidence form (and to the classical root systems in the positive case). Moreover, we prove the universality property for a large class of integral quadratic forms, provide computational methods to find solutions or to characterize the finiteness of the sets of solutions of various related Diophantine equations, show a variant of Whitney's theorem on line graphs using switching classes, and apply our techniques to give a conceptual and constructive proof of the non-negativity (and possible Dynkin types) of the Euler quadratic forms of a class of finite-dimensional gentle algebras.
\end{abstract}


\maketitle


\section{Introduction}\label{sec:intro}

Integral quadratic forms, that is, mappings $q: \Z^n \to \Z$ defined by a homogeneous polynomial of the second degree
\begin{equation}\label{EQ:iqf}
q(x_1,\ldots,x_n)=\sum_{i=1}^n q_ix_i^2 + \sum_{i<j}q_{ij}x_ix_j,
\end{equation}
with integer coefficients $q_i,q_{ij} \in \Z$, appear in many mathematical contexts, especially in number theory and algebra, but also in geometry, topology and some branches of computer science. They are often associated to more intricate structures to reflect their certain properties, take for instance Euler or Tits quadratic form of a finite dimensional associative algebra~\cite{ASS06,BPS11,Kac80, cmR}, or integral form induced by the Killing form of a (semi-simple) complex Lie algebra~\cite{BGZ06,jeH72}, cf.~\cite{BKL06,PR19}. Natural questions concerning integral quadratic forms, like the classification up to changes of basis of the group $\Z^n$ (the $\ZZ$-equivalence relation $\weak$, see~(\ref{eq:equivalence})) or searching for solutions of induced Diophantine equations
\begin{equation}\label{EQ:dio}
q(x_1,\ldots,x_n)=d,
\end{equation}
for fixed $d \in \Z$ lead to highly non-trivial problems (cf.~Xth Hilbert Problem \cite{Hilbert}) and have been a guiding light in mathematics for many years. For example, Lagrange's Theorem asserts that every integer $d \geq 1$ can be expressed as the sum of four square numbers, that is, equation~(\ref{EQ:dio}) always has a solution for such $d$ values and $q(x_1,x_2,x_3,x_4)=x_1^2+x_2^2+x_3^2+x_4^2$, see~\cite{BJP19} and subsection~\ref{subsec:universal} below. 
The set of solutions to the equation~(\ref{EQ:dio}) for $d \in \Z$ is denoted by
\begin{equation}\label{EQ:roo}
\mathcal{R}_q(d):=q^{-1}(d)=\{x \in \Z^n \colon q(x)=d\}
\end{equation}
and its elements  are called  \textbf{$d$-roots} of $q$. If the set $\CR_q(d)$ is non-empty, then we say that $q$ {\bf represents} the integer $d$. If $q$ represents all integers (or all non-negative integers in case $q$ is non-negative), then we call $q$ a {\bf universal} form, cf.~\cite{Ram,BH290}. We note that the sets of (positive) $0$-roots and $1$-roots of the Euler form of a finite-dimensional  algebra are of special importance -- in certain cases they classify indecomposable modules, see the celebrated result of Kac \cite{Kac80}, see also \cite{ASS06,cmR}.

On the other hand, the classification of integral forms up to $\weak$ is a largely unsolved, highly non-trivial problem, compare with the well-known classification of real quadratic forms up to $\mathbb{R}$-equivalence, following from Sylvester's Law of Inertia. Classifications of (certain kinds of) integral quadratic forms in small number of variables $n$ follow from deep results of many authors, among others, Gauss~\cite{cfG66}, Hasse, Minkowski, Eichler, Witt~\cite[Ch.~15]{CS99}; for $n \geq 24$ the problem is considered ``impracticable'', see~\cite[p.~353]{CS99}.

For these reasons it is natural to study subclasses of all integral quadratic forms defined by some restrictions on their coefficients. In the present paper we focus mainly on \textbf{unit} (resp. \textbf{semi-unit}) \textbf{forms}, that is, integral  forms~\eqref{EQ:iqf} whose diagonal coefficients satisfy $q_i=1$ (resp. $q_i \in \{0,1\}$), for $i=1,\ldots,n$. More generally, we consider so-called {\bf fully regular} and {\bf Cox-regular} integral quadratic forms whose integer coefficients satisfy certain divisibility conditions (see Definition \ref{def:Coxreg}) which imply that the associated Weyl group and the Coxeter transformation is integral, cf.~\cite{KS15a, MZ22}. Cox-regular forms are closely related to the so-called quasi-Cartan matrices in the sense of Barot-Geiss-Zelevinsky \cite{BGZ06} (see \cite{MM19} for more details, cf.~\cite{dS22, MM21qa, PR19}) and generalized intersection matrices in the sense of Slodowy \cite{Slodovy} (see also Gabrielov \cite{Gabrielov} and Simson \cite{dS22}, cf.~Remark \ref{rem:GabrielovLie}).

The $\weak$-classification problem for these classes seems to be much simpler (than for arbitrary integral quadratic forms); however, it remains non-trivial and is still unsolved. Classification of all non-negative  semi-unit forms was given by Barot-de la Pe\~na in~\cite{BP99} by means of the Dynkin type $\Dyn(q) \in \{\A_m,\D_m,\E_6,\E_7,$ $\E_8\}$, and the corank $\CRnk(q) \geq 0$ of a quadratic form $q$, see Theorem~\ref{T:wea}, cf.~\cite{dS16a, SZ17}. Note that the study of non-negative (semi-)unit forms is important form the point of view of algebraic applications, since, for instance, Euler and Tits forms of an algebra are often unitary (see~\cite{ASS06,cmR}) and the definiteness is a relevant property in the study of representation theory, cf.~\cite{BPS11} and references therein.

In the paper we propose a novel approach to the study of non-negative integral quadratic forms in terms of bidirected graphs, that is, multigraphs where edges can be directed (having one tail and one head) or bidirected (having two tails or two heads), see details in Definition~\ref{D:bdg}. Inspired by the line (signed) graph construction in classical graph theory \cite{tZ08, CGSS76}, we define the integral quadratic form $q_{\bdg}$ associated to a bidirected graph $\bdg$ by means of its incidence matrix $I(\bdg)$. It appears that this construction helps to study a large subclass of non-negative quadratic forms in relatively simple combinatorial terms. More precisely, we prove the following three main results. Theorems A and B establish a foundational relationship between bidirected graphs and the
non-negativity of associated quadratic forms of specific Dynkin types. Recall that a bidirected graph $\bdg$ is {\bf unbalanced} if there exists a closed walk in $\bdg$ having odd number of bidirected arrows, in which case we set $\cyccond_\bdg:=0$, otherwise $\bdg$ is {\bf balanced} and we set $\cyccond_\bdg:=1$.

\begin{theoremA}
Given a connected unit  form $q:\Z^n \to \Z$, the following two conditions are equivalent.
\begin{enumerate}[label={\textnormal{(\roman*)}},topsep=3px,parsep=0px]
\item The form $q$ is non-negative of Dynkin type $\A$ or $\D$.
\item $q=q_{\bdg}$ for some $($connected$)$ bidirected graph $\bdg$ without loops.
\end{enumerate}
Moreover, in this case
  \begin{enumerate}[label={\textnormal{(\alph*)}},topsep=3px,parsep=0px]
 \item $\Dyn(q)=\D$ if and only if $|\verts{\bdg}| \geq 4$ and $\bdg$ is unbalanced,
 \item $\CRnk(q)=|\edgs{\bdg}|-|\verts{\bdg}|+\cyccond_{\bdg}$.
\end{enumerate}
\end{theoremA}
 The proof of  the theorem is given in subsection \ref{subsec:proofA}. We also provide  a characterization of type $\A$ in terms of bidirected graphs having no bidirected arrow, see Proposition \ref{P:null}.

The second main result provides a description of the remaining connected incidence forms in terms of the Dynkin type $\CC$ in the sense of Definition \ref{D:tyc}. We prove this theorem in subsection \ref{subsec:proofAp}.

\begin{theoremB}
Given a connected irreducible integral quadratic form $q:\Z^n \to \Z$, the following two conditions are equivalent.
\begin{enumerate}[label={\textnormal{(\roman*)}},topsep=3px,parsep=0px]
\item The form $q$ is non-negative of Dynkin type $\CC$.
\item $q=q_{\bdg}$ for some $($connected$)$ bidirected graph $\bdg$ having a bidirected loop.
\end{enumerate}
 Moreover, in this case $\cyccond_\bdg=0$ and $\CRnk(q)=|\edgs{\bdg}|-|\verts{\bdg}|.$
\end{theoremB}

Given a bidirected graph $\bdg$, we define in~\eqref{eq:inc}  a mapping  $\inc: \extwalks(\bdg) \longrightarrow \Z^{|\edgs{\bdg}|}$ where $\extwalks(\bdg)$ is the set of all  walks in $\bdg$.
The following third main theorem provides a structural and algorithmic description of the sets of solutions of Diophantine equations~(\ref{EQ:dio}) for small $d$'s.

\begin{theoremC}
Given a connected bidirected graph $\bdg$, the set of vectors \[\CRinc_\bdg:=\{\pm\inc(\wlk) : \ \wlk\in\extwalks(\bdg)\}\subseteq \Z^{|\edgs{\bdg}|}\]  satisfies
\begin{equation}\label{eq:B}
\mathcal{R}_q(0) \cup \mathcal{R}_q(1)\  \subseteq  \ \CRinc_\bdg \ \subseteq \ \mathcal{R}_q(0) \cup \mathcal{R}_q(1) \cup \mathcal{R}_q(2),
\end{equation}
where $q:=q_{\bdg}$. Moreover,
\begin{enumerate}[label={\textnormal{(\alph*)}},topsep=4px,parsep=0px]
\item $\CR_q(0) = \{\pm \inc(\wlk) : \  \, \wlk\in\extwalks(\bdg)\ \text{is a closed walk with even  number of bidirected arrows}\}$,
\item $\CR_q(1) = \{\pm \inc(\wlk) : \  \, \wlk\in\extwalks(\bdg)\ \text{is an open walk}\}$,
\item $\CRinc_\bdg\cap\CR_q(2)=  \{\pm \inc(\wlk) : \  \, \wlk\in\extwalks(\bdg)\ \text{is a closed walk with odd number of bidirected arrows}\}$,
\item $\mathcal{R}_q(0) \cup \mathcal{R}_q(1) = \CRinc_\bdg$ if and only if $\bdg$ is balanced.
\end{enumerate}
\end{theoremC}

 We prove the theorem in subsection \ref{subsec:proofB}.  The above results generalize the special case of Dynkin type $\A$ and incidence forms of usual directed graphs (quivers) studied by the first named author in~\cite{jaJ2018,jaJ2020a,jaJ2020b,jaJ2020c}, see Theorem \ref{T:mainA}. The proofs of our results are rather lengthy and technical. But once we have them, they provide a useful tool to prove new results and to perform computations on integral quadratic forms and related objects. Among others, we provide the following applications of (the methods of) Theorems A, B and C:

\begin{itemize}[topsep=3px,parsep=0px]
\item a complete classification of all non-negative integral forms of Dynkin type $\CC$ with respect to Gabrielov equivalence $\Gweak$, a stronger (than $\weak$) equivalence of Cox-regular forms having its origins in Lie theory and singularity theory (cf.~\cite{BKL06, Gabrielov, Slodovy, PR19} and Remark \ref{rem:GabrielovLie}), see Theorem \ref{T:chc}, see also Theorem \ref{T:pr1} and Proposition \ref{prop:pr2};
\item a characterization of positive incidence quadratic forms by means of the finiteness of certain sets of roots (see Proposition \ref{prop:posfinite}) and  a simple combinatorial criterion (see Corollary \ref{cor:posfinite});

\item in the positive case, a proof that the set of vectors $\CRinc_{\bdg}$ of Theorem C is a finite root system in the sense of Bourbaki (up to the removal of the zero vector, cf.~\cite[9.2]{jeH72} and~\cite{CGSS76,BH12,Kac80}), which provides a graph theoretical interpretation of reflections and root systems of certain Dynkin type, see Lemma~\ref{L:reflection} and Proposition~\ref{P:rootSys}.

\item a proof that every connected non-negative unit form or a Cox-regular form of Dynkin type $\CC$ is universal, provided its rank is at least 4, see Theorem \ref{thm:univ};
\item partial solution of so-called switching class problem (see Problem \ref{prob:switching}), a variant of Whitney's Theorem on line graphs (cf.~\cite{hW32,hC21} and~\cite{tZ08,BH12,CGSS76}), see Corollary~\ref{C:swiAr};
\item a proof that the Euler form of a gentle algebra \cite{AS87, LP20, OPS18} of finite global dimension is non-negative as well as the description of its Dynkin type, see Theorem \ref{thm:gentle}.
\end{itemize}
 Moreover, in subsection \ref{subsec:comp} we discuss the computational aspects of our results and we outline three algorithms to perform efficient computations on the incidence forms and to solve related Diophantine equations \eqref{EQ:dio}.

Throughout the text we denote by $\underline{n}$ the set $\{1,\ldots,n\}$ for $n \geq 1$. By $\bas_1=\bas^{(n)}_1,\ldots, \bas_n=\bas^{(n)}_n$ we denote the canonical basis of the abelian group $\Z^n$, whose elements are treated as column vectors. The set of $n \times m$ matrices with coefficients in a ring $R$ (usually $R=\Z$ the integer numbers or $R=\R$ the real numbers) is denoted by $\M_{n,m}(R)$. We set $\M_n(R):=\M_{n,n}(R)$ and we denote by $\GlnZ=\{M\in\M_n(\ZZ): \,{\rm det}(M)=\pm1\}$ the group of $\Z$-invertible matrices, which we often identify with automorphisms of the group $\ZZ^n$. The identity matrix in $\M_n(\Z)$ is denoted by $\Id_n$ or simply $\Id$ if the size is clear.
By the \textbf{length}  (resp.~{\bf squared norm}) of a vector $x=[x_1,\ldots,x_n]^\tr\in\R^n$ we mean the value $|x|=\sum_{i=1}^n|x_i|$ (resp.~$||x||^2=x^\tr x=\sum_{i=1}^nx_i^2$). Recall that the \textbf{nullity} of a matrix $M\in\M_{n,m}(\Z)$ is the rank of the (right) null space $\Ker(M)$ of $M$, denoted here by $\Null(M)$. The rank of $M$ is denoted by $\Rnk(M)$. Clearly, $\Null(M)=m-\Rnk(M)$.

\smallskip

{\bf Acknowledgement}

\smallskip

A part of the results of this paper was obtained during a research stay of the first author at
The Faculty of Mathematics and Computer Science of Nicolaus Copernicus University in Toru\'n
 within project {\em University Centre of Excellence \lq\lq Dynamics, Mathematical Analysis and Artificial Intelligence''}. The author expresses his gratefulness to NCU for the hospitality and for providing excellent working conditions.

\section{Definitions and basic properties}\label{sec:basics}

For a quadratic form $q:\Z^n \to \Z$ as in~(\ref{EQ:iqf}), we take for convenience $q_{ii}:=q_{i}$ and $q_{ji}:=q_{ij}$ for $i<j$. Denote by $q(-,-):\Z^n \times \Z^n \to \Z$ the bilinear form given by $q(x,y)=q(x+y)-q(x)-q(y)$ for $x,y \in \Z^n$, called the \textbf{polarization} of $q$. The (symmetric) matrix $G_q:=[q(\bas_i,\bas_j)]_{i,j} \in \M_{n}(\Z)$ is called the \textbf{Gram matrix} of $q$. Note that $q(x,y)=x^{\tr}G_qy$ and $q(x)=\frac{1}{2}x^{\tr}G_qx$ for any $x,y \in \Z^n$. We say that $q$ is \textbf{non-negative} (resp. \textbf{positive}) if $q(x) \geq 0$ (resp. $q(x) > 0$) for any $0 \neq x \in \Z^n$, that is, if the associated Gram matrix $G_q$ is positive semi-definite (resp. positive definite).
We say that $q$ is \textbf{irreducible} if, whenever $q=aq'$ for an integral quadratic form $q'$ and an integer $a \in \Z$, then we have $a=\pm 1$. The \textbf{radical} of $q$ is the set
\begin{equation}\label{eq:rad}
\rad(q):=\{x \in \Z^n \colon q(x,-)=0\}=\Ker(G_q),
\end{equation}
which is a pure subgroup of $\Z^n$. It is known that if $q$ is non-negative, then $\rad(q)=\mathcal{R}_q(0)$, see for instance~\cite[Proposition~2.8]{dS11}.
The \textbf{rank} and \textbf{corank} of $q$ are the rank and nullity of  $G_q$, and they are denoted by $\Rnk(q):=\Rnk(G_q)$ and $\CRnk(q):=\Null(G_q)=n-\Rnk(G_q)$. Note that $\CRnk(q)$ is the rank of $\rad(q)$, and in case $q$ is non-negative, $\crk(q)=0$ if and only if $q$ is positive. Given a subset $X=\{j_1<j_2<\ldots< j_r\}\subseteq \un$ we define the {\bf restriction} $q^{X}:\ZZ^r\to\ZZ$ of $q$ to $X$ as the composition
\begin{equation}\label{eq:restriction}
q^{X}:=q\circ\iota_X
\end{equation}
 with the inclusion  $\iota_X:\ZZ^r\to\ZZ^n$ given by $\bas^{(r)}_t\mapsto \bas^{(n)}_{j_t}$ for all $1\leq t\leq r$. Clearly, if $q$ is non-negative, then so is $q^X$ and $\crk(q^X)\leq \crk(q)$. 
For $m\geq 1$ and a quadratic form $q':\ZZ^{m}\to\ZZ$ we define the {\bf direct sum} $q\oplus q':\ZZ^{n+m}\to\ZZ$ as the quadratic form given by $(q\oplus q')(x_1,\ldots,x_{n+m}):=q(x_1,\ldots,x_n)+q'(x_{n+1}, \ldots,x_{n+m})$. Observe that $G_{q\oplus q'}=G_{q}\oplus G_{q'}:={\scriptsize\left[\begin{array}{cc}\!\!G_{q}\!\!&\!\!0\!\!\!\\\!\!0\!\!&\!\!G_{q'}\!\!\!\end{array}\right]}.$ 

We say that two integral quadratic forms $q,q':\ZZ^n\to \ZZ$ are $\ZZ$-{\bf equivalent} (or $\ZZ$-{\bf congruent}) if
\begin{equation}\label{eq:equivalence}
q'=q\circ T, \qquad \text{for an automorphism $T$ of $\Z^n$,}
\end{equation}
or equivalently, if $G_{q'}=T^\tr G_{q}T$ for a matrix $T$ in $\GlnZ$. We write then $q\weak q'$ or $q\weak^T q'$.
Observe that if $q\weak q'$, then $q$ is non-negative if and only if so is $q'$, and in this case $\crk(q)=\crk(q')$; in particular, $q$ is positive if and only if so is $q'$.
We say that $q$ and $q'$ are {\bf trivially equivalent} and we write $q\cong q'$ if  $q\weak^P q'$, for some
permutation matrix $P\in\GlnZ$, that is, $P(\bas_i)=\bas_{\pi(i)}$ for all $i\in\un$, where  $\pi:\un\to\un$ is a  permutation (in this case we write $P=P^\pi$).

Following~\cite{BP99,dS13}, by a \textbf{bigraph} we mean a signed multi-graph $\Delta=(\verts{\Delta}, \edgs{\Delta}, \signbigr)$ with a set of vertices $\verts{\Delta}$,  a multi-set of (undirected) edges $\edgs{\Delta}$ and a sign function $\signbigr:\edgs{\Delta}\to\{\pm 1\}$ such that given two vertices $i$ and $j$ (maybe equal), all edges between $i$ and $j$ are of the same sign. We say that a bigraph $\Delta$ is \textbf{connected} if so is the underlying multi-graph of $\Delta$. By convention, edges $e\in\edgs{\Delta}$ of $\Delta$ with $\signbigr(e)=-1$ (resp.~$\signbigr(e)=+1$) are depicted as {\bf solid} (resp.~{\bf dotted}) edges\label{p:convsigns}\footnote{Note that we use the opposite convention to that usually used in signed graph theory, cf.~\cite{tZ08, tZ82}. We follow \cite{BP99, dS13}.}. With an integral quadratic form $q:\Z^n\to\Z$ we associate a bigraph $\Delta=\Delta_q$ with $\verts{\Delta}:=\underline{n}$ as follows. For $i \neq j$ there are exactly $|q_{ij}|$ edges between vertices $i$ and $j$, all of them of sign ${\rm sgn}(q_{ij})$. 
At every vertex $i$ there are exactly $|q_i-1|$ loops, all of them of sign ${\rm sgn}(q_{i}-1)$. 
Conversely, for any bigraph $\Delta$ with $\verts{\Delta}=\underline{n}$ we denote by $q_{\Delta}$ the (unique) integral quadratic form satisfying $\Delta_{q_{\Delta}}=\Delta$.
An integral quadratic form $q$ is said to be \textbf{connected} if the corresponding bigraph $\Delta_q$ is connected (equivalently, if $q$ is not trivially equivalent to a direct sum of integral quadratic forms). Note that the bigraph $\Delta_{q^X}$ of the restriction $q^X$ coincides  with the full subbigraph $\Delta^{X}$ of $\Delta=\Delta_q$ induced by the vertices from $X$.

 The following further definitions are fundamental for our work.

\begin{definition}\label{def:Coxreg}
Let $q:\Z^n \to \Z$ be an integral quadratic form as in~(\ref{EQ:iqf}).
\begin{itemize}
 \item[a)] We say that $q$ is \textbf{semi-unitary} (resp.~\textbf{unitary}) if $q_i\in \{0,1\}$ (resp.~$q_i=1$) for all $i \in \underline{n}$.

 \item[b)] We say that $q$ is \textbf{semi-fully regular} (resp.~\textbf{fully regular}) if $q_i \geq 0$ (resp.~$q_i >0$) for all $i \in \underline{n}$, and
 \[
 \frac{q_{ij}}{q_iq_j} \in \Z, \quad \text{for any $i \neq j$ in $\underline{n}$ with $q_i>0$ and $q_j>0$.}
 \]

 \item[c)] We say that $q$ is \textbf{semi-Cox-regular} (resp.~\textbf{Cox-regular} \cite{KS15a}) if $q_i \geq 0$ (resp.~$q_i >0$) for all $i \in \underline{n}$, and
 \[
 \frac{q_{ij}}{q_i} \in \Z, \ \  \text{if $q_i>0$, \ and}\quad  {\frac{q_{ij}}{q_j} \in \Z, \ \
 \text{if $q_j>0$}},
 \]
 for any $i \neq j$.
\end{itemize}
\end{definition}

If $q$ is (semi-)unitary, then we also say that $q$ is a (semi-)unit form. Cox-regular forms are also called \textit{Roiter's integral quadratic} forms in~\cite{MZ22}, where further comments on the nomenclature and motivation for the study of such forms may be found, see also \cite{Ro78, KS15a, dS22} and Remark \ref{rem:GabrielovLie}. Note that (semi-)unitary forms are (semi-)fully regular, and that (semi-)fully regular forms are (semi-)Cox-regular. Recall that a Cox-regular quadratic form is called \textbf{classic} if $q_{ij}\leq 0$ for $i \neq j$ (cf.~\cite{Ro78, MZ22}), equivalently, if all non-loop edges in $\Delta_q$ are solid. Note that an integral form $q$ is unitary if and only if $\Delta_q$ has no loops.

\begin{remark}\label{R:irr}
(a) A Cox-regular quadratic form $q$ is irreducible if and only if ${\rm gcd}(q_1,\ldots, q_n)=1$. Each Cox-regular $q$ is an integer multiple of a unique irreducible Cox-regular $\hat{q}$. 

(b) If an integral quadratic form $q$ is non-negative, then $q_i=0$ implies that $q_{ij}=0$ for all $j\neq i$ since $\rad(q)=\mathcal{R}_q(0)$, cf.~\cite[Lemma 3.2]{jaJ2018}. This means that $q\cong q'\oplus\zeta$ for the zero  form $\zeta:\Z\to\Z$, $\zeta(x_1)=0$, and some non-negative integral form $q':\ZZ^{n-1}\to\ZZ$ of $\crk(q')=\crk(q)-1$. In particular, if $q:\Z^n\to\Z$ for $n\geq 2$ is semi-Cox-regular (resp.~semi-fully regular, semi-unitary), non-negative and connected, then $q$ is Cox-regular (resp.~fully regular, unitary).
\end{remark}

For the following main definition we refer to~\cite{EJ01}, see also~\cite{tZ08}.

\begin{definition}\label{D:bdg}
By \textbf{bidirected graph} we mean a triple $\bdg=(\verts{\bdg},\edgs{\bdg},\edges_\bdg)$ consisting of (linearly ordered) finite sets $\verts{\bdg}$ and $\edgs{\bdg}$, together with a multi-set $\edges(i)=\edges_\bdg(i) \subset \verts{\bdg} \times \{\pm 1\}$ of {\bf signed endpoints} with exactly two elements $\edges(i)=\{(u,\epsilon),(u',\epsilon')\}$, for each $i \in \edgs{\bdg}$. The elements of $\verts{\bdg}$ and $\edgs{\bdg}$ are called \textbf{vertices} and \textbf{arrows} of $\bdg$, respectively. We say that $i$ is a \textbf{directed arrow} if $\epsilon \neq \epsilon'$, and a \textbf{bidirected arrow} if $\epsilon = \epsilon'$ (in this case, $i$ is a \textbf{two-tail arrow} if $\epsilon=1$, and a \textbf{two-head arrow} if $\epsilon=-1$). We consider the following additional notions:
\begin{enumerate}[label={\textnormal{(\alph*)}},topsep=3px,parsep=0px]
\item $\bdg$ is equipped with the  \textbf{sign} function $\sigma=\sigma_\bdg:\edgs{\bdg}\to\{\pm 1\}$ defined  by $\sigma(i):=(-1)\epsilon\epsilon'$ for  $i \in \edgs{\bdg}$ with $\edges(i)=\{(u,\epsilon),(u',\epsilon')\}$. Note that with this sign, positive and negative arrows are directed and bidirected, respectively.
\item    If both first entries of the members of $\edges(i)$ coincide, that is, $u=u'$, then we say that $i\in \edgs{\bdg}$ is a \textbf{loop}. The set of all directed (resp.~bidirected) loops in $\edgs{\bdg}$ is denoted by $\posloops{\bdg}$ (resp.~$\negloops{\bdg}$).

\item A bidirected graph $\bdg'=(\verts{\bdg'},\edgs{\bdg'},\edges')$ is called a \textbf{bidirected subgraph} of $\bdg$ if $\verts{\bdg'} \subseteq \verts{\bdg}$,  $\edgs{\bdg'}\subseteq \edgs{\bdg'}$ and $\edges'(i)=\edges(i)$ for all $i\in \edgs{\bdg'}$.
\end{enumerate}
\end{definition}

In the paper we will always assume that the set of arrows $\edgs{\bdg}$ of any connected bidirected graph $\bdg$ is non-empty. We will often identify $\edgs{\bdg}$ with the set $\un=\{1,\ldots,n\}$ for $n=|\edgs{\bdg}|$, and take $\verts{\bdg}=\{u_1,\ldots,u_m\}$ for $m=|\verts{\bdg}|$ (with natural orders on both sets). We will use the following graphical convention to display bidirected graphs. For an arrow $i$ with $\edges(i)=\{(u,\epsilon),(u',\epsilon')\}$, we take
\[
\overbracket{\xymatrix{\bulito_{u} \ar[rr]^-{i}_*+<1em>{\scriptsize \text{if $\epsilon=1$, $\epsilon'=-1$}} && \bulito_{u'}}}^{\text{directed arrow}} \qquad
\overbracket{\underbracket{\xymatrix{\bulito_{u} \ar@{|-|}[rr]^-{i}_*+<1em>{\scriptsize \text{if $\epsilon=1=\epsilon'$}} && \bulito_{u'}}}_{\text{two-tail arrow}} \qquad
\underbracket{\xymatrix{\bulito_{u} \ar@{<->}[rr]^-{i}_*+<1em>{\scriptsize \text{if $\epsilon=-1=\epsilon'$}} && \bulito_{u'}}}_{\text{two-head arrow}} }^{\text{bidirected arrows}}
\]

A bidirected graph with no bidirected arrows is just a directed multigraph, which we call a \textbf{quiver}. The \textbf{underlying (multi-)graph} $\overline{\bdg}$ of a bidirected graph $\bdg$ is obtained by ignoring the directions attached to the end-points of every arrow. To be precise, $\overline{\bdg}=(\verts{\bdg},\edgs{\bdg},\overline{\edges})$, where $\overline{\edges}(i)=\{u,u'\}$ if $\edges(i)=\{(u,\epsilon),(u',\epsilon')\}$ for an arrow $i$ of $\bdg$. Two arrows $i \neq j$ in a bidirected graph $\bdg$ are said to be \textbf{incident} (resp. \textbf{parallel}) if $\overline{\edges}(i) \cap \overline{\edges}(j) \neq \emptyset$ (resp. if $\overline{\edges}(i)=\overline{\edges}(j)$), and similarly, if $u \in \overline{\edges}(i)$ for a vertex $u$ and an arrow $i$, then we say that $u$ and $i$ are \textbf{incident}. A bidirected graph $\bdg$ is  \textbf{connected} (resp. a \textbf{tree}) if so is $\overline{\bdg}$.

 We consider also an extended bidirected graph $\extbdg{\bdg}$ by adding to $\bdg$ the set of formal inverses $(\posloops{\bdg})^{-1}:=\{i^{-1}: \, i\in \posloops{\bdg}\}$ of directed loops, satisfying $\edges(i^{-1})=\edges(i)$  {(so $\sigma(i^{-1})=\sigma(i)$)} for any $i\in \posloops{\bdg}$. Take $(i^{-1})^{-1}=i$.  By a \textbf{walk} of $\bdg$ we mean a walk in $\overline{\extbdg{\bdg}}$, that is, an alternating sequence of vertices and arrows of $\extbdg{\bdg}$, starting and ending with vertices,
\begin{equation}\label{eq:walk}
\wlk=(v_0,i_1,v_1,i_2,v_2,\ldots,v_{\ell-1},i_{\ell},v_{\ell}),
\end{equation}
where $\overline{\edges}(i_r)=\{v_{r-1},v_r\}$ for $r=1,\ldots,\ell$. The integer $\ell \geq 0$ is called the \textbf{length} of the walk $\wlk$. We define the composition of walks in an obvious way. If $v_0=v_{\ell}$ (equivalently, if $\wlk$ is composable with itself), then we call $\wlk$ a {\bf closed walk}; otherwise, we call $\wlk$ an {\bf open walk}.
We extend the sign function $\sigma=\sigma_\bdg$ to all walks by setting $\sigma(\wlk)=\sigma(i_1)\cdot \ldots \cdot \sigma(i_{\ell})$ for $\wlk$ as above, and
taking $\sigma(\wlk)=+1$ if $\wlk$ is a trivial walk, that is, if $\ell=0$.

For an arrow $i \in \edgs{\bdg}$ (resp. a vertex $u \in \verts{\bdg}$) we denote by $\bas_i$ the $i$-th canonical vector of $\Z^{\edgs{\bdg}}$ (resp. by $\bas_u$ the $u$-th canonical vector of $\Z^{\verts{\bdg}}$). The \textbf{(arrow-vertex) incidence matrix} $I(\bdg)$ associated to a bidirected graph $\bdg$ is the $|\edgs{\bdg}| \times |\verts{\bdg}|$ matrix given by
\begin{equation}\label{eq:IofB}
I(\bdg)^{\tr}\bas_i=\epsilon\bas_{u}+\epsilon'\bas_{u'}, \quad \text{where $\edges(i)=\{(u,\epsilon),(u',\epsilon')\}$.}
\end{equation}
Note that the $i$-th row of $I(\bdg)$ is zero if and only if $i$ is a directed loop. In particular, $I(\bdg)$ determines the shape of $\bdg$ uniquely up to the localization of directed loops.

\begin{definition}\label{D:inc}
With any bidirected graph $\bdg$ we associate a quadratic form $q_{\bdg}$ as follows:
\begin{equation}\label{eq:qB}
q_{\bdg}(x):=\frac{1}{2}x^\tr I(\bdg)I(\bdg)^{\tr}x=\frac{1}{2}||I(\bdg)^{\tr}x||^{2}, \quad \text{for any $x \in \Z^{\edgs{\bdg}}$.}
\end{equation}
The quadratic form $q_{\bdg}$ is referred to as \textbf{incidence quadratic form} of $\bdg$.
\end{definition}

The rank and corank (that is, the rank of the left null space) of the incidence matrix $I(\bdg)$ will be denoted by $\Rnk_\bdg$ and $\CRnk_\bdg$ respectively.

\begin{lemma}\label{L:ful}
For any bidirected graph $\bdg$ with $m=|\verts{\bdg}|$ and $n=|\edgs{\bdg}|$, the associated form $q:=q_{\bdg}$ is a non-negative integral quadratic form satisfying
\begin{equation*}
q_{i} = \left\{
\begin{array}{l l}
2, & \text{if $i$ is a bidirected loop}, \\
1, & \text{if $i$ is not a loop}, \\
0, & \text{if $i$ is a directed loop},
\end{array} \right.
\end{equation*}
for any arrow $i$ of $\bdg$. Moreover, $q$ is a semi-fully regular quadratic form and the rank and corank of $q$ are given by
\[
\Rnk(q_\bdg)=\Rnk_\bdg=m-\Null(I(\bdg)) \quad \text{and} \quad \CRnk(q_\bdg)=\CRnk_\bdg=n-m+\Null(I(\bdg)).
\]
\end{lemma}
\begin{proof}
That $q$ is an integral quadratic form follows by observing that the matrix $I(\bdg)I(\bdg)^\tr$ is a symmetric integer matrix with even integers on the diagonal, see \eqref{eq:IofB} and \eqref{eq:qB}.
The non-negativity of $q$ is clear since by \eqref{eq:qB} the value of $q(x)$ is given by a squared norm.

Consider an arbitrary arrow $i \in \edgs{\bdg}$, and take $\edges(i)=\{(u,\epsilon),(u',\epsilon')\}$. By definition we have
\begin{equation*}
q_i=q(\bas_i)=\frac{1}{2}||\epsilon\bas_{u}+\epsilon'\bas_{u'}||^2 = \left\{
\begin{array}{l l}
2, & \text{if $u=u'$ and $\epsilon=\epsilon'$}, \\
1, & \text{if $u \neq u'$}, \\
0, & \text{if $u=u'$ and $\epsilon \neq \epsilon'$},
\end{array} \right.
\end{equation*}
which shows the claim on the coefficient $q_i$. This claim also shows that if $q_i \neq 0$, then
\[
\frac{I(\bdg)^{\tr}\bas_i}{q_i}=\frac{\epsilon\bas_{u}+\epsilon'\bas_{u'}}{q_i}
\]
is a vector with integer coefficients (in $\Z^{\verts{\bdg}}$). In particular, for $j\neq i$ with $q_j \neq 0$ we get
\[
\frac{q_{ij}}{q_iq_j}=\frac{[I(\bdg)^{\tr}\bas_i]^{\tr}}{q_i}\frac{[I(\bdg)^{\tr}\bas_j]}{q_j} \in \Z,
\]
which shows the semi-full regularity of $q$. The final assertion follows from the equality $\Rnk(I(\bdg))=\Rnk(I(\bdg)I(\bdg)^\tr)$ which is a consequence of the well known nullity-rank theorem. 
\end{proof}

The bigraph $\Delta_{q_{\bdg}}$ associated to the incidence form of a bidirected graph $\bdg$, denoted by $\Delta(\bdg)$ in the text, will be referred to as \textbf{incidence bigraph} of $\bdg$. Note that  $q_\bdg=q_{\Delta(\bdg)}$.

\begin{remark}\label{R:con} (a) By definition, the Gram matrix of $q=q_\bdg$ has the form $G_q=I(\bdg)I(\bdg)^{\tr}$, cf.~Lemma \ref{L:ful}. It means that the incidence bigraph $\Delta=\Delta(\bdg)$ of $\bdg$ can be viewed as (a variant of) the  {\it line signed graph} of $\bdg$ (cf.~\cite{tZ08, CGSS76} and subsection~\ref{subsec:W}) and it can be described combinatorially as follows. Roughly speaking, let $\Delta'=(\verts{\Delta'}, \edgs{\Delta'}, \sigma')$ be a signed graph defined by $\verts{\Delta'}:=\edgs{\bdg}$ and $\{i,j\}\in\edgs{\Delta'}$ for all $i,j$ incident in $\bdg$. Moreover, we set $\sigma'(\{i,j\}):=-1$ if $i$ and $j$ are \lq\lq composable'' in $\bdg$ and $\sigma'(\{i,j\}):=+1$ otherwise. Then $\Delta$ is obtained from $\Delta'$ by removing all pairs of parallel edges of opposite signs.
More precisely, Table~\ref{TA:des} contains a local description of this process $\bdg\mapsto \Delta(\bdg)$.

(b) From (a) it follows that non-incident arrows in $\bdg$ correspond to non-adjacent vertices in $\Delta(\bdg)$. In particular, if $\Delta(\bdg)$ is connected (equivalently, $q_\bdg$ is a connected quadratic form), then so is $\bdg$. Note that the opposite implication does not hold in general, cf.~Table~\ref{TA:des} and Lemma \ref{L:irr}.
\end{remark}

The following example shows that a bidirected graph $\bdg$ is not unique for its incidence form $q=q_\bdg$. Even the number of vertices is not uniquely determined (in contrast to the number of arrows).
\begin{example}\label{ex:twobdg}
Take the following two bidirected graphs:
\[\xymatrixrowsep{0pt}
\xymatrixcolsep{35pt}
\bdg=\xymatrix@!0{\bulito_{u_1} \ar@{<->}@<-.5ex>[r]_-{3} \ar@<.5ex>[r]^-{1} & \bulito_{u_2} \ar[r]^-{2} & \bulito_{u_3}} \quad \text{and} \quad \bdg'=\xymatrix@!0{\bulito_{u_1} \ar[r]^-{1} & \bulito_{u_2} \ar[r]^-{2} & \bulito_{u_3} \ar[r]^-{3} & \bulito_{u_4} }
\]
The corresponding incidence matrices are given by
\[
I(\bdg)=\begin{bmatrix*}[r]1&-1&0\\0&1&-1\\-1&-1&0 \end{bmatrix*}, \quad \text{and} \quad I(\bdg')=\begin{bmatrix*}[r]1&-1&0&0\\0&1&-1&0\\0&0&1&-1 \end{bmatrix*}.
\]
The incidence forms $q=q_\bdg$ and $q'=q_{\bdg'}$ as well as their Gram matrices $G_q=I(\bdg)I(\bdg)^{\tr}$ and $G_{q'}=I(\bdg')I(\bdg')^{\tr}$ and bigraphs $\Delta=\Delta_q=\Delta(\bdg)$ and $\Delta'=\Delta_{q'}=\Delta(\bdg)$ coincide:
\[
\begin{array}{c}
q(x_1, x_2, x_3)=x_1^2+x_2^2+x_3^2-x_1x_2-x_2x_3=q'(x_1, x_2, x_3),\medskip\\  G_q=G_{q'}=\begin{bmatrix*}[r]2&-1&0\\-1&2&-1\\0&-1&2 \end{bmatrix*} \qquad
\xymatrixrowsep{0pt}
\xymatrixcolsep{30pt}
\Delta=\Delta'=\xymatrix@!0{1 \ar@{-}[r] & 2 \ar@{-}[r] & 3 }.
\end{array}
\]
\noindent (see also Example \ref{ex:twobdgcrks}).
\end{example}

\begin{table}[h]
\begin{center}
 \begin{tabular}{l c c}
 \hspace{25mm} $\bdg$ && $\Delta(\bdg)$ \\
  \hline
 $\xymatrix@C=1pc{\mathmiddlescript{\cdots} \ar@{->}[r]^(.35){i} & \bulito \ar@{|-}[r]^(.65){j} & \mathmiddlescript{\cdots}}$ && $\xymatrix{i \ar@{-}[r] & j}$ \\[10pt]
 $\xymatrix@C=1pc{\mathmiddlescript{\cdots} \ar@{->}[r]^(.35){i} & \bulito \ar@{<-}[r]^(.65){j} & \mathmiddlescript{\cdots}} \quad \xymatrix@C=1pc{\mathmiddlescript{\cdots} \ar@{-|}[r]^(.35){i} & \bulito \ar@{|-}[r]^(.65){j} & \mathmiddlescript{\cdots}}$ && $\xymatrix{i \ar@{.}[r] & j}$ \\[10pt]
 $\xymatrix{\bulito \ar@<.5ex>@{->}[r]^-{i} \ar@<-.5ex>@{<-}[r]_-{j} & \bulito } \quad \xymatrix{\bulito \ar@<.5ex>@{<->}[r]^-{i} \ar@<-.5ex>@{|-|}[r]_-{j} & \bulito }$ && $\xymatrix{i \ar@<-.3ex>@{-}[r] \ar@<.3ex>@{-}[r] & j}$ \\[10pt]
 $\xymatrix{\bulito \ar@<.5ex>@{->}[r]^-{i} \ar@<-.5ex>@{->}[r]_-{j} & \bulito } \quad \xymatrix{\bulito \ar@<.5ex>@{<->}[r]^-{i} \ar@<-.5ex>@{<->}[r]_-{j} & \bulito } \quad \xymatrix{\bulito \ar@<.5ex>@{|-|}[r]^-{i} \ar@<-.5ex>@{|-|}[r]_-{j} & \bulito }$ && $\xymatrix{i \ar@<-.3ex>@{.}[r] \ar@<.3ex>@{.}[r] & j}$ \\[10pt]
 $\xymatrix{\bulito \ar@<.5ex>@{->}[r]^-{i} \ar@<-.5ex>@{<->}[r]_-{j} & \bulito } \quad \xymatrix{\bulito \ar@<.5ex>@{->}[r]^-{i} \ar@<-.5ex>@{|-|}[r]_-{j} & \bulito } \quad \text{or $i$, $j$ non-incident}$ && $\xymatrix{i \ar@{}[r]  & j}$ \\[10pt]
 $\xymatrix@C=1pc{\bulito \ar@{<->}@(lu,ld)_-{i} \ar@{|-}[r]^(.65){j} & \mathmiddlescript{\cdots} } \quad \xymatrix@C=1pc{\bulito \ar@{|-|}@(lu,ld)_-{i} \ar@{<-}[r]^(.65){j} & \mathmiddlescript{\cdots} }$ && $\xymatrix{i \ar@{.}@(lu,ld) \ar@<-.3ex>@{-}[r] \ar@<.3ex>@{-}[r] & j}$ \\[10pt]
 $\xymatrix@C=1pc{\bulito \ar@{<->}@(lu,ld)_-{i} \ar@{<-}[r]^(.65){j} & \mathmiddlescript{\cdots} } \quad \xymatrix@C=1pc{\bulito \ar@{|-|}@(lu,ld)_-{i} \ar@{|-}[r]^(.65){j} & \mathmiddlescript{\cdots} }$ && $\xymatrix{i \ar@{.}@(lu,ld) \ar@<-.3ex>@{.}[r] \ar@<.3ex>@{.}[r] & j}$ \\[10pt]
 $\xymatrix@C=1pc{\bulito \ar@(lu,ld)_-{i} \ar@{<-}[r]^(.65){j} & \mathmiddlescript{\cdots} } \quad \xymatrix@C=1pc{\bulito \ar@(lu,ld)_-{i} \ar@{|-}[r]^(.65){j} & \mathmiddlescript{\cdots} }$ && $\xymatrix{i \ar@{-}@(lu,ld) \ar@{}[r] & j}$ \\[10pt]
 $\xymatrix@C=1pc{\bulito \ar@{<->}@(lu,ld)_-{i} \ar@{|-|}@(ru,rd)^-{j} }$ && $\xymatrix{i \ar@{.}@(lu,ld) \ar@/^3pt/@<.4ex>@{-}[r] \ar@<-.3ex>@{-}[r] \ar@<.3ex>@{-}[r] \ar@/_3pt/@<-.4ex>@{-}[r] & j \ar@{.}@(ru,rd)}$ \\[10pt]
 $\xymatrix@C=1pc{\bulito \ar@{<->}@(lu,ld)_-{i} \ar@{<->}@(ru,rd)^-{j} } \quad \xymatrix@C=1pc{\bulito \ar@{|-|}@(lu,ld)_-{i} \ar@{|-|}@(ru,rd)^-{j} }$ && $\xymatrix{i \ar@{.}@(lu,ld) \ar@/^3pt/@<.4ex>@{.}[r] \ar@<-.3ex>@{.}[r] \ar@<.3ex>@{.}[r] \ar@/_3pt/@<-.4ex>@{.}[r] & j \ar@{.}@(ru,rd)}$ \\[10pt]
 $\xymatrix@C=1pc{\bulito \ar@(lu,ld)_-{i} \ar@{<->}@(ru,rd)^-{j} } \quad \xymatrix@C=1pc{\bulito \ar@(lu,ld)_-{i} \ar@{|-|}@(ru,rd)^-{j} }$ && $\xymatrix{i \ar@{-}@(lu,ld) & j \ar@{.}@(ru,rd)}$ \\[10pt]
 $\xymatrix@C=1pc{\bulito \ar@(lu,ld)_-{i} \ar@(ru,rd)^-{j} }$ && $\xymatrix{i \ar@{-}@(lu,ld) & j \ar@{-}@(ru,rd)}$
 \end{tabular}
\end{center}
\caption{Bigraphs associated to pairs of arrows in a bidirected graph.}
\label{TA:des}
\end{table}

Denote by $L$ the bigraph containing one vertex and one solid loop. Note that $L=\Delta_\zeta$  is the bigraph associated to the zero  form $\zeta:\Z\to\Z$.

\begin{lemma}\label{L:irr}
Let $\bdg$ be a connected bidirected graph.
\begin{enumerate}[label={\textnormal{(\alph*)}},topsep=3px,parsep=0px]
 \item If $\bdg$ contains a directed loop $i$, then $\Delta(\bdg)=L \sqcup \Delta(\bdg')$, where $\bdg'$ is the bidirected graph obtained from $\bdg$ by deleting the loop $i$.
 \item  $\bdg$ has at least two vertices if and only if  $q_{\bdg}$ is an irreducible quadratic form.
 \item If $\bdg$ has at least three vertices and no directed loop, then $\Delta({\bdg})$ is connected.
\item If $|\verts{\bdg}|=2$ and $\bdg$ has no directed loop, then $\Delta(\bdg)$ is connected if and only if either $\bdg$ has a bidirected loop, or all arrows in $\bdg$ have the same sign.
\end{enumerate}
\end{lemma}
\begin{proof}
Since the form $q=q_\bdg$ is non-negative, part $(a)$ is a consequence of Remark \ref{R:irr}(b) and the definition of $\Delta_q$ (cf.~\cite[Lemma 3.2]{jaJ2018}). For $(b)$, since $\bdg$ is connected so if $\verts{\bdg}\geq 2$, then there must be an arrow $i$ which is not a loop. Then $q_i=1$ by Lemma~\ref{L:ful}, and therefore $q_{\bdg}$ is irreducible. For the converse, if $\verts{\bdg}=1$, then $\bdg$ consist of loops only and in this case $q$ is not irreducible, see Remark \ref{R:loo} where this case is discussed in detail.

Now, consider two arrows $i$ and $j$ in $\bdg$ which are incident but non-parallel. Then $\edges(i)=\{(v,\epsilon),(v_0,\epsilon_0)\}$ and $\edges(j)=\{(v,\epsilon'),(v_1,\epsilon_1)\}$ for some vertices $v,v_0,v_1$ with $v_0 \neq v_1$ and some signs $\epsilon,\epsilon_0,\epsilon_1 \in \{\pm 1\}$. In this case we have
\[
q_{ij}= [I(\bdg)^{\tr}\bas_i]^{\tr}[I(\bdg)^{\tr}\bas_j]=[\epsilon\bas_{v}^{\tr}+\epsilon_0\bas_{v_0}^{\tr}][\epsilon'\bas_{v}+\epsilon_1\bas_{v_1}]=\epsilon\epsilon'+(\epsilon\epsilon_1\bas_v^{\tr}\bas_{v_1}+\epsilon_0\epsilon'\bas_{v_0}^{\tr}\bas_v).
\]
Note that if neither $i$ nor $j$ is a loop, then $q_{ij}=\epsilon\epsilon' \neq 0$, and that if $i$ is a bidirected loop (that is, $v_0=v$ and $\epsilon_0=\epsilon$) and $j$ is not a loop, then $q_{ij}=2\epsilon\epsilon' \neq 0$.

To show $(c)$, using the connectedness of $\bdg$, for any pair of arrows $i$ and $j$ there is a sequence of arrows $i_0,i_1,\ldots,i_t$ with $i_0=i$, $i_t=j$ and such that $i_{r-1}$ and $i_r$ are incident arrows for $r=1,\ldots,t$. Moreover, since $\bdg$ has at least three vertices, such sequence may be chosen such that $t>1$, $i_{r-1}$ and $i_r$ are non-parallel arrows (for $r=1,\ldots,t$), and $i_{r}$ is not a loop for each $r\in \{1,\ldots,t-1\}$. By the above, we have $q_{i_{r-1}i_r}\neq 0$ for $r=1,\ldots,t$, which shows that $q_{\bdg}$ is a connected quadratic form (that is, $\Delta(\bdg)$ is connected).

Finally, assume that $\bdg$ is as in (d). If $\bdg$ does not have a bidirected loop and it contains arrows of distinct signs, then $\edgs{\bdg}=E^+\sqcup E^-$ where $E^+\neq \emptyset$ (resp.~$E^-\neq \emptyset$) are all non-loop directed (resp.~bidirected) arrows. So $i$ and $j$ are non-adjacent in $\Delta:=\Delta(\bdg)$ for each $i\in E^+$ and $j\in E^-$ (see Table \ref{TA:des}), which means that $\Delta$ is disconnected. To show the converse assume that $\bdg$ has a bidirected loop $i_0$ at a vertex $u_0$. Then by Table \ref{TA:des}, $i_0$ is adjacent in $\Delta$ with all $j$ incident in $\bdg$ with $u_0$. This means that $\Delta$ is connected since $|\verts{\bdg}|=2$. In the remaining case when $\bdg$ has no bidirected loop and all arrows in $\bdg$ have the same sign, again by Table \ref{TA:des} we easily see that in $\Delta$ each pair of vertices is joined by two parallel edges, hence $\Delta$ is connected.
\end{proof}

\begin{remark} \label{R:loo}
Assume that $\bdg$ is a bidirected graph with only one vertex, $\verts{\bdg}=\{u\}$. Then all arrows in $\edgs{\bdg}=\{i_1,\ldots,i_n\}$ are loops (recall that by our global assumption $n\geq 1$). We can assume that $i_1,\ldots, i_p$ are directed loops and  $i_{p+1},\ldots,i_{p+s}$ (resp.~$i_{p+s+1}, \ldots, i_{p+s+t}$) are bidirected loops with two tails (resp.~two heads) for $p,s,t\geq 0$ with $p+s+t=n$. We denote then $\bdg$ by $\onevertexbdg^{p,s,t}$. Then by Lemma \ref{L:ful} the diagonal coefficients of the incidence form $q=q_\bdg$ satisfy $q_{i_1}=\ldots=q_{i_p}=0$ and $q_{i_{p+1}}=\ldots=q_{i_n}=2$.
In particular,  $q_{\bdg}$ is non-irreducible (cf.~Remark \ref{R:irr}). More precisely, $q=2\hat{q}$, where $\hat{q}$ is irreducible integral form such that $\Delta_{\hat{q}}=L^p\sqcup\Delta'$, where $\Delta'$ is a loop-free full bigraph with $\verts{\Delta'}=\{i_{p+1},\ldots, i_n\}$ and precisely two edges joining every pair of vertices. Edges $\{i_k, i_{l}\}$ for $p<k\leq p+s$ and $p+s<l\leq n$ are solid, the remaining ones are dotted. Take for instance $\bdg=\onevertexbdg^{2,2,2}$. Then $q_\bdg=2\hat{q}$ with $\Delta_{\hat{q}}$ of the shape:
$$
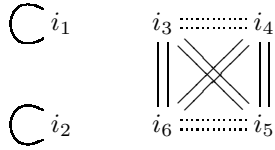
{\xymatrix{
i_1\ar@{-}@(lu,ld) & i_3\ar@<-.4ex>@{.}[r] \ar@<.4ex>@{.}[r] \ar@<-.4ex>@{-}[rd]\ar@<.4ex>@{-}[rd] \ar@<-.4ex>@{-}[d]\ar@<.4ex>@{-}[d]& i_4\ar@<-.4ex>@{-}[ld]\ar@<.4ex>@{-}[ld]\ar@<-.4ex>@{-}[d]\ar@<.4ex>@{-}[d]\\
i_2\ar@{-}@(lu,ld) & i_6\ar@<-.4ex>@{.}[r] \ar@<.4ex>@{.}[r] & i_5}}$$
\end{remark}

\begin{lemma}\label{L:sub}
Let $\bdg$ be a bidirected graph with $n=|\edgs{\bdg}|\geq 1$.
Then
   $\Delta'$ is a full subbigraph of $\Delta(\bdg)$ if and only if there is a bidirected subgraph $\bdg'$ of $\bdg$ with $\verts{\bdg'}=\verts{\bdg}$ such that $\Delta'=\Delta(\bdg')$.
\end{lemma}
\begin{proof} 
Fix  a full subbigraph $\Delta'$ of $\Delta(\bdg)$. Then $\Delta'=\Delta_{q^X}$ for a restriction $q^X=q\circ \iota_X$ of $q=q_\bdg$ to some subset $X\subseteq\un$, see \eqref{eq:restriction}. Thus $G_{q^X}=\iota^\tr G_q\iota = \iota^\tr I(\bdg)I(\bdg)^\tr\iota=I'I'^\tr$ where $\iota=\iota_X$ and $I':=\iota^\tr I(\bdg)$. Observe that $I'=I(\bdg')$ is the incidence matrix of the bidirected subgraph $\bdg'$ obtained from $\bdg$ by removing the arrows with indices in $\un\setminus X$. In particular, $\verts{\bdg'}=\verts{\bdg}$ and $\Delta'=\Delta_{q^X}=\Delta_{q_{\bdg'}}=\Delta(\bdg')$. The opposite implication holds easily by similar arguments.
\end{proof}

\section{Transformations of bidirected graphs and quadratic forms}\label{2}

In this section we review the main transformations used in the study of (Cox-regular or unitary) integral quadratic forms, and give their corresponding graphical interpretation within the class of incidence forms (Definition~\ref{D:tra} and Lemma~\ref{L:tra}). We also establish some first numerical conditions on incidence forms that are visible in the corresponding bidirected graph.

Let $q:\ZZ^n\to\ZZ$ be an integral quadratic form as in \eqref{EQ:iqf}. For $i \neq j$ in $\underline{n}$ with $q_i \neq 0$, consider the linear transformation $T_{ij}=T^q_{ij}:\R^n \to \R^n$ given by
\begin{equation}\label{EQ:gab}
T^q_{ij}(\bas_k) = \left\{
\begin{array}{l l}
\bas_k, & \text{if $k \neq j$}, \\
\bas_j-\frac{q_{ij}}{q_i}\bas_i, & \text{if $k=j$},
\end{array} \right.
\end{equation}
called (elementary) $\textbf{Gabrielov transformation}$ of $q$ at $(i,j)$, see \cite{cmR, Ze94}. In case $q_{ij}>0$ (resp. $q_{ij}<0$), the transformation $T_{ij}^q$ is also called an \textbf{inflation} (resp. a \textbf{deflation}) of $q$.  Observe that if $q$ is a semi-Cox-regular quadratic form, then $T^q_{ij}$ is defined over the integers for any $i \neq j$ with $q_i \neq 0$, and the (well-defined) composition $q':=q \circ T^q_{ij}$ is also a semi-Cox-regular integral quadratic form, cf.~\cite[Lemma~4.9]{MZ22}. To compute $q'$ explicitly, take $y=T^q_{ij}x$ and note that $y_k=x_k$ for $k \neq i$ and $y_i=x_i-\frac{q_{ij}}{q_i}x_j$. Then
\begin{eqnarray*}
q(T^q_{ij}x) & = & q(y)=\sum_{k=1}^nq_ky_k^2+\sum_{k<\ell}q_{k\ell}y_ky_\ell \\
& = & \left[ \sum_{k \neq i}q_kx_k^2+\sum_{\substack{k < \ell \\ \text{$k \neq i$ and $i \neq \ell$}}}q_{k\ell}x_kx_\ell \right]+q_iy_i^2+\sum_{\substack{k<i \\ k \neq j}}q_{ki}x_ky_i+\sum_{\substack{i<\ell \\ \ell \neq j}}q_{i\ell}x_\ell y_i+q_{ij}y_ix_j \\
& = & q(x)-2q_{ij}x_ix_j +\frac{q_{ij}^2}{q_i}x_j^2 -\sum_{\substack{k<i \\ k \neq j}}\frac{q_{ki}q_{ij}}{q_i}x_kx_j-\sum_{\substack{i<\ell \\ \ell \neq j}}\frac{q_{i\ell}q_{ij}}{q_i}x_\ell x_j-\frac{q_{ij}^2}{q_i}x_j^2 \\
& = & q(x)-2q_{ij}x_ix_j -\sum_{\substack{k<i \\ k \neq j}}\frac{q_{ki}q_{ij}}{q_i}x_kx_j-\sum_{\substack{i<\ell \\ \ell \neq j}}\frac{q_{i\ell}q_{ij}}{q_i}x_\ell x_j.
\end{eqnarray*}
This shows that the coefficients of $q'$ are given as follows: for $k \leq \ell$,
\begin{equation}\label{EQ:coe}
q'_{k\ell} = \left\{
\begin{array}{l l}
q_{k}, & \text{if $k=\ell$}, \\[1ex]
q_{k\ell}, & \text{if $k<\ell$ and $k,\ell \neq j$}, \\[1ex]
q_{kj}-\frac{q_{ki}q_{ij}}{q_{i}}, & \text{if $k<\ell$ and $k \neq i, \ell=j$}, \\[1ex]
q_{j\ell}-\frac{q_{i\ell}q_{ij}}{q_{i}}, & \text{if $k<\ell$ and $k=j$}, \\[1ex]
-q_{ij}, & \text{if ($k=i$ and $\ell=j$) or ($k=j$ and $\ell=i$)}.
\end{array} \right.
\end{equation}
Here we used our conventions $q_{\ell k}=q_{k\ell}$ for $k<\ell$, and $q_k=q_{kk}$. Note that the definition of $T^q_{ij}$ includes both cases $i<j$ and $i>j$, given $q_i \neq 0$. We extend this definition when $q_i=0$ by setting $T_{ij}^q=\Id_n$ in this case.
\begin{lemma}\label{L:inv}
Let $q:\Z^n \to \Z$ be a Cox-regular quadratic form, $T=T^q_{ij}$ a  Gabrielov transformation of $q$, and take $q':=q \circ T$. Then the following hold.
\begin{enumerate}[label={\textnormal{(\roman*)}},topsep=3px,parsep=0px]
 \item The inverse $T^{-1}=T^{q'}_{ij}$ of $T$ is a Gabrielov transformation of $q'$.
 \item If $q$ is connected $($resp. irreducible$)$, then so is $q'$.
 \item For every $k \in \underline{n}$ we have $q'_k=q_k$.
 \item If $q$ is fully regular, then so is $q'$.
\end{enumerate}
\end{lemma}
\begin{proof}
For the properties (i)-(ii) we refer to \cite[Lemma~4.9]{MZ22}, cf.~\cite{Ze94}. Claim (iii) follows directly from equation~(\ref{EQ:coe}). To show (iv) we apply the equation~(\ref{EQ:coe}) to observe that for $k \neq \ell$ in $\underline{n}$ we have
\begin{equation*}
\frac{q'_{k\ell}}{q'_kq'_{\ell}} = \left\{
\begin{array}{l l}
\frac{q_{k\ell}}{q_kq_{\ell}}, & \text{if $k \neq j$ and $\ell \neq j$}, \\
\frac{q_{kj}}{q_kq_j}-\left[ \frac{q_{ki}}{q_kq_i} \right]\frac{q_{ij}}{q_j}, & \text{if $k \neq i$ and $\ell=j$}, \\
\frac{-q_{ij}}{q_iq_j}, & \text{if $k=i$ and $\ell=j$}.
\end{array} \right.
\end{equation*}
Note that this equation includes all cases since $k \neq \ell$ are arbitrary. Then $q'_{k\ell}/(q'_kq'_{\ell}) \in \Z$ since $q$ is fully regular, which shows the full regularity of $q'$.
\end{proof}

A diagonal matrix  $T=\diag(s_{1},\ldots,s_{m})$ with $s_{i}\in\{+1,-1\}$  is called a \textbf{sign inversion matrix}. If $s_i=-1$ for some $i$ and $s_j=1$ for all $j\neq i$, then we call $T_i:=T$ \textbf{sign inversion at $i$}. Observe that if $q:\ZZ^n\to\ZZ$ is a Cox-regular (resp.~irreducible, connected) form, then so is $q'=q\circ T_i$.  Moreover, $q'_j=q_j$ for all $j\in\un$ and $q'_{ik}=-q_{ik}$ for all $k\neq i$.

Recall that an integer matrix $O\in\M_m(\Z)$ is \textbf{orthogonal} if $OO^\tr=\Id_m$.
The following known fact  provides a useful characterization of such matrices.

\begin{lemma}[{{\cite[Lemma 2.3]{dS11}}}]\label{L:ortSimson}
 An integer matrix $O\in\M_m(\Z)$ is orthogonal if and only if $O=TP$ for a $($unique$)$
  sign inversion matrix $T$
  and a $($unique$)$ permutation matrix $P=P^\pi$ associated to a permutation $\pi$ of the set $\{1,\ldots,m\}$.
\end{lemma}


\begin{lemma}\label{L:equivs}
Let $\bdg$ and $\bdg'$ be two bidirected graphs with $m$ vertices and $n$ arrows. Assume that
 there is  $T\in\GlnZ$ and an orthogonal matrix $O\in\M_m(\Z)$ such that $I(\bdg)'=T^{\tr}I(\bdg)O$. Then:
\begin{enumerate}[label={\textnormal{(\alph*)}},topsep=6px,parsep=0px]
\item  $q_\bdg\weak^T q_{\bdg'}$,
\item  $q_\bdg\cong q_{\bdg'}$ provided  $T$ is a permutation matrix,
\item  $q_\bdg= q_{\bdg'}$ provided $T=\Id_n$ is the identity matrix.
\end{enumerate}
\end{lemma}
\begin{proof}
Since $O$ is orthogonal we have $G_{q_{\bdg'}}=I(\bdg')I(\bdg')^{\tr}=(T^{\tr}I(\bdg)O)(T^{\tr}I(\bdg)O)^\tr=
T^{\tr}I(\bdg) I(\bdg)^\tr T=T^{\tr}G_{q_\bdg} T$, thus (a) holds. Claims (b) and (c) follow clearly by (a).
\end{proof}

Fix $n\geq 1$ and two semi-Cox-regular forms $q,q':\ZZ^n\to\ZZ$. Following \cite{Ze94, PR19} we say that $q$ and $q'$ are \textbf{Gabrielov equivalent}, or \textbf{G-equivalent}, if $q'=q\circ T$ (that is, $q \weak^T q'$) for $T\in\GlnZ$ being a composition of
 sign inversions, permutation matrices and (well-defined)  Gabrielov transformations. We write then $q \Gweak q'$ or $q \Gweak^T q'$ and we call $T$ a \textbf{G-transformation} (of $q$).

\begin{corollary}\label{cor:Gweak}
G-equivalence is an equivalence relation on the set of semi-Cox-regular integral forms. Moreover, assuming that $q\Gweak q'$, we have that if $q$ is connected $($resp.~irreducible, unitary, fully regular, Cox-regular$)$, then so is $q'$, and there exists a permutation $\pi$ such that $q'_i=q_{\pi(i)}$ for all $i$.
\end{corollary}
\begin{proof}
Follows from Lemma \ref{L:inv} and obvious properties of sign inversions and permutation matrices.
\end{proof}

\begin{remark}\label{rem:GabrielovLie}
G-equivalence $\Gweak$  has its origins and relevant interpretation in Lie theory and singularity theory. For instance, recall that the authors of \cite{Slodovy, BKL06} (resp.~of \cite{PR19}) showed that in case of unit forms (resp.~positive Cox-regular forms) G-equivalence corresponds to isomorphism of the  associated Lie algebras (by means of Serre-type relations induced by quasi-Cartan and intersection matrices related with quadratic forms). We refer also to \cite{BGZ06,  Slodovy, Gabrielov, dS22, PR21} for corresponding results from singularity theory and the theory of cluster algebras.
\end{remark}

We want to express G-transformations in terms of the corresponding bidirected graphs. We consider the following graphical transformations of bidirected graphs.

\begin{definition}\label{D:tra}
Let $\bdg$ be a bidirected graph with $m=|\verts{\bdg}|$ and $n=|\edgs{\bdg}|$.
\begin{enumerate}[label={\textnormal{(\alph*)}},topsep=3px,parsep=0px]
 \item Assume that $i$ and $j$ are incident arrows in $\bdg$, with $\edges_\bdg(i)=\{(v_0,\epsilon_0),(v_1,\epsilon_1)\}$, $\edges_\bdg(j)=\{(w_0,\eta_0),(w_1,\eta_1)\}$ and $v_0=w_0$. A bidirected graph $\bdg'=\bdg\mathcal{T}_{ij}$ is obtained from $\bdg$ by
     taking $\verts{\bdg'}:=\verts{\bdg},$ $\edgs{\bdg'}:=\edgs{\bdg}$ and $\edges_{\bdg'}(i'):=\edges_{\bdg}(i')$ for $i'\neq j$, and replacing the signed endpoints $\edges_{\bdg}(j)$ of $j$  as follows:
 \begin{enumerate}[label={\textnormal{(\alph*)}},topsep=3px,parsep=0px]
  \item[(a1)] if $w_1 \notin \{v_0,v_1\}$, then take $\edges_{\bdg'}(j):=\{(v_1,\sigma(i)\eta_0 ),(w_1,\eta_1)\}$;
  \item[(a2)] if $w_1=v_1 \neq v_0$, then take $\edges_{\bdg'}(j):=\{(v_0,\sigma(i)\eta_1 ),(v_1, \sigma(i)\eta_0)\}$;
  \item[(a3)] if $w_1=w_0$, then take $\edges_{\bdg'}(j):=\{(v_1,\sigma(i)\eta_0 ),(v_1, \sigma(i)\eta_1)\}$.
 \end{enumerate}

\item Let $i$ be an arrow in $\bdg$ with $\edges_\bdg(i)=\{ (v_0,\epsilon_0),(v_1,\epsilon_1)\}$. A bidirected graph $\bdg\mathcal{T}_i$ is obtained from $\bdg$ by replacing the signed endpoints $\edges_{\bdg}(i)$ by setting $\edges_{\bdg\mathcal{T}_i}(i):=\{(v_0,-\epsilon_0),(v_1,-\epsilon_1)\}$.

\item Let $\pi$ be a permutation of the set of arrows $\edgs{\bdg}$ of $\bdg$. A bidirected graph $\bdg\mathcal{P}^{\pi}$ is obtained from $\bdg$ by reordering its arrows via the permutation $\pi$. To be precise,
$\verts{\bdg \mathcal{P^{\pi}}}:=\verts{\bdg},$  $\edgs{\bdg \mathcal{P^{\pi}}}:=\edgs{\bdg}$ and \quad $\edges_{\bdg \mathcal{P^{\pi}}}:=\edges_{\bdg}\circ \pi$.
Observe that $(\bdg P^{\pi})\mathcal{P}^{\pi'}=\bdg \mathcal{P}^{\pi \pi'}$ for any permutation $\pi'$ of $\edgs{\bdg}$.

 \item Given $\verts{\bdg}=\{u_1,\ldots,u_m\}$, we define a right action of the group of integer orthogonal $m \times m$ matrices on the set $\verts{\bdg} \times \{\pm 1\}$ as follows. For such a matrix $O$ we take $(u_t,\epsilon) \cdot O :=(u_r,\epsilon \eta)$ if $O^{\tr}\bas_{u_t}=\eta \bas_{u_r}$ (cf. Lemma~\ref{L:ortSimson}; that this is indeed a right action can be verified directly). A bidirected graph $\bdg^{O}$, called \textbf{switching} of $\bdg$ with respect to the orthogonal matrix $O$, is obtained as $\bdg^{O}=(\verts{\bdg},\edgs{\bdg},\edges^{O})$ where $\edges^{O}(i):=\{(u,\epsilon)\cdot O,(u',\epsilon')\cdot O\}$ given that $\edges(i)=\{(u,\epsilon),(u',\epsilon')\}$.

\end{enumerate}
\end{definition}

Let us illustrate the constructions of Definition~\ref{D:tra} with an example.
\begin{example} \label{exa:tra}
Consider the following successive transformations on a bidirected graph $\bdg^0$ with $4$ vertices and $5$ arrows (the titles of the constructions $\bdg^i \mapsto \bdg^{i+1}$ refer to next lemma):
\[
\xy 0;/r.20pc/:
(-90, 25)="Fr1" *{};
( 90, 25)="Fr2" *{};
( 90,-5)="Fr3" *{};
(-90,-5)="Fr4" *{};
( -42, 5)="T1" *{\xymatrix{\bulito_{u_1} \ar@{|-|}[r]^-{3} \ar@<1ex>@/^5pt/@{|-|}[d]^-{5} \ar@{|-|}[d]^-{2} \ar@<-1ex>@{|-|}[d]_-{1} & \bulito_{u_3} \\ \bulito_{u_2} \ar@{|-|}[r]_-{4} & \bulito_{u_4}}};
( -27, 5)="T2" *{\xymatrix{\bulito_{u_1} \ar@{|-|}[r]^-{3} \ar@{|-|}[d]^-{2} \ar@<-1ex>@{|-|}[d]_-{1} & \bulito_{u_3} \\ \bulito_{u_2} \ar[ru]_-{5} \ar@{|-|}[r]_-{4} & \bulito_{u_4}}};
( -12, 5)="T3" *{\xymatrix{\bulito_{u_1} \ar@{|-|}[r]^-{3} \ar@{|-|}[d]^-{2} \ar@<-1ex>@{|-|}[d]_-{1} & \bulito_{u_3} \ar@{<->}[d]^-{5} \\ \bulito_{u_2} \ar@{|-|}[r]_-{4} & \bulito_{u_4}}};
(  3,  5)="T4" *{\xymatrix{\bulito_{u_1} \ar@{|-|}[r]^-{3} \ar@{<->}[d]^-{2} \ar@<-1ex>@{<->}[d]_-{1} & \bulito_{u_3} \ar@{<->}[d]^-{5} \\ \bulito_{u_2} \ar@{|-|}[r]_-{4} & \bulito_{u_4}}};
(  18, 5)="T5" *{\xymatrix{\bulito_{u_1} \ar@{|-|}[r]^-{1} \ar@{<->}[d]^-{4} \ar@<-1ex>@{<->}[d]_-{5} & \bulito_{u_3} \ar@{<->}[d]^-{2} \\ \bulito_{u_2} \ar@{|-|}[r]_-{3} & \bulito_{u_4}}};
(  33, 5)="T6" *{\xymatrix{\bulito_{u_1} \ar[r]^-{1} \ar@{<-}[d]^-{4} \ar@{<-}@<-1ex>[d]_-{5} & \bulito_{u_2} \ar[d]^-{2} \\ \bulito_{u_4} \ar@{<-}[r]_-{3} & \bulito_{u_3}}};
(-75, 20)="E1" *{\begin{matrix} \text{bidirected graph} \\ \bdg^0 \end{matrix}};
(-45, 20)="E2" *{\begin{matrix} \text{inflation} \\ \bdg^1=\bdg^0 \mathcal{T}_{3,5} \end{matrix}};
(-15, 20)="E3" *{\begin{matrix} \text{inflation} \\ \bdg^2=\bdg^1 \mathcal{T}_{4,5} \end{matrix}};
( 15, 20)="E4" *{\begin{matrix} \text{sign inversion} \\ \bdg^3=\bdg^2 \mathcal{T}_1\mathcal{T}_2 \end{matrix}};
( 45, 20)="E5" *{\begin{matrix} \text{permutation} \\ \bdg^4=\bdg^3\mathcal{P}^\pi \end{matrix}};
( 75, 20)="E6" *{\begin{matrix} \text{switching} \\ \bdg^5=(\bdg^4)^\mathcal{O} \end{matrix}};
\endxy
\]
Here $\pi$ is the permutation of arrows $\left( \begin{smallmatrix} 1&2&3&4&5\\3&5&4&2&1 \end{smallmatrix} \right)$, and $\mathcal{O}$ is the orthogonal matrix $\left[ \begin{smallmatrix}1&0&0&0\\0&0&0&-1\\0&-1&0&0\\ 0&0&1&0 \end{smallmatrix} \right]$ corresponding to the switching of $\bdg^4$ on the vertices $u_2$ and $u_3$ followed by the permutation of vertices $\left( \begin{smallmatrix} u_1&u_2&u_3&u_4\\u_1&u_4&u_2&u_3 \end{smallmatrix} \right)$. Note that $\bdg^5=\bdgA^2_3$ as in Definition~\ref{D:canAD} below.
\end{example}

The proof of the following result relies on straightforward manipulations of incidence matrices. For convenience, we give the details for some general situations.

\begin{lemma}\label{L:tra}
Let $\bdg$ be a bidirected graph with $m=|\verts{\bdg}|, n=|\edgs{\bdg}|\geq 1$, and consider one of the following linear transformations $T\in\GlnZ$.
\begin{enumerate}[label={\textnormal{(\alph*)}},topsep=3px,parsep=0px]
 \item $T=T^{q_{\bdg}}_{ij}$ is the Gabrielov transformation of $q_{\bdg}$ with respect to the indices $(i,j)$.
 \item $T=T_i$ is the sign inversion at an index $i$,
 \item $T=P^{\pi}$ is a permutation matrix. 
\end{enumerate}
Then $T^\tr I(\bdg)=I(\bdg\mathcal{T})$ and $q_{\bdg}\circ T=q_{\bdg\mathcal{T}}$, where $\mathcal{T}$ is one of the transformation $\mathcal{T}=\mathcal{T}_{ij}$, $\mathcal{T}=\mathcal{T}_i$ or $\mathcal{T}=\mathcal{P}^{\pi}$ as given in Definition~\ref{D:tra}, for the cases $(a)$, $(b)$ or $(c)$ respectively. Moreover,
\begin{enumerate}[label={\textnormal{(\alph*)}},topsep=3px,parsep=0px]
 \item[$({\rm d})$] if $\bdg^{O}$ is the switching of $\bdg$ with respect to an orthogonal matrix $O \in \MM_m(\ZZ)$, then $I(\bdg^O)=I(\bdg) O$; in this case $q_{\bdg^O}=q_{\bdg}$.
\end{enumerate}
\end{lemma}
\begin{proof} (a) Let $q=q_\bdg$ and take $i,j\in\edgs{\bdg}$ with $\edges(i)=\{(v_0,\epsilon_0),(v_1,\epsilon_1)\}$, $\edges(j)=\{(w_0,\eta_0),(w_1,\eta_1)\}$ and $v_0=w_0$. Consider the case (a1) of Definition \ref{D:tra} with $v_0\neq v_1$. Then the incidence matrices of $\bdg$ and $\bdg'=\bdg\mathcal{T}_{ij}$ take the forms:
\def\prz{\!\!\!\!}
$$I(\bdg)={\scriptsize\begin{blockarray}{cccccccc}
    \prz& &\prz v_0\prz &  &\prz v_1\prz & &\prz w_1\prz &  \\
    \begin{block}{c[ccccccc]}
      \prz&\, \prz&\prz\vdots\prz&\prz \prz&\prz\vdots\prz&\prz \prz&\prz\vdots\prz&\prz\\
      i \prz&\, \cdots\prz&\prz\epsilon_0\prz&\prz\cdots\prz&\prz\epsilon_1\prz&\prz\cdots\prz&\prz0\prz&\prz\cdots\vspace{-3pt}\\
       \prz &\, \prz&\prz\vdots\prz&\prz \prz&\prz\vdots\prz&\prz \prz&\prz\vdots\prz&\prz\\
        j\prz&\, \cdots\prz&\prz\eta_0\prz&\prz\cdots\prz&\prz0\prz&\prz\cdots\prz&\prz\eta_1\prz&\prz\cdots\vspace{-3pt}\\
        \prz  &\, \prz&\prz\vdots\prz&\prz \prz&\prz\vdots\prz&\prz \prz&\prz\vdots\prz&\prz\\
    \end{block}
    \end{blockarray}},
    \quad
    I(\bdg')={\scriptsize\begin{blockarray}{cccccccc}
    \prz& &\prz v_0\prz &  &\prz v_1\prz & &\prz w_1\prz &  \\
    \begin{block}{c[ccccccc]}
      \prz&\, \prz&\prz\vdots\prz&\prz \prz&\prz\vdots\prz&\prz \prz&\prz\vdots\prz&\prz\\
      i \prz&\, \cdots\prz&\prz\epsilon_0\prz&\prz\cdots\prz&\prz\epsilon_1\prz&\prz\cdots\prz&\prz0\prz&\prz\cdots\vspace{-3pt}\\
        \prz&\, \prz&\prz\vdots\prz&\prz \prz&\prz\vdots\prz&\prz \prz&\prz\vdots\prz&\prz\\
        j\prz&\, \cdots\prz&\prz0\prz&\prz\cdots\prz&\prz-\epsilon_0\epsilon_1\eta_0\prz&\prz\cdots\prz&\prz\eta_1\prz&\prz\cdots\vspace{-3pt}\\
          \prz&\, \prz&\prz\vdots\prz&\prz \prz&\prz\vdots\prz&\prz \prz&\prz\vdots\prz&\prz\\
    \end{block}
    \end{blockarray}},
$$
cf.~\eqref{eq:IofB}. By applying \eqref{eq:qB} we get that $\frac{q_{ij}}{q_i}=\epsilon_0\eta_0$. Now by a straightforward calculation we check that $T^\tr I(\bdg)=I(\bdg')$ for $T=T^q_{ij}$, see \eqref{EQ:gab}.  This means that $q_\bdg\circ T=q_{\bdg'}$ by Lemma \ref{L:equivs}(a). The case when $v_0=v_1$ as well as the cases (a2)-(a3) of Definition \ref{D:tra} follow by analogous calculations, we skip the explicit details. 

For (b) let $T=T_i\in\GlnZ$ be the sign inversion at $i$. Then clearly $TI(\bdg)=T^\tr I(\bdg)=I(\bdg\mathcal{T}_i)$, see Definition \ref{D:tra}(b) and \eqref{eq:IofB}. This means that $q_\bdg\circ T=q_{\bdg\mathcal{T}_i}$ by Lemma \ref{L:equivs}(a). The claim (c) holds by similar arguments since $(P^\pi)^\tr I(\bdg)=I(\bdg\mathcal{P}^{\pi})$.

For (d), by definition we have $I(\bdg)^{\tr}\bas_i=\epsilon \bas_u+\epsilon'\bas_{u'}$ if $\edges(i)=\{(u,\epsilon),(u',\epsilon')\}$. Then, by the definition of $\bdg^O$, we have $I(\bdg^O)^{\tr}\bas_i=\epsilon O^{\tr}\bas_u+\epsilon'O^{\tr}\bas_{u'}=O^{\tr}I(\bdg)^{\tr}\bas_i$. Since this holds for all $i \in \edgs{\bdg}$, then $I(\bdg^O)=I(\bdg)O$. The final claim of (d) follows by Lemma \ref{L:equivs}(c). 
\end{proof}

\begin{corollary}\label{C:tra}
Let $\bdg$ be a bidirected graph, and take $q:=q_{\bdg}$ the associated incidence quadratic form. If $q \Gweak^T q'$, then there is a bidirected graph $\bdg'$ such that $T^{\tr}I(\bdg)=I(\bdg')$. In particular, $q'=q_{\bdg'}$ is also an incidence quadratic form.
\end{corollary} 

Recall that a bidirected graph $\bdg$ is {\bf balanced} if all closed walks in $\bdg$ are positive, that is, each of them  contains an even number of bidirected arrows. Otherwise, that is, when $\bdg$ contains a negative closed walk, we say $\bdg$ is {\bf unbalanced}.
 Observe that a bidirected graph $\bdg$ is a quiver (that is,  $\bdg$ has no bidirected arrow) if and only if $I(\bdg)\mathbbm{1}_m=0$, where $m=|\verts{\bdg}|$ and $\mathbbm{1}_m$ is the (column) vector in $\Z^m$ with all entries equal to~$1$. The main step for the proof of the following useful graphical interpretation of the (right) nullity of the incidence matrix $I(\bdg)$ of a bidirected graph $\bdg$ is to relate the balance property of $\bdg$ with the possibility of switching $\bdg$ into a quiver.

\begin{proposition} \label{P:null}
Let $\bdg$ be a connected bidirected graph. Then the $($right$)$ nullity of the incidence matrix $I(\bdg)$ of $\bdg$ satisfies $\Null(I(\bdg)) \in \{0,1\}$. Moreover,
\begin{enumerate}[label={\textnormal{(\alph*)}},topsep=4px,parsep=0px]
 \item $\Null(I(\bdg))=0$ if and only if $\bdg$ is unbalanced.
 \item $\Null(I(\bdg))=1$ if and only if there is a switching $\bdg^{O}$ of $\bdg$ which is a quiver.
\end{enumerate}
\end{proposition}
\begin{proof} Let $I:=I(\bdg)\in\M_{n,m}(\ZZ)$ for $n:=|\edgs{\bdg}|$, $m:=|\verts{\bdg}|$, and let $\delta:=\Null(I(\bdg))$. By connectedness, if $Ix=0$ for a vector $x=[x_1,\ldots,x_m]^\tr \in \Z^m$, then $|x_1|=\ldots =|x_m|$. This implies that $\delta \leq 1$. Indeed, if $x$ and $x'$ are two vectors in $\Z^m$ with $Ix=Ix'=0$ and $x_1=x_1'$, then $I(x-x')=0$, and therefore $x=x'$. If $x_1 \neq x_1'$, then we may assume that $x'_1 \neq 0$, and take $x''=\lambda x'$ with $\lambda=x_1/x'_1$. Then $Ix''=0$ and $x_1=x''_1$, and by the above $x=x''$, which means that $x$ is a scalar multiple of $x'$. That is, $\delta \leq 1$, and the first statement of the proposition holds.

To prove (a)-(b) we first show the following auxiliary claim:
\[
\text{$\bdg$ is balanced if and only if there is a switching $\bdg^{O}$ of $\bdg$ which is a quiver.} \tag{$*$} 
\]
Indeed, if  $\bdg^O$ is a quiver, then every walk (in particular, every closed walk) in $\bdg^O$ is positive. Thus the same holds for the closed walks of $\bdg$ since it is easy to see that the sign of the corresponding closed walks in $\bdg$ and $\bdg^O$ is unchanged. To show the opposite implication of $(*)$ assume that every closed walk of $\bdg$ is positive. Note that if $\bdg'$ is a connected (full) subgraph of $\bdg$ that is actually a quiver, and $u$ is a vertex in $\verts{\bdg}\setminus\verts{\bdg'}$ connected to some vertex $u'$ in $\bdg'$, then all arrows connecting $u$ to some vertex in $\bdg'$ have the same sign. Indeed, if $i'$ and $i''$ are arrows in $\bdg$ with different sign connecting vertex $u$ to (maybe equal) vertices $u'$ and $u''$ in $\bdg'$, then there is a (maybe trivial) walk $\wlk$ inside $\bdg'$ from $u'$ to $u''$. Since $\bdg'$ is a quiver, $\wlk$ is a positive walk. Then the composition $(u,i',u')w(u'',i'',u)$ is a negative closed walk, which is impossible. If all arrows from $u$ to a vertex in $\verts{\bdg'}$ are directed, then the full subgraph of $\bdg$ determined by the vertices $\verts{\bdg'} \sqcup \{u\}$ is a quiver. Otherwise, we obtain a similar quiver restriction by multiplying the $u$-th column of $I(\bdg)$ by $-1$. In this way, starting with the full subgraph $\bdg'$ of $\bdg$ determined by an arbitrary vertex, by connectedness we may add to $\bdg'$ one by one all the vertices of $\bdg$, and find adequate signs $s_u$ for each vertex $u$, such that if $O=\diag(s_u)$, then $\bdg^O$ has only directed arrows. That is, $\bdg^O$ is a quiver. This closes the proof of $(*)$.

Now by $(*)$ and the proved fact that $\delta\in\{0,1\}$, to conclude the proof of the proposition it remains to show (b). 
Assume that $\delta=\Null(I(\bdg))=1$, and take $O=\diag(x)$ for $x\in\ZZ^m$ such that $I(\bdg)x=0$ and all entries of $x$ are $\pm 1$ (it exists by the first part of the proof). Then $\bdg^O$ is a quiver, since
\[
I(\bdg^O) \mathbbm{1}_m=I(\bdg)O\mathbbm{1}_m=I(\bdg)\diag(x)\mathbbm{1}_m=I(\bdg)x=0.
\]
Conversely, if some switching $\bdg^O$ is a quiver, then $I(\bdg)O\mathbbm{1}_m=0$ so $\delta=1$ since $\delta\leq 1$.
\end{proof}

Given a bidirected graph $\bdg$, recall that we defined the parameter $\cyccond_\bdg\in \{0,1\}$  in the introduction by setting  $\cyccond_\bdg:=1$ if $\bdg$ is balanced and $\cyccond_\bdg:=0$ otherwise.
\begin{corollary}\label{cor:crk}
Let $\bdg$ be a connected bidirected graph with $m=|\verts{\bdg}|\geq 1$ and $n=|\edgs{\bdg}|\geq 1$. Then
\[
\Rnk(q_\bdg)=m-\cyccond_\bdg \quad \text{and} \quad \CRnk(q_\bdg)=n-m+\cyccond_\bdg.
\]
\end{corollary}
\begin{proof}
Recall that $\Rnk(q_\bdg)=m-\Null(I(\bdg))$ and $\CRnk(q_\bdg)=n-m+\Null(I(\bdg))$
 by Lemma \ref{L:ful}. Moreover, by Proposition \ref{P:null} the value of $\Null(I(\bdg))\in\{0,1\}$  coincides with the parameter $\cyccond_\bdg$. In this way we get the desired formulae.
\end{proof}

\begin{example}\label{ex:onevertexbdgcrk}
Take a one-vertex bidirected graph $\bdg=\onevertexbdg^{p,s,t}$ as defined in
Remark \ref{R:loo}, for some $p+s+t=n\geq 1$. Then $\Rnk(q_\bdg)=1-\cyccond_\bdg$ and $\CRnk(q_\bdg)=n-1+\cyccond_\bdg$ by Corollary \ref{cor:crk}. Note that $\cyccond_\bdg=0$ if and only if $p<n$, that is, if $\bdg$ has at least one bidirected loop. If $p=n$, then $\Rnk(q_\bdg)=0$ and $\CRnk(q_\bdg)=n$. Indeed, observe that if all loops in $\bdg$ are directed, then $q_\bdg=\zeta^n$ is a zero form, cf.~Lemma \ref{L:irr}(a).
We refer to Example \ref{ex:twobdgcrks} for another illustration of Corollary \ref{cor:crk}.
\end{example}

\section{Dynkin types $\AA$ and $\DD$}\label{sec:AD}

In this section we describe those unitary integral quadratic forms that are incidence forms of bidirected graphs (Theorem A, see subsection~\ref{subsec:proofA}). We start recalling the shape of (simply-laced) Dynkin graphs (Table \ref{TA:Dynkins}) and the classification of non-negative connected unit forms in terms of Dynkin type and corank (Theorem~\ref{T:wea} and Corollary~\ref{cor:weakvsweakGforunit}). The main step towards the proof of Theorem A is a description of the incidence unit forms whose bigraph is connected and has no dotted edges, achieved in Proposition~\ref{L:graph}.
\begin{table}[h!]
\begin{center}
\begin{tabular}[h!]{l}
		$\AA_n:$\
			{\scriptsize
				\begin{tikzpicture}[auto]
				\node (n1) at (0  , 0  ) {$1$};
				\node (n2) at (0.7, 0  ) {$2$};
				\node (n3) at (1.4, 0  ) {$3$};
				\node (n4) at (2.1, 0  ) {};
				\node (n5) at (2.7, 0  ) {};
				\node (n6) at (3.4, 0  ) {$n$};
				\node (n8) at (0  , 0.4) {};
				\draw [dotted, -] (n4) to  (n5);
				\foreach \x/\y in {1/2, 2/3, 3/4, 5/6}
				\draw [-] (n\x) to  (n\y);
				\end{tikzpicture}
		} 
    \\	$\DD_n:$\
			{\scriptsize
				\begin{tikzpicture}[auto]
				\node (n1) at (0  , 0  ) {$1$};
				\node (n2) at (0.7, 0.7) {$2$};
				\node (n3) at (0.7, 0  ) {$3$};
				\node (n4) at (1.4, 0  ) {$4$};
				\node (n5) at (2.1, 0  ) {};
				\node (n6) at (2.7, 0  ) {};
				\node (n7) at (3.4, 0  ) {$n$};
				\draw [dotted, -] (n5) to  (n6);
				\foreach \x/\y in {1/3, 2/3, 3/4, 4/5, 6/7}
				\draw [-] (n\x) to  (n\y);
				\end{tikzpicture}
		} 
\end{tabular}\quad 
\begin{tabular}{l}
		 $\EE_6:$\
			{\scriptsize
				\begin{tikzpicture}[auto]
				\node (n1) at (0  , 0  ) {$1$};
				\node (n2) at (0.7, 0  ) {$2$};
				\node (n3) at (1.4, 0  ) {$3$};
				\node (n4) at (1.4, 0.7) {$4$};
				\node (n5) at (2.1, 0  ) {$5$};
				\node (n6) at (2.8, 0  ) {$6$};
				\foreach \x/\y in {1/2, 2/3, 3/4, 3/5, 5/6}
				\draw [-] (n\x) to  (n\y);
				\end{tikzpicture}
		} 
        \\
		$\EE_7:$\
			{\scriptsize
				\begin{tikzpicture}[auto]
				\node (n1) at (0  , 0  ) {$1$};
				\node (n2) at (0.7, 0  ) {$2$};
				\node (n3) at (1.4, 0  ) {$3$};
				\node (n4) at (1.4, 0.7) {$4$};
				\node (n5) at (2.1, 0  ) {$5$};
				\node (n6) at (2.8, 0  ) {$6$};
				\node (n7) at (3.5  , 0  ) {$7$};
				\foreach \x/\y in {1/2, 2/3, 3/4, 3/5, 5/6, 6/7}
				\draw [-] (n\x) to  (n\y);
				\end{tikzpicture}
		} 
        \\
		$ \EE_8:$\
			{\scriptsize
				\begin{tikzpicture}[auto]
				\node (n1) at (0  , 0  ) {$1$};
				\node (n2) at (0.7, 0  ) {$2$};
				\node (n3) at (1.4, 0  ) {$3$};
				\node (n4) at (1.4, 0.7) {$4$};
				\node (n5) at (2.1, 0  ) {$5$};
				\node (n6) at (2.8, 0  ) {$6$};
				\node (n7) at (3.5  , 0  ) {$7$};
				\node (n8) at (4.2, 0  ) {$8$};
				\foreach \x/\y in {1/2, 2/3, 3/4, 3/5, 5/6, 6/7, 7/8}
				\draw [-] (n\x) to  (n\y);
				\end{tikzpicture}
		} 
\end{tabular}
\end{center}
\caption{(Simply laced) Dynkin graphs.  $\A_n$ (resp.~$\DD_n$) has $n\geq 1$ (resp.~$n\geq 4$) vertices.}
\label{TA:Dynkins}
\end{table}

\begin{definition}\label{D:canAD}
Consider the following two families of bidirected graphs.
\begin{enumerate}[label={\textnormal{(\alph*)}},topsep=3px,parsep=0px]
 \item For $r \geq 1$ and $c \geq 0$, take
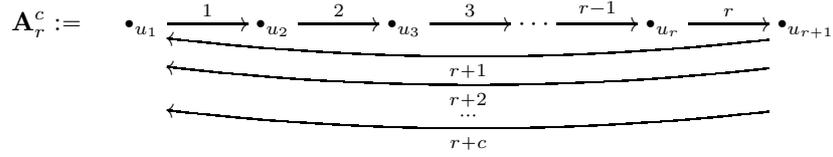
\[
\bdgA_r^c := \quad \xymatrix@C=2.5pc{{\bulito}_{u_1} \ar[r]^-{1} & {\bulito}_{u_2} \ar[r]^-{2} & {\bulito}_{u_3} \ar[r]^-{3} & \cdots \ar[r]^-{r-1} & {\bulito}_{u_r} \ar[r]^-{r} & {\bulito}_{u_{r+1}} \ar@/^8pt/@<1ex>[lllll]^-{r+1} \ar@/^8pt/@<3.5ex>[lllll]^-{r+2} \ar@/^8pt/@<7.3ex>[lllll]^-{r+c}_-{\cdots} }
\]

 \item For $r \geq 3$ and $c \geq 0$, take
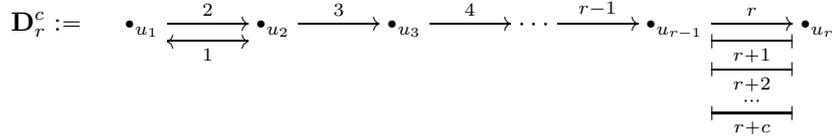
\[
\bdgD_r^c := \quad \xymatrix@C=2.5pc{{\bulito}_{u_1} \ar[r]^-{2} \ar@{<->}@<-1.5ex>[r]_-{1} & {\bulito}_{u_2} \ar[r]^-{3} & {\bulito}_{u_3} \ar[r]^-{4} & \cdots \ar[r]^-{r-1} & {\bulito}_{u_{r-1}} \ar[r]^-{r} \ar@{|-|}@<-1.5ex>[r]_-{r+1} \ar@{|-|}@<-4ex>[r]_-{r+2} \ar@{|-|}@<-7.8ex>[r]_-{r+c}^-{\cdots} & {\bulito}_{u_r} }
\]
\end{enumerate}
\end{definition}

\begin{remark}\label{R:canAD} (a)  It can be verified directly that
 the incidence bigraphs of $\bdgA_r^c$ and $\bdgD_r^c$ coincide with the so-called \textbf{canonical $c$-extensions}
 \[\widehat{\A}_r^{(c)}=\Delta(\bdgA_r^c)\quad \text{and}\quad \widehat{\D}_r^{(c)}=\Delta(\bdgD_r^c)
  \]
  of $\A_r$ ($r\geq 1$) and $\D_r$ ($r\geq 4$), respectively, introduced by Simson in \cite[Definition 2.2]{dS16a}. In particular, $\Delta(\bdgA_r^0)=\widehat{\A}_r^{(0)}=\A_r$ and $\Delta(\bdgD_r^0)=\widehat{\D}_r^{(0)}=\D_r$ are Dynkin graphs, and $q_{\bdgA_r^c}=q_{\widehat{\A}_r^{(c)}}$ and $q_{\bdgD_r^c}=q_{\widehat{\D}_r^{(c)}}$ are connected non-negative unit forms of corank $c\geq 0$, see \cite[Theorem 2.12(a)]{dS16a}, cf.~Lemma \ref{L:ful}.  There is a redundance of notation since one checks that $q_{\bdgA_3^{c}}\cong q_{\bdgD_3^{c}}$ for all $c\geq 0$. We keep $\bdgD_3^{c}$ for technical reasons, cf.~Proposition \ref{L:graph}.

  (b) Recall that in \cite{dS16a} also the canonical $c$-extensions $\widehat{\E}_r^{(c)}$ of $\E_r$ for $r=6,7,8$ were introduced, see also \cite{SZ17} for the details. Recall that $\widehat{\E}_r^{(0)}=\E_r$ and that $\E_r$ is a full subbigraph of $\widehat{\E}_r^{(c)}$ for any $r=6,7,8$ and $c\geq 1$. Note that by the results of the present paper $\widehat{\E}_r^{(c)}$ are not incidence bigraphs of any bidirected graph, see Corollary \ref{C:typeE}.
\end{remark}

 The following classification of non-negative unit forms is a reformulation of the main results of  Barot-de la Pe\~na \cite{BP99} and Simson and Zaj\k{a}c \cite{dS16a, SZ17}.

\begin{theorem}\label{T:wea}
Let $q:\ZZ^n\to\ZZ$ be a non-negative connected unit form with $n=r+c$ for the rank $r\geq 1$ and corank $c\geq 0$ of $q$. Then the following holds.
\begin{enumerate}[label={\textnormal{(\alph*)}},topsep=5px, parsep=0px]
  \item There exists a unique Dynkin graph $D_r \in \{\A_r,\D_r,\E_6,\E_7,$ $\E_8\}$   such that $q\Gweak q_{\widehat{D}_r^{(c)}}$,
 \item There exists a unique Dynkin graph $D_r \in \{\A_r,\D_r,\E_6,\E_7,$ $\E_8\}$   such that $q\weak q_{\widehat{D}_r^{(c)}}$,
 \item There exists a unique Dynkin graph $D_r \in \{\A_r,\D_r,\E_6,\E_7,$ $\E_8\}$   such that $q\weak q_{D_r}\oplus \xi^c$,
\end{enumerate} 
where $\zeta^c:\ZZ^c\to\ZZ$ denotes the zero form. Moreover, Dynkin graphs $D_r$ in $(a)$-$(c)$ coincide.
\end{theorem}

\begin{proof} 
The assertion (a) is the statement of \cite[Theorem 3.7(a)]{SZ17}. Since G-equivalence is a $\ZZ$-equivalence, the assertion (b) follows from (a), cf.~\cite[Theorem 2.12(c)]{dS16a}. Note that the uniqueness of $D_r$ in (a) and (b) is a consequence of the fact that $q_{\widehat{D}_r^{(c)}}\weak q_{\widehat{D}_r^{,(c)}}$ implies that $D_r=D'_{r}$. The latter follows by \cite[Theorem 2.12(b)]{dS16a}, \cite[Claim (a2), p.~29]{dS16a} and the fact that the unit forms associated to distinct Dynkin graphs are not $\Z$-equivalent (cf.~\cite[Lemma 3.5(a)]{MZ22}).

The assertion (c) is the statement of \cite[Theorem 2(a)]{BP99}. For the final claim observe first that the coincidence of Dynkin graphs $D_r$ from (a) and from (b) follow obviously  from the uniqueness of $D_r$ in these assertions. Whereas to show that $D_r$'s in (b) and (c) coincide it is enough to prove that $q_{\widehat{D}_r^{(c)}}\weak q_{D_r}\oplus \xi^c$. This latter claim can be easily shown by applying the reduction arguments from \cite[2.1(ii)]{BP99}.
\end{proof}

The unique (simply laced) Dynkin graph $D_r$ in Theorem~\ref{T:wea} is called the {\bf Dynkin type} of $q$, and is denoted by $\Dyn(q)=D_r$. We note that the papers \cite{BP99,dS16a, SZ17} provide few algorithmic procedures to compute $\Dyn(q)$. The following non-obvious fact is an easy conclusion from Theorem~\ref{T:wea}.

\begin{corollary}\label{cor:weakvsweakGforunit}
Given two non-negative connected unit forms $q$ and $q'$, the following conditions are equivalent:
\begin{enumerate}[label={\textnormal{(\roman*)}},topsep=3px,parsep=0px]
\item $q\weak q'$,
\item $q\Gweak q'$,
\item $\Dyn(q)=\Dyn(q')$ and $\crk(q)=\crk(q')$.
\end{enumerate}
\end{corollary}

\begin{remark}\label{rem:weakvsweakGforunit}
An alternative proof of the fact that the relations $\Gweak$ and  $\weak$ coincide on the class on non-negative connected unit forms is given in \cite[Proposition 1.2]{BKL06}. It should be emphasized that this coincidence does not hold in general. A counterexample for indefinite unit forms is given in \cite[Remark 4.1]{BKL06}, and for  non-negative connected Cox-regular forms in Remark \ref{rem:weakDynkins} below.
\end{remark}

\subsection{Auxiliary results} \label{subsec:auxThmA}
Before we prove Theorem A we need some more technical facts, mainly concerning the characterization of bidirected graphs $\bdg$ which induce an incidence bigraph $\Delta=\Delta(\bdg)$ of given specific properties.
For instance, the following proposition describes the connected incidence bigraphs having no dotted edge. The strategy of its proof is one common in considerations of unit forms: identify the \lq\lq critical'' minimal subgraphs satisfying a desired property (in this case, being an incidence graph having dotted edges), and describe the bidirected graphs that do not contain any \lq\lq critical'' subgraph (cf. Lemma~\ref{L:sub}).

\begin{proposition}\label{L:graph}
Let $\bdg$ be a bidirected graph such that $\Delta(\bdg)$  is a connected bigraph without dotted edges and with $n \geq 1$ vertices. Then there is a permutation $\pi$ of the set of arrows $\edgs{\bdg}$ of $\bdg$ and a switching $\bdg':=(\bdg \mathcal{P}^{\pi})^{O}$ of $\bdg \mathcal{P}^{\pi}$ having one of the following shapes $($cf.~Remark~\ref{R:loo} and Definition~\ref{D:canAD}):
 $$\onevertexbdg^{1,0,0}\  (n=1), \quad
  \mathbf{A}^0_n\  (n \geq 1), \quad
  \mathbf{A}^1_{n-1}\  (n \geq 2), \quad
  \mathbf{D}^0_n\  (n \geq 3), \ \  \text{or}\ \
  \mathbf{D}^1_{n-1}\  (n \geq 4).$$

\end{proposition}

\begin{proof}
Since $\Delta(\bdg)$ is connected then so is $\bdg=(\verts{\bdg},\edgs{\bdg},\edges)$, see Remark~\ref{R:con}(b). Assume first that $\Delta(\bdg)$ contains a (solid) loop. By connectedness and Lemmas~\ref{L:ful} and \ref{L:irr}(a), the bigraph $\Delta(\bdg)$ consists entirely of one vertex ($n=1$) and one solid loop, that is, $\bdg=\onevertexbdg^{1,0,0}$.  Assume that $\Delta(\bdg)$ has no loop. Then $\bdg$ has no loop by Lemma~\ref{L:ful} (see Table~\ref{TA:des}).

\medskip
\noindent \textit{Step~1.} The underlying graph $\widebar{\bdg}$ of $\bdg$ does not contain as subgraph any of the following:
\begin{enumerate}[label={\textnormal{(\roman*)}},topsep=3px,parsep=0px]
 \item $\xymatrix@!0@R=7pt@C=45pt{\bulito \ar@{-}[rd]^-{i_1} \\ & \bulito_v \ar@{-}[r]^-{i_3} & \bulito \\ \bulito \ar@{-}[ru]_-{i_2} }$ Indeed, in this case we have $(v,\epsilon_t) \in \edges(i_t)$ with signs $\epsilon_t \in \{\pm 1\}$, for $t=1,2,3$. Two of those signs must be equal, say $\epsilon_1=\epsilon_2$. Then, as vertices in $\Delta(\bdg)$, $i_1$ and $i_2$ are joint by a dotted edge (cf. Table~\ref{TA:des}), which is impossible.
 \item $\xymatrix@!0@R=7pt@C=45pt{\bulito \ar@{-}[r]^-{i} & \bulito_v \ar@{-}@<.4ex>[r]^-{j_1} \ar@{-}@<-.4ex>[r]_-{j_2} & \bulito_{v'} \ar@{-}[r]^-{k} & \bulito }$ In this case we have $\edges(j_t)=\{(v,\epsilon_t),(v',\epsilon'_t)\}$ with $\epsilon_t,\epsilon'_t \in \{\pm 1\}$ for $t=1,2$. Since both $j_1$ and $j_2$ are joint to $i$ by solid edges in $\Delta(\bdg)$, then $\epsilon_1=\epsilon_2$, and similarly $\epsilon'_1=\epsilon'_2$ using the incidence of $j_1$ and $j_2$ with $k$. Then $\edges(j_1)=\edges(j_2)$, and there are two dotted edges between $j_1$ and $j_2$ in $\Delta(\bdg)$, which is impossible.
 \item $\xymatrix@!0@R=15pt@C=45pt{& \bulito_v \ar@{-}@<.4ex>[dd]^-{j_2} \ar@{-}@<-.4ex>[dd]_-{j_1} \\ \bulito \ar@{-}[ru]^-{i} \ar@{-}[rd]_-{k} \\ & \bulito_{v'} }$ Same argument as in case (ii).
 \item $\xymatrix@!0@R=15pt@C=45pt{\bulito_v \ar@{-}@<1ex>[r]^-{i_1} \ar@{-}[r]|-{i_2} \ar@{-}@<-1ex>[r]_-{i_3} & \bulito_{v'}}$ If $|\verts{\bdg}|=2$, then $i_1,i_2,i_3$ must have the same sign $\sigma(i_1)=\sigma(i_2)=\sigma(i_3)$ as arrows in $\bdg$ (since $\bdg$ has no loop, cf.~Lemma \ref{L:irr}(d)). Then two of those arrows, say $i_1$ and $i_2$, must satisfy $\edges(i_1)=\edges(i_2)$. In this case $i_1$ and $i_2$ are joint in $\Delta(\bdg)$ by two dotted edges, which is impossible. If $|\verts{\bdg}|>2$, then by the connectedness of $\bdg$ there is an arrow $j$ which is incident but not parallel to $i_1$, say through vertex $v$. Then $i_1$, $i_2$ and $i_3$ are adjacent  to $j$ in $\Delta(\bdg)$, and taking $\edges(i_t)=\{(v,\epsilon_t),(v',\epsilon'_t)\}$ for $t=1,2,3$, we have $\epsilon_1=\epsilon_2=\epsilon_3$ because $\Delta(\bdg)$ does not contain dotted edges. Moreover, two of the signs $\epsilon'_1,\epsilon'_2,\epsilon'_3$ must coincide, say $\epsilon'_1=\epsilon'_2$, that is, $\edges(i_1)=\edges(i_2)$. Then there are two dotted edges between $i_1$ and $i_2$ in $\Delta(\bdg)$, which is impossible.
\end{enumerate}

\medskip
\noindent \textit{Step~2.} By Step~1 and the connectedness of $\bdg$, there is a permutation $\pi$ of the set of arrows $\edgs{\bdg}$ of $\bdg$ such that the underlying graph $\widebar{\bdg''}$ of $\bdg'':=\bdg \mathcal{P}^{\pi}$ is one of the following:
\begin{enumerate}[label={\textnormal{(\roman*)}},topsep=3px,parsep=0px]
 \item $\xymatrix@C=35pt{\bulito_{v_0} \ar@{-}[r]^-{1} & \bulito_{v_1} \ar@{-}[r]^-{2} & \bulito_{v_2} \ar@{-}[r]^-{3} & \cdots \ar@{-}[r]^-{n-2}  &  \bulito_{v_{n-2}} \ar@{-}[r]^-{n-1} & \bulito_{v_{n-1}} \ar@{-}[r]^-{n} & \bulito_{v_n}}$, for $n \geq 1$.
 \item $\xymatrix@C=35pt{\bulito_{v_0} \ar@{-}@/_15pt/[rrrrr]^-{n} \ar@{-}[r]^-{1} & \bulito_{v_1} \ar@{-}[r]^-{2} & \bulito_{v_2} \ar@{-}[r]^-{3} & \cdots \ar@{-}[r]^-{n-2}  &  \bulito_{v_{n-2}} \ar@{-}[r]^-{n-1} & \bulito_{v_{n-1}} }$, for $n \geq 2$.
 \item $\xymatrix@C=35pt{\bulito_{v_1} \ar@<.4ex>@{-}[r]^-{1} \ar@<-.4ex>@{-}[r]_-{2} & \bulito_{v_2} \ar@{-}[r]^-{3} & \cdots \ar@{-}[r]^-{n-2}  &  \bulito_{v_{n-2}} \ar@{-}[r]^-{n-1} & \bulito_{v_{n-1}} \ar@{-}[r]^-{n} & \bulito_{v_n}}$, for $n \geq 3$.
 \item $\xymatrix@C=35pt{\bulito_{v_1} \ar@<.4ex>@{-}[r]^-{1} \ar@<-.4ex>@{-}[r]_-{2} & \bulito_{v_2} \ar@{-}[r]^-{3} & \cdots \ar@{-}[r]^-{n-2}  &  \bulito_{v_{n-2}} \ar@<.4ex>@{-}[r]^-{n-1} \ar@<-.4ex>@{-}[r]_-{n} & \bulito_{v_{n-1}} }$, for $n \geq 4$.
\end{enumerate}

\medskip
\noindent \textit{Step~3.} In each of the cases (i-iv) of Step~2, there is a switching $\bdg':=(\bdg'')^O$ of $\bdg''$ satisfying
\[
\bdg' \in \{\mathbf{A}^0_n,\mathbf{A}^1_{n-1},\mathbf{D}^0_n,\mathbf{D}^1_{n-1}\}.
\]
Indeed, let $\edges''$ be the function of (signed) end-points for the arrows in $\bdg''$.

\textit{Case $($i$)$}. Take $\edges''(i)=\{(v_{i-1},\epsilon'_i),(v_i,\epsilon_i)\}$ for $i=1,\ldots,n$. Consider the signs $\eta_{i-1}:=\epsilon'_i$ for $i=1,\ldots,n$, $\eta_n:=-\epsilon_n$, and take $S:=\diag(\eta_0,\eta_1,\ldots,\eta_n)$. Since $\Delta(\bdg'')$ has no dotted edges, then $\epsilon_i\epsilon'_{i+1}=-1$ for $i=1,\ldots,n-1$. Therefore, the switching $(\bdg'')^S$ is a quiver of the shape as $\bdgA_n^0$, that is, it has an arrow $i$ from vertex $v_{i-1}$ to vertex $v_i$ for $i=1,\ldots,n$, which can be verified directly.

\textit{Case $($ii$)$}. Take $\edges''(i)=\{(v_{i-1},\epsilon'_i),(v_i,\epsilon_i)\}$ for $i=1,\ldots,n-1$ and $\edges''(n)=\{(v_0,\epsilon_0),(v_{n-1},\epsilon'_n)\}$. Consider the signs $\eta_i:=\epsilon'_i$ for $i=1,\ldots,n$, and take $S:=\diag(\eta_1,\ldots,\eta_{n})$. Since $\Delta(\bdg'')$ has no dotted edges, then $\epsilon_i\epsilon'_{i+1}=-1$ for $i=0,1,\ldots,n-1$. Therefore, the switching $(\bdg'')^S$ is a directed cycle as $\bdgA_{n-1}^1$.

\textit{Case $($iii$)$}. Take $\edges''(1)=\{(v_1,\epsilon_1),(v_2,\epsilon_1')\}$ and $\edges''(i)=\{(v_{i-1},\epsilon'_i),(v_i,\epsilon_i)\}$ for $i=2,\ldots,n$. Consider the signs $\eta_1:=\epsilon'_2$, $\eta_i:=\epsilon'_{i+1}$ for $i=2,\ldots,n-1$ and $\eta_n:=-\epsilon_n$, and take $S:=\diag(\eta_1,\ldots,\eta_n)$. Since $\Delta(\bdg'')$ has no dotted edges, we have $\epsilon'_1=\epsilon_2=-\epsilon'_3$, $\epsilon_1\epsilon'_2=-1$ and $\epsilon_i\epsilon'_{i+1}=-1$ for $i=3,\ldots,n-1$. Therefore, the switching $(\bdg'')^S$ has the shape as $\bdgD_n^0$.

\textit{Case $($iv$)$}. Take $\edges''(1)=\{(v_1,\epsilon_1),(v_2,\epsilon_1')\}$, $\edges''(i)=\{(v_{i-1},\epsilon'_i),(v_i,\epsilon_i)\}$ for $i=2,\ldots,n-1$ and $\edges''(n)=\{(v_{n-2},\epsilon_n),(v_{n-1},\epsilon'_n)\}$. Consider the signs $\eta_1:=\epsilon'_2$ and $\eta_i:=\epsilon'_{i+1}$ for $i=2,\ldots,n-1$, and take $S:=\diag(\eta_1,\ldots,\eta_{n-1})$. Since $\Delta(\bdg'')$ has no dotted edges, we have $\epsilon'_1=\epsilon_2=-\epsilon'_3$, $\epsilon_1\epsilon'_2=-1$, $\epsilon_i\epsilon'_{i+1}=-1$ for $i=3,\ldots,n-1$, and $\epsilon'_{n-1}=\epsilon_n=-\epsilon_{n-2}$. Therefore, the switching $(\bdg'')^S$ has the shape as $\bdgD_{n-1}^1$.

In each of the cases (i-iv) of Step~3, take an appropriate permutation matrix $Q$ to reorder the set of vertices of $\bdg''$ if necessary, and take the orthogonal matrix $O:=SQ$. Then the switching $\bdg':=(\bdg'')^O$ is precisely one of the bidirected graphs $\mathbf{A}^0_n$, $\mathbf{A}^1_{n-1}$, $\mathbf{D}^0_n$ or $\mathbf{D}^1_{n-1}$ respectively, as wanted.
\end{proof}

As first consequence of Proposition~\ref{L:graph} we discard the class of connected non-negative unit forms of Dynkin type $\EE$ as candidates to be incidence forms of bidirected graphs.

\begin{corollary}\label{C:typeE}
If $q$ is a connected non-negative unit form of Dynkin type $\EE_r$ for $r\in \{6,7,8\}$, then $q$ is not an incidence quadratic form.
\end{corollary}

\begin{proof}
Toward a contradiction, suppose that $q:\ZZ^n \to \ZZ$ is a connected non-negative unit form of corank $c\geq 0$ and Dynkin type $\Dyn(q)=\EE_r$. By Theorem~\ref{T:wea}(a), we have $q \Gweak q_{\widehat{\EE}^{(c)}_r}$ where $\widehat{\EE}^{(c)}_r$ is the canonical $c$-extension of the Dynkin graph $\EE_r$, and hence, by Corollary~\ref{C:tra}, $q_{\widehat{\EE}^{(c)}_r}$ is also an incidence quadratic form. By the definition of $\widehat{\EE}^{(c)}_r$, there is a full subbigraph of $\widehat{\EE}^{(c)}_r$ which coincides with the Dynkin graph $\EE_r$ (see Remark~\ref{R:canAD}(b)), and by Lemma~\ref{L:sub} it follows that $q_{\EE_r}$ is also an incidence form. That is, there is a bidirected graph $\bdg$ such that $\Delta(\bdg)=\EE_r$.  Since $\EE_r$ has no dotted edges (cf. Table~\ref{TA:Dynkins}), by Proposition~\ref{L:graph} there is a permutation $\pi$ of the set of arrows $\edgs{\bdg}$ of $\bdg$ and a switching $\bdg':=(\bdg \mathcal{P}^{\pi})^{O}$ of $\bdg \mathcal{P}^{\pi}$ such that $\bdg' \in \{\mathbf{A}^0_n,\mathbf{A}^1_{n-1},\mathbf{D}^0_n,\mathbf{D}^1_{n-1}\}$. In particular, $q_{\EE_r} \cong q_{\bdg'}$ has Dynkin type $\A$ or $\D$ (cf.~Lemma \ref{L:tra} and Remark \ref{R:canAD}), contradicting the uniqueness in Theorem~\ref{T:wea}(a).
\end{proof}

The following technical facts are needed for the proof of assertion (a) of Theorem A. Next lemma makes use of the uniqueness of the Dynkin type in Theorem~\ref{T:wea}, together with the description of rank and corank given in Corollary~\ref{cor:crk}.

\begin{lemma}\label{L:equal}
Assume that $\bdg$ and $\bdg'$ are different bidirected graphs in the set
$
\{\mathbf{A}^c_r\}_{r \geq 1, c\geq 0} \cup \{\mathbf{D}^c_r\}_{r \geq 3, c\geq 0}.
$
Then $q_{\bdg} \cong q_{\bdg'}$ if and only if $\{\bdg,\bdg'\}=\{\mathbf{A}^c_3,\mathbf{D}^c_3\}$ for some $c \geq 0$.
\end{lemma}
\begin{proof}
A direct verification shows that if $\{\bdg,\bdg'\}=\{\mathbf{A}^c_3,\mathbf{D}^c_3\}$ for any $c \geq 0$, then $q_{\bdg}\cong q_{\bdg'}$ (cf.~Remark~\ref{R:canAD}). For the converse, recall that  $\Rnk(q_{\mathbf{A}^c_r})=r=\Rnk(q_{\mathbf{D}^c_r})$ and $\CRnk(q_{\mathbf{A}^c_r})=c=\CRnk(q_{\mathbf{D}^c_r})$, cf.~Corollary \ref{cor:crk}. If $q_{\bdg} \cong q_{\bdg'}$, then the ranks and coranks of $q_{\bdg}$ and $q_{\bdg'}$ coincide. Thus $\bdg \neq \bdg'$ implies that $\{\bdg,\bdg'\}=\{\mathbf{A}^c_r,\mathbf{D}^c_r\}$ for some $r \geq 3$ and $c \geq 0$. If $r \geq 4$, then $\Dyn(q_{\mathbf{A}^c_r})=\A_r \neq \D_r =\Dyn(q_{\mathbf{D}^c_r})$ by Theorem \ref{T:wea}, thus $q_{\mathbf{A}^c_r}\not\sim q_{\mathbf{D}^c_r}$ and so $q_{\mathbf{A}^c_r}\not\cong q_{\mathbf{D}^c_r}$, cf.~Corollary \ref{cor:weakvsweakGforunit}. Therefore, the remaining option is $\{\bdg,\bdg'\}=\{\mathbf{A}^c_3,\mathbf{D}^c_3\}$ for some $c \geq 0$.
\end{proof}

The following lemma and corollary concern unit forms of Dynkin type $\AA$ and their relation to balanced bidirected graphs. The proofs uses the characterization of the (right) nullity of incidence matrices given in Proposition~\ref{P:null}.

\begin{lemma}\label{L:typeA}
Let $\bdg$ be a bidirected graph such that $q_{\bdg}= q_{\mathbf{A}^c_r}$ for some $r \geq 4$ and $c \geq 0$. Then $\bdg$ is balanced.
\end{lemma}
\begin{proof}
Assume first that $c \in \{0,1\}$. Then $\Delta(\bdg)=\Delta(q_{\bdg})$ has no dotted edges, and by Proposition~\ref{L:graph} there is a permutation $\pi$ of $\edgs{\bdg}$ and a switching $D:=(\bdg \mathcal{P}^{\pi})^{O}$ of $\bdg \mathcal{P}^{\pi}$ such that $D \in \{\mathbf{A}^c_r,\mathbf{D}^c_s\}_{r \geq 1, s \geq 3}$. Then we have $q_{D} \cong q_{\mathbf{A}_r^c}$ (see Lemma~\ref{L:tra}) for some $r \geq 4$, hence $D=\mathbf{A}_r^c$ by Lemma~\ref{L:equal}. Since $\mathbf{A}_r^c$ is balanced and $\Null(I(\bdg))=\Null(I((\bdg \mathcal{P}^{\pi})^O))$ (see Lemma~\ref{L:tra}), then $\bdg$ is balanced by Proposition~\ref{P:null}.

Assume now that $c>1$ and that the claim holds for any $r \geq 4$ and $c-1$. Take $q:=q_{\mathbf{A}^c_r}$, which is a quadratic form on $n:=r+c$ indeterminates satisfying $q_{n-1,n}=2$ (cf. Definition~\ref{D:canAD}(a) and Table~\ref{TA:des}), and let $\bdg=(\verts{\bdg},\edgs{\bdg},\edges)$ be a bidirected graph with $q_{\bdg}=q$. Consider the restriction $q':=q^{\un\setminus\{n\}}:\ZZ^{n-1}\to\ZZ$ of $q$ to the set $\un\setminus\{n\}$ (cf.~\eqref{eq:restriction}) and let $\bdg'$ be the bidirected graph obtained from $\bdg$ by removing arrow $n$. Clearly, $q'=q_{\mathbf{A}^{c-1}_r}$ and $q_{\bdg'}=q'$ (cf.~Lemma \ref{L:sub}). By induction hypothesis, $\bdg'$ is balanced. Since $q_{n-1,n}=2$, then the arrows $n-1$ and $n$ are parallel in $\bdg$ satisfying $\edges(n-1)=\edges(n)$, see Table~\ref{TA:des}. Then $\bdg$ is balanced.
\end{proof}

\begin{corollary}\label{C:typeA}
Any bidirected graph $\bdg$ such that $\Dyn(q_{\bdg}) =\A_r$ for some $r \geq 4$ is balanced.
\end{corollary}
\begin{proof}
By Theorem~\ref{T:wea}(a) and Remark~\ref{R:canAD}, there is a G-transformation $T$ of $q_{\bdg}$ such that $q_{\mathbf{A}_r^c}=q_{\bdg} \circ T$ where $c$ is the corank of $q_{\bdg}$. By Corollary~\ref{C:tra}, there is a bidirected graph $\bdg'$ such that $I(\bdg')=T^{\tr}I(\bdg)$, which implies that $q_{\bdg'}=q_{\mathbf{A}_r^c}$ and $\Null(I(\bdg'))=\Null(I(\bdg))$. Then $\bdg$ is balanced by Lemma~\ref{L:typeA} and Proposition~\ref{P:null}.
\end{proof}

Note that the assumption $r\geq 4$ in Lemma \ref{L:typeA} and Corollary \ref{C:typeA} can not be omitted, see Example \ref{ex:twobdgcrks}.

For the proof of Theorem A we will make use of the following version of the main result in~\cite{jaJ2018}, Theorem~5.5, formulated with the language and notation of this paper.

\begin{theorem}[{{\cite{jaJ2018}}}]\label{T:mainA}
The following are equivalent for an integral quadratic form $q$.
\begin{enumerate}[label={\textnormal{(\roman*)}},topsep=3px,parsep=0px]
\item The form $q$ is non-negative, connected and unitary with $\Dyn(q)=\A_r$ $(r \geq 1)$ and $\CRnk(q)=c \geq 0$.
\item There is a connected loop-less quiver $\bdg$ with $|\verts{\bdg}|=r+1$ and $|\edgs{\bdg}|=r+c$ such that $q=q_{\bdg}$.
\end{enumerate}
In this case, any  quiver $\bdg'$ satisfying $q_{\bdg}=q_{\bdg'}$ is given as a switching $\bdg'=\bdg^O$ of $\bdg$ with $O=(\pm 1)P$ for some permutation matrix $P$.
\end{theorem}

\subsection{{\it Proof of Theorem A from Section \ref{sec:intro}.}}\label{subsec:proofA}

Let $q:\ZZ^n\to \ZZ$ be a connected unit form. Assume that $q$ is non-negative of Dynkin type $\A$ or $\D$. By Theorem~\ref{T:wea}(a) we have $q \Gweak q_{\Delta}$, where $\Delta=\hat{D}^{(c)}_r$ is the canonical extension of Dynkin graph $D_r=\A_r$ or $D_r=\D_r$ respectively. By Remark~\ref{R:canAD}, $q_{\Delta}$ is an incidence quadratic form, and so is $q$ by Corollary~\ref{C:tra}, that is, $q=q_{\bdg}$ for some bidirected graph $\bdg$. Since $q$ is connected and unitary, it follows that $\bdg$ is connected and has no loop, see Remark~\ref{R:con}(b) and Lemmas~\ref{L:ful}, \ref{L:irr}(a). Conversely, if $q$ is an incidence quadratic form, then $q$ is non-negative by Lemma \ref{L:ful} and the Dynkin type of $q$ is not $\EE_r$ by Corollary~\ref{C:typeE}. Then $q$ has Dynkin type $\A$ or $\D$ by Theorem~\ref{T:wea}.

To show (a), assume first that $\Dyn(q_{\bdg})=\D_r$ for some connected loop-less bidirected graph $\bdg$  ($r \geq 4$). If $\Null(I(\bdg))=1$, then by Proposition~\ref{P:null}(b) there is a switching $\bdg^O$ of $\bdg$ such that $\bdg^O$ is a quiver, and therefore $\Dyn(q_{\bdg^O})=\A_s$ for some $s \geq 1$ by Theorem~\ref{T:mainA}. This contradicts the uniqueness of the Dynkin type in Theorem~\ref{T:wea}, since $q_{\bdg}=q_{\bdg^O}$ by Lemma~\ref{L:tra}(d). Then $\Null(I(\bdg))=0$, that is, $\bdg$ is unbalanced by Proposition~\ref{P:null}(a). Morevoer, by Corollary~\ref{cor:crk} we have $|\verts{\bdg}|=\Rnk(q_{\bdg})=r \geq 4$. Conversely, assume that $r:=|\verts{\bdg}|\geq 4$ and that $\bdg$ is unbalanced for a connected loop-less bidirected graph $\bdg$. If the Dynkin type of $q_{\bdg}$ is not $\D_r$, then by the main claim of the theorem we have $\Dyn(q_{\bdg})=\A_r$. By Corollary~\ref{C:typeA} $\bdg$ is balanced, a contradiction.

Claim (b) follows by Corollary~\ref{cor:crk}. This completes the proof.\hfill$\Box$

\smallskip

\begin{example}\label{ex:twobdgcrks}  Take $\bdg$ and $\bdg'$ with $q:=q_\bdg=q_{\bdg'}$ as in Example \ref{ex:twobdg}. Note that $\bdg\cong\bdgD_3^0$, $\bdg'=\bdgA_3^0$ and $q=q_{\AA_3}$ so $q$ has Dynkin type $\AA_3$ and $\crk(q)=0$, that is, $q$ is positive. Observe that $\bdg$ is unbalanced (hence $\cyccond_\bdg=0$), and $|\verts{\bdg}|=3<4$, compare with Theorem A(a). Whereas $\cyccond_{\bdg'}=1$ and we have
\[|\edgs{\bdg}|-|\verts{\bdg}|+\cyccond_{\bdg}\, =\, 0\, =\, |\edgs{\bdg'}|-|\verts{\bdg'}|+\cyccond_{\bdg'},
\]
cf.~Theorem A(b). Moreover, $\Null(I(\bdg))=0$, $\Null(I(\bdg'))=1$ and $\bdg'$ is a quiver, cf.~Proposition \ref{P:null}.
\end{example}

\section{Dynkin type $\CC$}\label{sec:C}
In contrast to the unitary case and simply laced Dynkin type, it seems that there is no well-established notion of the Dynkin type for the non-unitary (Cox-regular) non-negative integral quadratic forms in the literature, see the discussion in \cite{MZ22}, cf.~\cite{KS15a, KS15b,jeH72} and Remark \ref{rem:weakDynkins}. For our purposes we use the following definition. Its meaningfulness is explained later, see Lemma \ref{L:cla} and subsections \ref{subsec:DynCclass}, \ref{subsec:rs}.

\begin{definition}\label{D:tyc}
Given $n\geq 1$, a non-negative connected irreducible integral quadratic form $q:\Z^n \to \Z$ is said to have \textbf{Dynkin type $\CC$} if it satisfies the following conditions:
   \begin{enumerate}[label={\textnormal{(\roman*)}},topsep=3px,parsep=1px]
    \item[\Ci] $q$ is non-unitary;
    \item[\Cii] $1\leq q_i \leq 2$ for all $i \in \underline{n}$;
    \item[\Ciii] $q$ is fully regular (thus also Cox-regular).
   \end{enumerate}
If the form $q$ has rank $r$, then we also say that its Dynkin type is $\CC_r$, written $\Dyn(q)=\CC_r$.
\end{definition}

The following observation  is an easy consequence of Corollary \ref{cor:Gweak}. 
\begin{lemma}\label{lem:GweakDynC}
Let $q$ and $q'$ be two G-equivalent Cox-regular integral quadratic forms, that is, $q\Gweak q'$. Then $q$ is non-negative of Dynkin type $\CC_r$ if and only if so is $q'$.
\end{lemma}

The following limitations can be derived from general bounds on the coefficients of non-negative Cox-regular forms from \cite[Theorem 5.3]{MM21coeffs}. However, in the  particular case of Dynkin type $\CC$ the proof is elementary.

\begin{lemma}\label{lem:DynCbounds}
Let $q:\Z^n \to \Z$ be a non-negative Cox-regular form of Dynkin type $\CC_r$ and fix  $i\neq j$ in $\un$. Then
\begin{enumerate}[label={\textnormal{(\alph*)}},topsep=3px,parsep=0px]
\item $n,r\geq 2$ and $q_{i_0}=1$ for at least one $i_0\in\un$,
\item $-2\leq q_{ij}\leq 2$ provided $q_i=q_j=1$,
\item $q_{ij}\in\{-2,0,2\}$ provided $\{q_i, q_j\}=\{1,2\}$,
\item $q_{ij}\in\{-4,0,4\}$ provided $q_i=q_j=2$.
\end{enumerate}
\end{lemma} 
\begin{proof} Since $q$ is fully regular and irreducible then it has to admit $q_{i_0}=1$ by Remark \ref{R:irr}(a).
On the other hand, by the conditions \Ci-\Cii\, there exists $i_1\neq i_0$ with $q_{i_1}=2$. Thus $n\geq 2$ and $G_q$ has the minor $\left|{\scriptsize\begin{array}{cc}4&\!\! a\\a&\!\!2\end{array}}\right|=8-a^2\neq 0$ since $a:=q_{i_1i_0}\in\ZZ$. Therefore $r=\Rnk(q)=\Rnk(G_q)\geq 2$. 

To show the remaining claims note that by the non-negativity of $q$ we have $0\leq q(\bas_i+\xi\bas_j)=q_i+q_j+\xi q_{ij}$ for $\xi\in\{\pm1\}$. Now (b)-(d) follow easily by applying the conditions \Cii-\Ciii\, of Definition \ref{D:tyc}.
\end{proof}

\begin{remark}\label{rem:borderDynC} 
A connected incidence form $q=q_\bdg$ of a (connected) bidirected graph $\bdg$ satisfies the conditions \Ci-\Ciii\, of Definition \ref{D:tyc} with $q_i=2$ for all $i\in\un$ if and only if $\bdg=\onevertexbdg^{0,s,t}$ is a one-vertex bidirected graph as defined in
Remark \ref{R:loo}, for some $s+t=n$, cf.~Lemma \ref{L:ful}.
\end{remark}

Fix a connected unit form $q:\ZZ^n\to\Z$ and assume that $q$ is classic, that is, $q_{ij}\leq 0$ for all $i\neq j$. It is well known that then $q$ is non-negative of  Dynkin type $\AA$ if and only if $q=q_{D}$ for the Dynkin graph $D=\AA_n$ (and in this case $\CRnk(q)=0$) or the Euclidean graph $D=\widetilde{\A}_{n-1}=\widehat{\A}_{n-1}^{(1)}$ (and in this case $\CRnk(q)=1$), see \cite{BP99, dS13} (note that this can also be derived from our Proposition \ref{L:graph}). 
Similarly, $q$ is non-negative of  Dynkin type $\DD$ if and only if $q=q_{D}$ for the Dynkin graph $D=\DD_n$ with $n\geq 4$ (and in this case $\CRnk(q)=0$) or the Euclidean graph $D=\widetilde{\DD}_{n-1}=\widehat{\DD}_{n-1}^{(1)}$ with $n\geq 5$ (and in this case $\CRnk(q)=1$). We have the following analog of these known facts.

\begin{lemma} \label{L:cla}
Let $q:\Z^n \to \Z$ be a connected, irreducible,  non-unitary, classic and Cox-regular integral quadratic form. Then $q$ is non-negative of Dynkin type $\CC$ if and only if one of the following conditions hold:
\begin{enumerate}[label={\textnormal{(\alph*)}},topsep=3px,parsep=1px]
\item $q=q_{\C_n}$ for the Dynkin bigraph $\C_n$ with $n\geq 2$ $($and in this case $\CRnk(q)=0)$, where
\begin{equation*}\label{eq:C}
\begin{tabular}{ll}
		$\C_n:$&\begin{minipage}{4.8cm}
			{\scriptsize
				\begin{tikzpicture}[auto]
				\node (q1) at (0  , 0  ) {$1$};
				\node (q2) at (0.8, 0  ) {$2$};
				\node (q3) at (1.6, 0  ) {$3$};
				\node (q4) at (2.3, 0  ) {};
				\node (q5) at (2.8, 0  ) {};
				\node (q6) at (3.6, 0  ) {$n$};
				\draw[dotted, -] (q4) edge node{}(q5);
				\foreach \x/\y in {1/2, 2/1}
				\draw[-] (q\x) edge [bend left=8] node{}(q\y);
				\foreach \x/\y in {2/3, 3/4, 5/6}
				\draw[-] (q\x) edge node{}(q\y);
				\draw[dashed] ([yshift=4pt,xshift=4pt]q1) arc (-50:240:2.2mm);
				\end{tikzpicture}
		} \end{minipage}
\end{tabular}
\end{equation*}
\item $q=q_D$ for one of the Euclidean bigraphs $D=\wt{\C}_{n-1}$ or $D=\wt{\mathcal{CD}}_{n-1}$ with $n\geq 3$ $($and in this case $\CRnk(q)=1)$.
\begin{equation*}\label{eq:CtCDt}
\begin{tabular}{llll}
		$\widetilde{\C}_s:$&
		\begin{minipage}{5cm}
			{\scriptsize
				\begin{tikzpicture}[auto]
				\node (q1) at (0  , 0  ) {$1$};
				\node (q2) at (0.8, 0  ) {$2$};
				\node (q3) at (1.6, 0  ) {$3$};
				\node (q4) at (2.3, 0  ) {};
				\node (q5) at (2.8, 0  ) {};
				\node (q7) at (3.6, 0  ) {$s$};
				\node (q8) at (4.4, 0  ) {$s{\tiny+}1$};
				\draw[dotted, -] (q4) edge node{}(q5);
				\foreach \x/\y in {1/2, 2/1, 7/8, 8/7}
				\draw[-] (q\x) edge [bend left=8] node{}(q\y);
				\foreach \x/\y in {2/3, 3/4, 5/7}
				\draw[-] (q\x) edge node{}(q\y);
				\draw[dashed] ([yshift=4pt,xshift=3pt]q1) arc (-50:240:2.2mm);
				\draw[dashed] ([yshift=4pt,xshift=4pt]q8) arc (-50:240:2.2mm);
				\end{tikzpicture}
		} \end{minipage} & & \\
$\widetilde{\C\CD}_s:$&
		\begin{minipage}{5.2cm}
			{\scriptsize
				\begin{tikzpicture}[auto]
				\node (q1) at (0  , 0  ) {$1$};
				\node (q2) at (3.6, 0.7) {$s{\tiny+}1$};
				\node (q3) at (0.8, 0  ) {$2$};
				\node (q4) at (1.6, 0  ) {$3$};
				\node (q5) at (2.3, 0  ) {};
				\node (q6) at (2.8, 0  ) {};
				\node (q8) at (3.6, 0  ) {$s{\tiny-}1$};
				\node (q9) at (4.5, 0  ) {$s$};
				\draw[dotted, -] (q5) edge node{}(q6);
				\foreach \x/\y in {2/8, 3/4, 4/5, 6/8, 8/9}
				\draw[-] (q\x) edge node{}(q\y);
				\foreach \x/\y in {1/3, 3/1}
				\draw[-] (q\x) edge [bend left=8] node{}(q\y);
				\draw[dashed] ([yshift=4pt,xshift=3pt]q1) arc (-50:240:2.2mm);
				\end{tikzpicture}
		} \end{minipage} & $\widetilde{\C\CD}_2=\wt{\mathcal{B}}_2:$&
\begin{minipage}{2cm}
			{\scriptsize
				\begin{tikzpicture}[auto]
				\node (q1) at (0  , 0  ) {$2$};
				\node (q2) at (0.8, 0  ) {$1$};
				\node (q3) at (1.6, 0  ) {$3$};
				\foreach \x/\y in {1/2, 2/1, 2/3, 3/2}
				\draw[-] (q\x) edge [bend left=8] node{}(q\y);
				\draw[dashed] ([yshift=4pt,xshift=4pt]q2) arc (-50:240:2.2mm);
				\end{tikzpicture}
		} \end{minipage}
\end{tabular}
\end{equation*}
\end{enumerate}
\end{lemma}
We refer to \cite[Tables 1\! \&\! 2]{MZ22} for the full lists of Dynkin and Euclidean bigraphs. Recall that they arise via integral forms from \lq\lq classical'' simply laced and  not simply laced Dynkin and Euclidean diagrams  well-known in the theory of species and Lie algebras, see \cite{DR76} and \cite{jeH72, vK90}. We refer to \cite{MM19} for the detailed discussion of this correspondence, see also \cite[Remark 2.6]{MZ22} and \cite{KS15a}.

\begin{proof}
It is clear by a direct checking  that $q_{\C_n}$ and $q_D$ for $D=\wt{\C}_{n-1}$ or $D=\wt{\mathcal{CD}}_{n-1}$ are connected, irreducible, non-unitary, classic and fully regular forms. We also verify that they are non-negative of coranks 0 and 1, respectively (cf.~\cite[Lemma 3.1]{MZ22}). In particular, they satisfy Definition \ref{D:tyc}.

To show the opposite implication, assume that $q$ is a connected, irreducible,  classic and Cox-regular integral quadratic form. If it is non-negative, then it can be derived from the result of Dlab-Ringel \cite[Proposition 1.2]{DR76} that
$q=q_D$ for $D$ being one of the Dynkin bigraphs $D\in\{\AA_n, \DD_n, \EE_{6/7/8}, \CB_n, \C_n, \CF_4, \CG_2\}$ or one of the Euclidean bigraphs $D\in\{\widetilde{\A}_m,\widetilde{\D}_m,\widetilde{\EE}_{6/7/8}, \widetilde{\CA}_{11},\widetilde{\CB}_m,\widetilde{\C}_m,\widetilde{\CB\C}_m,\widetilde{\CB\CD}_m,
\widetilde{\C\CD}_m,\widetilde{\CF}_{41},$ $\widetilde{\CF}_{42},\widetilde{\CG}_{21},\widetilde{\CG}_{22}\}$ from \cite[Tables 1\! \&\! 2]{MZ22}, see \cite[Proposition 2.9]{MM19}. Now, if we additionally assume that $q$ has Dynkin type $\CC$, then a case by case verification shows that $D$ has to be  $\C_n, \wt{\C}_{n-1}$ or $\wt{\mathcal{CD}}_{n-1}$.
\end{proof}

\begin{remark}\label{rem:weakDynkins}
It follows from \cite[Proposition 2.4(a)]{KS15b} that $q_{\C_n}\weak q_{\D_n}$ and $q_{\wt{\C}_{n}}\weak q_{\wt{\D}_n}\weak q_{\wt{\mathcal{CD}}_{n}}$ for each $n\geq 4$. This shows that the $\ZZ$-equivalence does not preserve unit forms and that (some of) the non-negative unit forms of Dynkin type $\DD$ are $\ZZ$-equivalent to (some of) the non-negative Cox-regular forms of Dynkin type $\CC$. However, observe that the quadratic forms $q_{\C_n}, q_{\D_n}$, resp.~$q_{\wt{\C}_{n}}$, $q_{\wt{\D}_n}$, $q_{\wt{\mathcal{CD}}_{n}}$ are pairwise not G-equivalent since by Corollary \ref{cor:Gweak} G-equivalence preserves the number of (dotted) loops in the associated bigraphs. More comprehensive discussion of these phenomena is provided in subsection \ref{subsec:DynCclass}.
\end{remark}


\subsection{Auxiliary results}\label{subsec:aux}

Before we prove Theorem B we need some more technical facts. By \textbf{rigid G-transformation} of a Cox-regular form $q:\ZZ^n \to \ZZ$ we mean a composition $T$ of sign inversions and (well-defined) Gabrielov transformations. In this case, using Lemma~\ref{L:inv}(iii) and a straightforward computation for sign inversion, if $q':=q \circ T$, then we have $q'_i=q_i$ for all $i \in \underline{n}$ (compare with Corollary~\ref{cor:Gweak}).

\begin{lemma}\label{L:piv}
Let $q:\Z^n\to\Z$ be a connected Cox-regular quadratic form, and fix a vertex $i_0 \in \underline{n}$. Then there is a rigid G-transformation $T$ such that the quadratic form $q':=q \circ T$ satisfies
\begin{equation}\label{EQ:pos}
q'_{i_0j} >0, \qquad \text{for all $j \in \underline{n}$.}
\end{equation}
\end{lemma}
\begin{proof}
Consider the set $S_q(i_0)=\{i \in \underline{n} \colon q_{i_0i} =0\}$, and assume that $S_q(i_0) \neq \emptyset$. By connectedness we may find vertices $i \notin S_q(i_0)$ and $j \in S_q(i_0)$ such that $i \neq i_0$ and $q_{ij} \neq 0$. Take $q'=q \circ T^q_{ij}$, and using equations (\ref{EQ:coe}) observe that $S_{q'}(i_0) = S_q(i_0) \setminus \{j\}$. Since $q'$ is a connected Cox-regular quadratic form (cf.~Lemma \ref{L:inv} and \cite[Lemma~4.9]{MZ22}), we may proceed in this way to find an iterated Gabrielov transformation $T'$ such that $S_{q''}(i_0)=\emptyset$ for $q''=q\circ T''$. To complete the proof, compose $T'$ with appropriate sign inversions.
\end{proof}

The following technical lemma contains the induction step for the proof of Theorem B. See subsection~\ref{subsec:comp} for an implementable summary of this result. As usual, we take $\delta_{i,j}=1$ if $i=j$, and $\delta_{i,j}=0$ otherwise.

\begin{lemma}\label{L:techC}
Let $q:\Z^n \to \Z$ be a non-negative Cox-regular form of Dynkin type $\CC$ such that and $q_{1,i}>\delta_{1,i}$ for $i=1,\ldots,n$. Then there exist $m \geq 2$ and a partition $U$ of $\{1,\ldots,n\}$,
\begin{equation}\label{eq:parC}
U=\{U^2_{1,-1},U^1_{2,+1},U^1_{2,-1},\ldots,U^1_{m,+1},U^1_{m,-1}\},
\end{equation}
with some of the parts $U^1_{v,\epsilon}$ possibly empty, but $U^2_{1,-1} \neq \emptyset$ and $U^1_{v,+1}\cup U^1_{v,-1} \neq \emptyset$ for $v=2,\ldots,m$, such that for $1 \leq i \leq j \leq n$ with $i \in U^k_{v,\epsilon}$ and $j \in U^{k'}_{v',\epsilon'}$ we have
\begin{equation}\label{eq:conC}
q_{ij}= \left\{
\begin{array}{c l}
\frac{k|\epsilon+\epsilon'|}{1+\delta_{i,j}}, & \text{if $v=v'$ (hence, $k=k'$)}, \\[1ex]
kk', & \text{if $v \neq v'$}.
\end{array} \right.
\end{equation}
\end{lemma}
\begin{proof}
Since $q_i \in \{1,2\}$ for $i=1,\ldots,n$ and $q_1=q_{1,1}>1$, then we have $q_1=2$. If $n=2$, then necessarily $q_2=1$ and $q_{1,2}=2$ (see Lemma~\ref{lem:DynCbounds}), and we can take $m=2$, $U^2_{1,-1}=\{1\}$, $U^1_{2,+1}=\{2\}$ and $U^1_{2,-1}=\emptyset$. Assume that $n>2$, and note that for $i \neq j$ in $\{2,\ldots,n\}$ we have $0 \leq q(\bas_i+\bas_j-\bas_1)=q_1-(q_{1,i}-q_i)-(q_{1,j}-q_j)+q_{ij}$, that is,
\[
q_{ij} \geq (q_{1,i}-q_i)+(q_{1,j}-q_j)-2. \tag{$*$}
\]
Since $q$ is fully regular and $q_1=2$, we have $q_{1,i}-q_i >0$ and $q_{1,i}-q_i >0$. Using~($*$), this shows that
\[
q_{ij} \geq 0, \qquad \text{for all $i,j \in \{1,\ldots,n\}$.}
\]
Moreover, if $q_i=q_j=2$, then $q_{1,i}=q_{1,j}=4$, and $q_{ij}\geq (2)+(2)-2=2$ by $(*)$. However, since $q$ is fully regular, it follows that $q_{ij} \geq 4$. Similarly, if $q_i=2$ and $q_j=1$, then  $q_{1,i}=2$ and $q_{1,j}=4$, and thus $q_{ij}\geq (1)+(2)-2=1$ by $(*)$. Since $q$ is Cox-regular, then $q_{ij}\geq 2$. Therefore, by Lemma~\ref{lem:DynCbounds} we have
\[
q_{ij} =q_iq_j, \qquad \text{for $i \neq j$ in $\{2,\ldots,n\}$ with $q_iq_j>1$}. \tag{$\dagger$}
\]
In particular, this shows the following: in case $q_i=2$ for $i=1,\ldots,n-1$ (hence $q_n=1$ since $q$ is irreducible), then the claim of the lemma holds for $m=2$ and the partition $U^2_{1,-1}=\{1,\ldots,n-1\}$ and $U^1_{2,+1}=\{n\}$ (with $U^1_{2,-1}=\emptyset$). The remaining cases, that is, when $n>2$ and $q_i=1$ for some $i\in \{2,\ldots,n-1\}$, are constructed by induction as follows. Note that the restriction $q^{\underline{n-1}}:\ZZ^{n-1}\to\ZZ$ of $q$ to the set $\underline{n-1}=\{1,\ldots,n-1\}$ (cf.~\eqref{eq:restriction}) is connected by the condition $q_{1,j}>0$ for all $i$ in $\{1,\ldots,n\}$. It is easy to verify that $q^{\underline{n-1}}$ has Dynkin type $\CC$ (using that $q_i=1$ for some $i \in \{2,\ldots,n-1\}$) with $q_{1,i}>\delta_{1,i}$ for $i=1,\ldots,n-1$. By induction hypothesis, there is an integer $\wt{m} \geq 2$ and a partition $\wt{U}$ of $\{1,\ldots,n-1\}$,
\[
\wt{U}=\{ \wt{U}^2_{1,-1}, \wt{U}^1_{2,+1}, \wt{U}^1_{2,-1},\ldots, \wt{U}^1_{m,+1}, \wt{U}^1_{\wt{m},-1}\},
\]
where~(\ref{eq:conC}) holds for $i\leq j$ in $\{1,\ldots,n-1\}$. Take $\wt{U}^1:=\bigcup \wt{U}^1_{v,\epsilon}$ for $v=2,\ldots,\wt{m}$ and $\epsilon \in \{\pm 1\}$, and consider the following:

\vspace{2mm}
\noindent \textit{Case 1.} Assume that $q_n=2$. Using~($\dagger$), for $1<i<n$ we have $q_{in}=4$ if $q_i=2$, and $q_{in}=2$ if $q_i=1$. In this case, take $m:=\wt{m}$ and $U^2_{1,-1}:=\wt{U}^2_{1,-1}\cup \{n\}$.

\vspace{2mm}
\noindent \textit{Case 2.} Assume that $q_n=1$ and $q_{in}=2$ for some $i\in \wt{U}^1$. In this case we claim that $q_{jn}=q_{ij}$ for all $j \in \wt{U}^1 \setminus \{i\}$. Indeed, otherwise one of the following bigraphs would be fully contained in the bigraph $\Delta_q$ of $q$,
\[
\xymatrix@C=1pc@R=1pc{ \textrm{(x1)}& {n} \ar@{.}[rd] \ar@{.}@<1.5pt>[dd] \ar@{.}@<-1.5pt>[dd] \\ {1} \ar@{.}@(lu,ld) \ar@{.}@<1.5pt>[rr] \ar@{.}@<-1.5pt>[rr] \ar@{.}@<1.5pt>[rd] \ar@{.}@<-1.5pt>[rd] \ar@{.}@<1.5pt>[ru] \ar@{.}@<-1.5pt>[ru] && {j} \ar@{.}@<1.5pt>[ld] \ar@{.}@<-1.5pt>[ld] \\ & {i} }
\qquad \qquad
\xymatrix@C=1pc@R=1pc{ \textrm{(x2)}& {n}  \ar@{.}@<1.5pt>[dd] \ar@{.}@<-1.5pt>[dd] \\ {1} \ar@{.}@(lu,ld) \ar@{.}@<1.5pt>[rr] \ar@{.}@<-1.5pt>[rr] \ar@{.}@<1.5pt>[rd] \ar@{.}@<-1.5pt>[rd] \ar@{.}@<1.5pt>[ru] \ar@{.}@<-1.5pt>[ru] && {j} \\ & {i} \ar@{.}[ru] }
\qquad \qquad
\xymatrix@C=1pc@R=1pc{ \textrm{(x3)}& {n} \ar@{.}@<1.5pt>[dd] \ar@{.}@<-1.5pt>[dd] \\ {1} \ar@{.}@(lu,ld) \ar@{.}@<1.5pt>[rr] \ar@{.}@<-1.5pt>[rr] \ar@{.}@<1.5pt>[rd] \ar@{.}@<-1.5pt>[rd] \ar@{.}@<1.5pt>[ru] \ar@{.}@<-1.5pt>[ru] && {j} \ar@{.}@<1.5pt>[ld] \ar@{.}@<-1.5pt>[ld] \\ & {i} }
\]
or one of the bigraphs $\textrm{(x1', x2', x3')}$ obtained by swapping positions of vertices $n$ and $i$. This is impossible, since the quadratic form associated to each of such bigraphs is indefinite (evaluate at vector $2\bas_n-2\bas_i+\bas_j$ for the cases $\textrm{(x1, x2, x3)}$, and at vector $2\bas_i-2\bas_n+\bas_j$ for the cases $\textrm{(x1', x2', x3')}$). If $i \in \wt{U}^1_{v,\epsilon}$ take $m:=\wt{m}$ and $U^1_{v,\epsilon}:=\wt{U}^1_{v,\epsilon} \cup \{n\}$.

\vspace{2mm}
\noindent \textit{Case 3.} Assume that $q_n=1$ and $q_{in} \leq 1$ for all $i\in \wt{U}^1$, and that there exists $i\in \wt{U}^1$ with $q_{in}=0$. In this case we claim that $q_{jn}+q_{ij}=2$ for all $j \in \wt{U}^1 \setminus \{i\}$. Indeed, otherwise one of the following bigraphs would be fully contained in the bigraph $\Delta_q$ of $q$,
\[
\xymatrix@C=1pc@R=1pc{ \textrm{(y1)}& {n} \\ {1} \ar@{.}@(lu,ld) \ar@{.}@<1.5pt>[rr] \ar@{.}@<-1.5pt>[rr] \ar@{.}@<1.5pt>[rd] \ar@{.}@<-1.5pt>[rd] \ar@{.}@<1.5pt>[ru] \ar@{.}@<-1.5pt>[ru] && {j}  \\ & {i} }
\qquad
\xymatrix@C=1pc@R=1pc{ \textrm{(y2)}& {n} \\ {1} \ar@{.}@(lu,ld) \ar@{.}@<1.5pt>[rr] \ar@{.}@<-1.5pt>[rr] \ar@{.}@<1.5pt>[rd] \ar@{.}@<-1.5pt>[rd] \ar@{.}@<1.5pt>[ru] \ar@{.}@<-1.5pt>[ru] && {j} \\ & {i} \ar@{.}[ru] }
\qquad
\xymatrix@C=1pc@R=1pc{ \textrm{(y3)}& {n} \ar@{.}[rd] \\ {1} \ar@{.}@(lu,ld) \ar@{.}@<1.5pt>[rr] \ar@{.}@<-1.5pt>[rr] \ar@{.}@<1.5pt>[rd] \ar@{.}@<-1.5pt>[rd] \ar@{.}@<1.5pt>[ru] \ar@{.}@<-1.5pt>[ru] && {j} \ar@{.}@<1.5pt>[ld] \ar@{.}@<-1.5pt>[ld] \\ & {i} }
\qquad
\xymatrix@C=1pc@R=1pc{ \textrm{(y4)}& {n} \ar@{.}@<1.5pt>[rd] \ar@{.}@<-1.5pt>[rd] \\ {1} \ar@{.}@(lu,ld) \ar@{.}@<1.5pt>[rr] \ar@{.}@<-1.5pt>[rr] \ar@{.}@<1.5pt>[rd] \ar@{.}@<-1.5pt>[rd] \ar@{.}@<1.5pt>[ru] \ar@{.}@<-1.5pt>[ru] && {j} \ar@{.}@<1.5pt>[ld] \ar@{.}@<-1.5pt>[ld] \\ & {i} }
\]
or one of the bigraphs $\textrm{(y2', y3')}$ obtained by swapping positions of vertices $n$ and $i$. This is impossible, since the quadratic form associated to each of such bigraphs is indefinite (evaluate at vector $2\bas_n+2\bas_i-2\bas_1-\bas_j$ for the cases $\textrm{(y1, y2, y2')}$, and at vector $2\bas_n+2\bas_i-3\bas_j$ for the cases $\textrm{(y3, y3', y4)}$). If $i \in \wt{U}^1_{v,\epsilon}$ take $m:=\wt{m}$ and $U^1_{v,-\epsilon}:=\wt{U}^1_{v,-\epsilon} \cup \{n\}$. In this situation we must have $\wt{U}^1_{v,-\epsilon}=\emptyset$, that is, $U^1_{v,-\epsilon}=\{n\}$; indeed, if $j \in \wt{U}^1_{v,-\epsilon}$, then by induction hypothesis we have $q_{ij}=|\epsilon-\epsilon|=0$, and thus $q_{jn}=2$ since $q_{jn}+q_{ij}=2$, contradicting the assumptions of this case (note that $j \in \wt{U}^1$).

\vspace{2mm}
\noindent \textit{Case 4.} Assume that $q_n=1$ and $q_{in} = 1$ for all $i\in \wt{U}^1$. In this case we take $m:=\wt{m}+1$, $U^1_{m,+1}:=\{n\}$ and $U^1_{m,-1}:=\emptyset$.

\vspace{2mm}
All together, we get a partition $U$ of $\{1,\ldots,n\}$ as in~(\ref{eq:parC}) by adding vertex $n$ to the corresponding partition $\wt{U}$ of the restriction $q^{\underline{n-1}}$ as indicated in the above cases. Using~($\dagger$) and the given properties of $q_{jn}$ for $j \in \wt{U}^1$ in each of the cases $1-4$, a straightforward verification shows that equation~(\ref{eq:conC}) is satisfied. Let us give the details of case 3 as illustration. Since in this case we have $i \in U^1_{v,\epsilon}$ for some $v \in \{2,\ldots,m\}$ and $\epsilon \in \{\pm 1\}$, and by construction $n \in U^1_{v,-\epsilon}$, to verify equation~(\ref{eq:conC}) we need to show that for $j \in \wt{U}^{k'}_{v',\epsilon'}$ we have
\[
q_{jn}= \left\{
\begin{array}{c l}
|\epsilon'-\epsilon|, & \text{if $v=v'$ (hence, $k'=1$)}, \\
k', & \text{if $v \neq v'$}.
\end{array} \right.
\]
Indeed, if $k'=2$ (hence, $v'=1 \neq v$), then $q_{jn}=2=k'$ by ($\dagger$). If $k'=1$ and $v' \neq v$, then $q_{ij}=1$, and since $q_{jn}+q_{ij}=2$ we have $q_{jn}=1=k'$. If $k'=1$, $v'=v$ and $\epsilon' = \epsilon$, then $q_{ij}=2$, and from $q_{jn}+q_{ij}=2$ it follows that $q_{jn}=0=|\epsilon'-\epsilon|$. These are all cases, since $\wt{U}^1_{v,-\epsilon}=\emptyset$ in case 3, as shown before. With this we complete the proof.
\end{proof}

\subsection{{\it Proof of Theorem B from Section \ref{sec:intro}.}}\label{subsec:proofAp}

Let $q:\ZZ^n\to\ZZ$ be a connected irreducible non-negative quadratic form Dynkin type $\CC$ as in Definition~\ref{D:tyc}. Permuting variables if necessary, we may assume that $q_1=2$. Apply Lemma~\ref{L:piv} to the form $q$ around vertex $1$ to get a rigid G-transformation $T$ such that the form $q'=q\circ T$ satisfies $q'_{1,i}>\delta_{1,i}$ for $i=1,\ldots,n$. By Lemma~\ref{L:techC}, there are $m\geq 2$ and a partition $U$ of $\{1,\ldots,n\}$ as in~(\ref{eq:parC}),
\[
U=\{U^2_{1,-1},U^1_{2,+1},U^1_{2,-1},\ldots,U^1_{m,+1},U^1_{m,-1}\},
\]
such that equation~(\ref{eq:conC}) holds for $U$ and $q'$. We consider a bidirected graph $\bdg'$ with $m$ vertices $\{1,\ldots,m\}$ and $n$ arrows $\{1,\ldots,n\}$ given by:
\begin{itemize}
 \item if $i \in U^2_{1,-1}$, then $i$ is a two-head loop at vertex $1$;
 \item if $i \in U^1_{v,+1}$ for some $v \in \{2,\ldots,m\}$, then $i$ is a directed arrow from vertex $v$ to vertex $1$;
 \item if $i \in U^1_{v,-1}$ for some $v \in \{2,\ldots,m\}$, then $i$ is a two-head arrow with end-vertices $v$ and $1$.
\end{itemize}
It is straightforward to verify that the incidence form $q_{\bdg'}$ satisfies equation~(\ref{eq:conC}) with partition $U$ (cf. Table~\ref{TA:des}), and therefore, $q'=q_{\bdg'}$. By Corollary~\ref{C:tra}, since $q'=q\circ T$ for a G-transformation $T$, then $q=q_{\bdg}$ for some bidirected graph $\bdg$ having at least one bidirected loop, cf.~Lemma~\ref{L:ful}.

Assume now that $q=q_{\bdg}$ with $\bdg$ having a bidirected loop $i_0 \in E(\bdg)$. Since $q$ is irreducible, then $|V(\bdg)| \geq 2$ (see Lemma~\ref{L:irr}(b)), and since $q$ is connected, then so is $\bdg$, and $n=|E(\bdg)| \geq 2$ since $\bdg$ contains a loop. By Lemma~\ref{L:ful}, $q$ is non-negative, we have $1 \leq q_i \leq 2$ for all $i \in \underline{n}$ (cf.~Lemma~\ref{L:irr}(a)), $q$ is non-unitary (since $q_{i_0}=2$) and $q$ is fully regular. Therefore, $q$ has Dynkin type $\CC$, cf.~Definition~\ref{D:tyc}. This shows the equivalence of conditions (i) and (ii) in Theorem B. Since $\bdg$ is unbalanced (because it contains a bidirected loop), then $\cyccond_{\bdg}=0$. Then the corank of $q$ is given by $\CRnk(q)=|E(\bdg)|-|V(\bdg)|$, by Corollary~\ref{cor:crk}. This completes the proof of Theorem B.\hfill$\Box$

\subsection{Classification and further properties}\label{subsec:DynCclass}

In this subsection we show that the methods used in the proof of Theorem B above can be extended to obtain a stronger result, namely, a classification of integral forms of Dynkin type $\CC$ up to G-equivalence, see Theorem \ref{T:chc}.

\begin{definition}\label{D:canC}
Consider the following family of bidirected graphs. For $r \geq 2$, $c_1 \geq 0$ and $c_2 \geq 0$, take
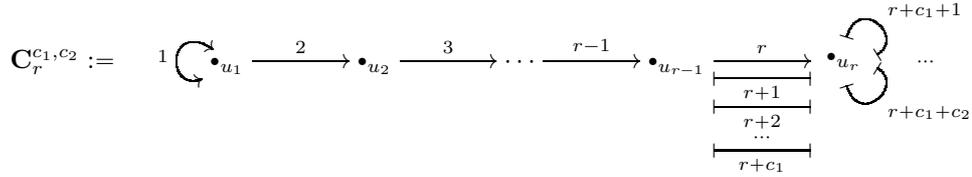
\[
\bdgC_r^{c_1,c_2} := \quad \xymatrix@C=3pc{ {\bulito}_{u_1} \ar@{<->}@(lu,ld)_-{1} \ar[r]^-{2} & {\bulito}_{u_2} \ar[r]^-{3} & \cdots \ar[r]^-{r-1} & {\bulito}_{u_{r-1}} \ar[r]^-{r} \ar@{|-|}@<-1.5ex>[r]_-{r+1} \ar@{|-|}@<-4ex>[r]_-{r+2} \ar@{|-|}@<-7.8ex>[r]_-{r+c_1}^-{\cdots} & *++[]{\bulito_{u_r}} \ar@{|-|}@(u,r)^(.5){r+c_1+1} \ar@{|-|}@(d,r)_(.5){r+c_1+c_2} \ar@{}[r]|-{\cdots} & {} }
\]
We denote by $\widehat{\mathcal{C}}^{(c_1,c_2)}_r:=\Delta(\bdgC_r^{c_1,c_2})$ the associated incidence bigraph, the {\bf canonical $(c_1,c_2)$-extension} of Dynkin bigraph $\mathcal{C}_r$ (cf.~Remark \ref{R:canAD}).
\end{definition}

\begin{remark}\label{rem:dl}
Given a Cox-regular form $q:\ZZ^n\to\ZZ$  we set
$
\dl(q):=\sum_{i=1}^n(q_i-1)\geq 0.
$
Note that $\dl(q)$ is a G-equivalence invariant (cf.~Corollary \ref{cor:Gweak}), it equals the number of dotted loops in $\Delta_q$ and $\dl(q)=0$ if and only if $q$ is unitary. Moreover, if $q$ has Dynkin type $\CC$, then $\dl(q)=|\{i\in\un :\,q_i=2\}|\geq 1$,  cf.~Definition \ref{D:tyc}\Ci-\Cii.
\end{remark}

\begin{lemma}\label{L:canC} Given $r \geq 2$, $c_1 \geq 0$ and $c_2 \geq 0$, the incidence form $q=q_\bdgC=q_{\Delta(\bdgC)}:\ZZ^{r+c_1+c_2}\to\ZZ$ of $\bdgC:=\bdgC_r^{c_1,c_2}$ satisfies the following:
\begin{enumerate}[label={\textnormal{(\alph*)}},topsep=3px,parsep=0px]
\item $q$ is a connected,  irreducible, non-negative integral quadratic form of Dynkin type $\CC$,
\item
$\widehat{\mathcal{C}}^{(0,0)}_r=\C_r,$ \ $\widehat{\mathcal{C}}^{(0,1)}_r=\widetilde{\C}_r$  \text{and}  $\widehat{\mathcal{C}}^{(1,0)}_r=\widetilde{\C\mathcal{D}}_r$,
 for $r\geq 2$,
\item we have $\dl(q)=c_2+1$ $($cf.~Remark \ref{rem:dl}$)$, \ $\Rnk(q)=r$ and $\crk(q)=c_1+c_2$,
\item $q_{\bdgC}\Gweak q_{\bdgC'}$ for $\bdgC' = \bdgC_{r'}^{c'_1,c'_2}$ if and only if $(r, c_1, c_2)=(r',c'_1,c'_2)$, that is, if $\bdgC= \bdgC'$.
 \end{enumerate}
\end{lemma}
\begin{proof} (a) Connectedness for $r=2$ (resp.~for $r\geq 3$) follows by a direct check (resp.~from Lemma \ref{L:irr}(c)). To show the remaining properties in (a) apply Lemma \ref{L:irr}(b) and Lemma \ref{L:ful} by observing that $|\verts{\bdgC}|\geq 2$ and that all loops in  $\bdgC$ are bidirected.
The claim (b) follows by a direct verification by applying Definition \ref{D:inc}, see also Table \ref{TA:des}.

To show the first formula in (c) observe that $\bdgC$ has precisely $c_2+1$ bidirected loops and apply Lemma \ref{L:ful}. The remaining claims in (c) follow from Corollary \ref{cor:crk} since  $|V(\bdgC)|=r$,   $|\edgs{\bdgC}|=r+c_1+c_2$ and  $\cyccond_\bdgC=0$ (note that $\bdgC$ has at least one bidirected loop so it is obviously unbalanced).

To show (d) assume that $q_\bdgC\Gweak q_{\bdgC'}$. Then obviously $q_{\bdgC}$ and $q_{\bdgC'}$ are $\ZZ$-equivalent hence $r=\Rnk(q_{\bdgC})=\Rnk(q_{\bdgC'})=r'$ by (c). Moreover, G-equivalence preserves the number of dotted loops $\dl(q)$, thus $c_2=c_2'$ by (c). This implies that $c_1=c_1'$, since clearly $r+c_1+c_2=r'+c'_1+c'_2$.
\end{proof}

We have the following generalization of Theorem \ref{T:wea}(a) for Dynkin type $\CC$.
\begin{theorem}\label{T:chc}
Let $q:\Z^n \to \Z$ be a connected  irreducible Cox-regular form with $n\geq 2$. Then the following conditions are equivalent:
\begin{enumerate}[label={\textnormal{(\roman*)}},topsep=4px,parsep=0px]
\item $q$ is non-negative of Dynkin type $\CC_r$,
\item $q\Gweak q_{\Delta}$ for the unique canonical $(c_1,c_2)$-extension $\Delta:=\widehat{\mathcal{C}}^{(c_1,c_2)}_r$ of $\mathcal{C}_r$.
\end{enumerate}
In this case, $c_1+c_2=\crk(q)$ and $c_2=\dl(q)-1\geq 0$.
\end{theorem}
\begin{proof}
 Recall that $q_{\Delta}$ has Dynkin type $\CC_r$, and $\CRnk(q_\Delta)=c_1+c_2$, $\dl(q_\Delta)=c_2+1$ by Lemma \ref{L:canC}. Therefore, if $q \sim_G q_{\Delta}$, then $q$ shares all these properties by Lemma \ref{lem:GweakDynC} and Corollary \ref{cor:Gweak}. Uniqueness of $\Delta$ follows from Lemma \ref{L:canC}(d).

For the converse implication, assume that $q$ has Dynkin type $\CC_r$ and let $c:=\crk(q)=n-r\geq 0$ and  $c_2:=\dl(q)-1$ ($c_2 \geq 0$, since $q$ is non-unitary).  As in the proof of Theorem B in \ref{subsec:proofAp}, we may find a G-transformation $T$ such that $q':=q \circ T'=q_{\bdg'}$ for some bidirected graph $\bdg'$ whose shape is determined by a partition $U$ of the set $\{1,\ldots,n\}$ as in~(\ref{eq:parC}). To be precise, the vertices of $\bdg'$ are $\verts{\bdg'}=\{1,\ldots,m\}$ and its arrows $\edgs{\bdg'}=\{1,\ldots,n\}$ are partitioned as follows:
\begin{eqnarray*}
U^2_{1,-1} & = & \{\text{two-head loops at vertex $1$}\}, \\
U^1_{v,+1} & = & \{\text{directed arrows from vertices $v$ to $1$}\}, \\
U^1_{v,-1} & = & \{\text{two-head arrows between vertices $v$ and $1$}\}.
\end{eqnarray*}
We need a couple of constructions to manipulate the bidirected graphs that appear in this way. Observe first that, after switching at vertex $v$ if necessary, we may assume that $|U^1_{v,+1}| \geq |U^1_{v,+1}|$ for $v=2,\ldots,m$. Moreover:
\begin{itemize}
 \item[(1)]  Assume that $j \in U^1_{v,\epsilon}$ for some $v \in \{2,\ldots,m\}$ and $\epsilon \in \{\pm 1\}$.  Then the transformation $\bdg' \mapsto (\bdg' \circ \mathcal{T}_{1,j}) \circ \mathcal{T}_j$ (cf. Definition~\ref{D:tra}) effectively ``moves arrow $j$ from $U^1_{v,\epsilon}$ to $U^1_{v,-\epsilon}$'', meaning that if $\wt{q}=(q'\circ T_{1,j})\circ T_j$ (cf. Lemma~\ref{L:tra}), then $\wt{q}$ satisfies the hypothesis of Lemma~\ref{L:techC}, and a corresponding partition $\wt{U}$ is given by
 \[
\wt{U}^{k'}_{v',\epsilon'}= \left\{
\begin{array}{l l}
U^1_{v,\epsilon}\setminus \{j\}, & \text{if $(k',v',\epsilon')=(1,v,\epsilon)$}, \\
U^1_{v,-\epsilon} \cup \{j\}, & \text{if $(k',v',\epsilon')=(1,v,-\epsilon)$}, \\
U^{k'}_{v',\epsilon'}, & \text{otherwise}.
\end{array} \right.
 \]
 \item[(2)] Assume that $i \in U^1_{v,+1}$ and $j \in U^1_{v,\epsilon}$ for some $v \in \{2,\ldots,m\}$ and $\epsilon \in \{\pm 1\}$, with $i \neq j$. Assume also that $i_0 \in U^1_{v',+1}$ with $v' \neq v$. If $\epsilon=1$ (resp. $\epsilon=-1$), then the transformation $\bdg' \mapsto ((\bdg' \circ \mathcal{T}_{i_0,j}) \circ \mathcal{T}_{i,j}) \circ \mathcal{T}_j$ (resp. the transformation $\bdg' \mapsto (\bdg' \circ \mathcal{T}_{i_0,j}) \circ \mathcal{T}_{i,j}$) ``moves arrow $j$ from $U^1_{v,\epsilon}$ to $U^1_{v',\epsilon}$''.
\end{itemize}
We follow some steps to complete the proof.

\medskip
\noindent \textit{Step~1.} As in the proof of Theorem B in \ref{subsec:proofAp}, we may find a G-transformation $T^1$ such that $q^1=q \circ T^1$ satisfies $(q^1_{1,i}) \geq \delta_{1,i}$ for $i \in \{1,\ldots,n\}$. In particular, there is a bidirected graph $\bdg'$ with $q^1=q_{\bdg'}$.

\medskip
\noindent \textit{Step~2.} Apply repeatedly the construction (1) above to find a (rigid) G-transformation $T^2$ such that $q^2:=q^1 \circ T^2=q_{\bdg''}$ where $\bdg''$ is obtained from $\bdg'$ by changing all its two-head arrows between vertices $v$ and $1$ to directed arrows from $v$ to $1$ (where $v\in \{2,\ldots,m\}$).

\medskip
\noindent \textit{Step~3.} Apply repeatedly the construction (2) above to find a (rigid) G-transformation $T^3$ such that $q^3:=q^2 \circ T^3=q_{\bdg'''}$ where $\bdg'''$ is obtained from $\bdg''$ by moving all but one of the parallel directed arrows from $v$ to $1$ (for $v \in \{3,\ldots,m\}$) to be directed arrows from vertex $2$ to vertex $1$.

\medskip
\noindent \textit{Step~4.} Permute the variables of $q^3$ with a permutation matrix $T^4$ so that, with the corresponding permutation of arrows of $\bdg'''$, we obtain a bidirected graph $\bdg_0$ as follows (left) with $c_2+1$ loops, $m\geq 2$ vertices and $c_1:=c-c_2\geq 0$ and (here $c=\CRnk(q)$ and $r=\Rnk(q)=m$ by Corollary \ref{cor:crk}):
\begin{equation}\label{eq:bdg0}
\bdg_0=\xymatrix@!0@R=12pt@C=14pt{&&&&& \bulito \ar[lllddd]_-{2}
\\ \\  && {} \ar@<-.7ex>@{<->}@(ru,lu)_-{1}
\\ {{}^{\cdots}} & {} \ar@<-1.2ex>@{<->}@(lu,l)_-{r+c_1+c_2} \ar@{<->}@(l,ld)_-{r+c_1+1} & *+++[]{\bulito} &&& \bulito \ar[lll]_-3
\\ &&&&& {} \ar@{}[d]|(.1){\cdots}
\\ &&&&& {}
\\ &&&&& \bulito \ar[llluuu]_(.3){r-1}
\\ && \bulito \ar@<2ex>[uuuu] \ar@<1.6ex>@{}[uuuu]|(.3)*[@]{{}_{r+c_1}} \ar@<1ex>@{}[uuuu]|-{\vdots} \ar[uuuu] \ar@<-2ex>[uuuu]_-r \ar@<-2ex>@{}[uuuu]|(.35)*[@]{{}_{r+1}} } \quad\qquad
\bdg'_0=\xymatrix@!0@R=12pt@C=14pt{&&&&& \bulito \ar[lllddd]_-{2}
\\ \\  && {} \ar@<-.7ex>@{<->}@(ru,lu)_-{1}
\\ {{}^{\cdots}} & {} \ar@<-1.2ex>@{<->}@(lu,l)_-{r+c_1+c_2} \ar@{<->}@(l,ld)_-{r+c_1+1} & *+++[]{\bulito} &&& \bulito \ar[lll]_-3
\\ &&&&& {} \ar@{}[d]|(.1){\cdots}
\\ &&&&& {}
\\ &&&&& \bulito \ar[llluuu]_(.3){r-1}
\\ && \bulito \ar@{|-|}@<2ex>[uuuu] \ar@<1.6ex>@{}[uuuu]|(.3)*[@]{{}_{r+c_1}} \ar@<1ex>@{}[uuuu]|-{\vdots} \ar@{|-|}[uuuu] \ar@<-2ex>[uuuu]_-r \ar@<-2ex>@{}[uuuu]|(.35)*[@]{{}_{r+1}} }
\end{equation}

\medskip
\noindent \textit{Step~5.} If $c_1>0$, then take recursively $\bdg_i:=\bdg_{i-1}\mathcal{T}_{1,r+i}$ for $i=1,\ldots,c_1$. Note that $\bdg'_0:=\bdg_{c_1}$ has the shape presented above (right). Take $T^5:=T_{1,r+1}T_{1,r+2} \cdots T_{1,r+c_1}$ the corresponding iterated Gabrielov transformation of $q_{\bdg_0}$, see Lemma \ref{L:tra}(a). In case $c_1=0$, take $\bdg'_0:=\bdg_0$ and $T^5:=\Id$.

\medskip
\noindent \textit{Step~6.} Take recursively $\bdg'_i:=\bdg'_{i-1}\mathcal{T}_{i+1,i}$ for $i=1,\ldots,r-1$ and $T^6:=T_{2,1}T_{3,2}\cdots T_{r,r-1}$ the corresponding iterated Gabrielov transformation of $q_{\bdg'_0}$. Note that $\bdg''_0:=\bdg'_{r-1}$ has the following shape,
\[
\bdg''_0=\xymatrix@C=3pc{ {\bulito} \ar@{<->}@(lu,ld)_-{1} \ar[r]^-{2} & {\bulito} \ar[r]^-{3} & \cdots \ar[r]^-{r-1} & {\bulito} \ar[r]^-{r} \ar@{|-|}@<-1.5ex>[r]_-{r+1}  \ar@{|-|}@<-5.3ex>[r]_-{r+c_1}^-{\cdots} & *++[]{\bulito} \ar@{<->}@(u,r)^(.5){r+c_1+1} \ar@{<->}@(d,r)_(.5){r+c_1+c_2} \ar@{}[r]|-{\cdots} & {} }
\]

\noindent \textit{Step~7.} If $c_2>0$, then take recursively $\bdg''_i:=\bdg''_{i-1}\mathcal{T}_{r+c_1+i}$ for $i=1,\ldots,c_2$, and let $T^7:=T_{r+c_1+1}\cdots T_{r+c_1+c_2}$ be the corresponding sign inversion (taking $T^7:=\Id$ if $c_2=0$), see Lemma \ref{L:tra}(b). Take $\bdg'''_0:=\bdg''_{c_2}$ (including the case $c_2=0$), and note that $\bdg'''_0=\bdgC_r^{c_1,c_2}$.

Finally, with the G-transformation $T:=T^1\cdots T^7$ we have $q \circ T=q_{\bdgC_r^{c_1,c_2}}$, which concludes the proof since $\Delta(\bdgC_r^{c_1,c_2})=\widehat{\mathcal{C}}^{(c_1,c_2)}_r$.
\end{proof}

As an immediate conclusion of the theorem above we have the following.
\begin{corollary}\label{cor:weakGforDynC}
Given two connected non-negative integral quadratic forms $q,q':\ZZ^n\to\ZZ$ of Dynkin type $\CC$, the following conditions are equivalent:
\begin{enumerate}[label={\textnormal{(\roman*)}},topsep=3px,parsep=0px]
\item $q\Gweak q'$,
\item $\crk(q)=\crk(q')$ and $\dl(q)=\dl(q')$.
\end{enumerate}
\end{corollary}

For convenience, we consider some transformations that do not preserve the loops in a bidirected graph (correspondingly, $\Z$-invertible linear mappings that do not preserve the loops of the bigraph associated to an integral quadratic form), see Table~\ref{TA:non}. For $n\geq 2$, a sign $\epsilon \in \{ \pm 1\}$ and $i\neq j$ in $\underline{n}$, consider the linear transformation $S_{i,j}^{\epsilon}:\Z^n \to \Z^n$ determined by $S_{i,j}^{\epsilon}(\bas_k)=\bas_k$ if $k \neq j$, and $S_{i,j}^{\epsilon}(\bas_j)=\bas_j-\epsilon \bas_i$. Clearly, $S^{\epsilon}_{i,j}$ is $\Z$-invertible, with inverse given by $S_{i,j}^{-\epsilon}$.

\begin{table}[ht]
\begin{center}
 \begin{tabular}{c c c}
 $\bdg$ &$\epsilon$& $\bdg\mathcal{S}^\epsilon_{i,j}$ \\
  \hline
$\xymatrix{ \bulito \ar@{<-}@<.5ex>[r]^-{i} \ar@{<-}@<-.5ex>[r]_-{j} & \bulito }$ & $+$ &
$\xymatrix{ \bulito \ar@{<-}[r]^-{i} \ar@(lu,ld)_-{j} & \bulito }$ \\[10pt]
$\xymatrix{ \bulito \ar@{<->}@(ru,rd)^-{i} \ar@{<->}@(lu,ld)_-{j} }$ & $+$ &
$\xymatrix{ \bulito \ar@{<->}@(ru,rd)^-{i} \ar@(lu,ld)_-{j} }$ \\[10pt]
$\xymatrix{ \bulito \ar@{<->}@(lu,ld)_-{j} \ar@{->}[r]^-{i} & \bulito }$ & $-$ &
$\xymatrix{ \bulito \ar@{->}@<.5ex>[r]^-{i} \ar@{<->}@<-.5ex>[r]_-{j} & \bulito }$
 \end{tabular}
\end{center}
\caption{Some auxiliary transformations of bidirected graphs. A new bidirected graph $\bdg\mathcal{S}^\epsilon_{i,j}$ is obtained from $\bdg$ by replacing the signed endpoints $\edges(j)$ of arrow $j$ as indicated in each case.}
\label{TA:non}
\end{table}

We have the following analog of Lemma \ref{L:tra}(a).
\begin{lemma}\label{L:non}
Let $\bdg$ be a bidirected graph having two arrows $i\neq j$ satisfying one of the three cases  in the left column of Table~\ref{TA:non}. Then $q_{\bdg} \circ S_{i,j}^\epsilon=q_{\bdg'}$ where $\bdg':=\bdg \mathcal{S}_{i,j}^\epsilon$ is obtained from $\bdg$ by replacing the signed endpoints of the arrow $j$ and $\epsilon$ is the corresponding sign as indicated in Table~\ref{TA:non}.
\end{lemma}
\begin{proof} 
By a similar straightforward calculation as in the proof of Lemma  \ref{L:tra}(a)  we check that $(S_{i,j}^\epsilon)^\tr I(\bdg)=I(\bdg')$  in all  three cases of Table~\ref{TA:non}.  This means that $q_\bdg\circ S_{i,j}^\epsilon=q_{\bdg'}$ by Lemma \ref{L:equivs}(a).
\end{proof}
Clearly,  $S_{i,j}^\epsilon$ defines a $\ZZ$-equivalence between $q_{\bdg}$ and $q_{\bdg'}$. However, $S_{i,j}^\epsilon$ is in general not G-equivalence since it may not preserve loops neither the connectedness of quadratic forms, cf.~Corollary \ref{cor:Gweak}.

Now we  can prove the following generalization of  Theorem \ref{T:wea}(c) for Dynkin type $\CC$. Recall that $\zeta^c$ denotes the zero quadratic form on $c$ variables.
\begin{theorem}\label{T:pr1}
Let $q:\Z^n \to \Z$ be a connected, irreducible and non-negative quadratic form of Dynkin type $\CC_r$ and corank $\CRnk(q)=c$, that is, $n=r+c$. Then there exist the following $\ZZ$-equivalences:
\begin{enumerate}[label={\textnormal{(\alph*)}},topsep=4px,parsep=0px]
 \item
 $q \weak q_{\mathcal{C}_r}\oplus \zeta^c$,
 \item $q \weak q_{\mathcal{\DD}_r}\oplus \zeta^c$ provided $r \geq 4$.
\end{enumerate}
\end{theorem}

\begin{proof} 
Let $c_2:=\dl(q)-1\geq 0$ and $c_1:=c-c_2$. Proceeding as in the  proof of Theorem~\ref{T:chc}, note that $c_1\geq 0$ and that there is a G-transformation $\widehat{T}P$ such that $q \circ (\widehat{T}P)=q_{\bdg_0}$, where $\bdg_0$ is the bidirected graph described in Step~4 of that proof.

\medskip
\noindent \textit{Step~1.} If $c_1>0$, then define recursively $\bdg_i:=\bdg_{i-1}\mathcal{S}^+_{r,r+i}$ for $i=1,\ldots,c_1$. Take the corresponding $\ZZ$-equivalence $S':=S^+_{r,r+1}\cdots S^+_{r,r+c_1}$ (cf.~Lemma \ref{L:non}) and $\bdg'_0:=\bdg_{c_1}$. In case $c_1=0$, take $S':=\Id$ and $\bdg'_0:=\bdg_0$.

\medskip
\noindent \textit{Step~2.} If $c_2>0$, then define recursively $\bdg'_i:=\bdg'_{i-1}\mathcal{S}^+_{1,r+c_1+i}$ for $i=1,\ldots,c_2$. Take $S'':=S^+_{1,r+c_1+1}\cdots S^+_{1,r+c_1+c_2}$ and $\bdg''_0:=\bdg'_{c_2}$. In case $c_2=0$ take $S''=\Id$ and $\bdg''_0=\bdg'_0$.

\medskip
\noindent \textit{Step~3.} Take recursively $\bdg''_i:=\bdg''_{i-1}\mathcal{T}_{i+1,i}$ for $i=1,\ldots,r-1$, and let $T''':=T_{2,1}T_{3,2}\cdots T_{r,r-1}$ be the corresponding iterated Gabrielov transformation of $q_{\bdg''_0}$. Then $\bdg''':=\bdg''_{r-1}$ has the following shape,
\[
\bdg'''=\xymatrix@C=3pc{ {\bulito} \ar@{<->}@(lu,ld)_-{1} \ar[r]^-{2} & {\bulito} \ar[r]^-{3} & \cdots \ar[r]^-{r-1} & {\bulito} \ar[r]^-{r}  & *++[]{\bulito} \ar@{<-}@(u,r)^(.5){r+1} \ar@{<-}@(d,r)_(.5){r+c_1+c_2} \ar@{}[r]|-{\cdots} & {} }
\]
Then by Lemmas \ref{L:tra}(a) and \ref{L:non} we have $q \circ T^a= q_{\bdg'''}$ for $T^a:=\widehat{T}PS'S''T'''$. Observe that $\bdg'''$ contains exactly $c=c_1+c_2$ directed loops, precisely those with labels $r+1,\ldots,r+c_1+c_2$. This means that $q_{\bdg'''}=q_{\C_r}\oplus\zeta^c$ by Lemmas \ref{L:canC}(b) and \ref{L:irr}(a), thus the claim $(a)$ holds.

Take now $\widehat{\bdg}:=\bdg'''\mathcal{S}^-_{2,1}$, and observe that if $r \geq 4$, then $q\circ T^b=q_{\widehat{\bdg}}=q_{\D_r}\oplus \zeta^c$, where $T^b:=T^aS^-_{2,1}$, cf.~Definition \ref{D:canAD}(b). This shows $(b)$, completing the proof.
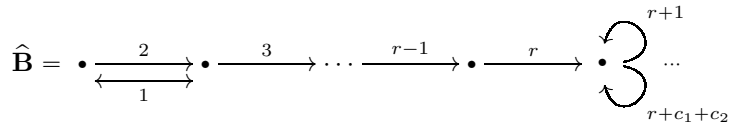
\[
\widehat{\bdg}=\xymatrix@C=3pc{ {\bulito} \ar@{<->}@<-1.5ex>[r]_-{1} \ar[r]^-{2} & {\bulito} \ar[r]^-{3} & \cdots \ar[r]^-{r-1} & {\bulito} \ar[r]^-{r}  & *++[]{\bulito} \ar@{<-}@(u,r)^(.5){r+1} \ar@{<-}@(d,r)_(.5){r+c_1+c_2} \ar@{}[r]|-{\cdots} & {} }
\]
\end{proof}

In particular, the proof of Theorem \ref{T:pr1} provides explicit $\ZZ$-equivalences $q_{\C_r}\weak q_{\D_r}$ for each $r\geq 4$, and $q_{\widetilde{\C}_r}\weak q_{\C_r}\oplus\xi\weak q_{\widetilde{\C\CD}_r}$  for each $r\geq 2$, cf.~Remark \ref{rem:weakDynkins}. More generally,
we have the following.

\begin{corollary}\label{cor:weakforDynC}
Given two connected non-negative integral quadratic forms $q,q':\ZZ^n\to\ZZ$ of Dynkin type $\CC$, the following two conditions are equivalent:
\begin{enumerate}[label={\textnormal{(\roman*)}},topsep=3px,parsep=0px]
\item $q\weak q'$,
\item $\crk(q)=\crk(q')$.
\end{enumerate}
In particular, the following $\ZZ$-equivalences hold:
\begin{enumerate}[label={\textnormal{(\alph*)}},topsep=4px,parsep=0px]
 \item $q_{\widehat{\mathcal{C}}^{(c_1,c_2)}_r}\weak q_{\widehat{\mathcal{C}}^{(c'_1,c'_2)}_r}$ for each $r\geq 2$ and $c_1,c_2,c_1',c_2'\geq 0$ such that $c_1+c_2=c_1'+c_2'$,
\item $q_{\widehat{\mathcal{C}}^{(c_1,c_2)}_r}\weak q_{\widehat{\D}^{(c)}_r}$ for each $r\geq 4$ and $c, c_1,c_2\geq 0$ such that $c=c_1+c_2$.
\end{enumerate}
\end{corollary}
\begin{proof} 
Equivalence of (i) and (ii) follows from Theorem \ref{T:pr1}(a) since $\crk(q)=\crk(q')$ implies $\Rnk(q)=\Rnk(q')$ ($q$ and $q'$ have the same number $n$ of indeterminates). In particular, the claim (a) holds by
Lemma \ref{L:canC}(c).

To show (b) combine Theorems \ref{T:pr1}(b) and \ref{T:wea}(c) to get $q_{\widehat{\mathcal{C}}^{(c_1,c_2)}_r}\weak  q_{\D_r}\oplus\xi^{c_1+c_2}\weak q_{\widehat{\D}^{(c_1+c_2)}_r}$.
\end{proof}

We finish this section with a useful fact which is an analog of the known property of non-negative unit forms (see \cite[Theorem 2(c)]{BP99}) and it has a surprisingly simple proof by our techniques.
Given a subset $X\subseteq \edgs{\bdg}$ of the set of arrows of a bidirected graph $\bdg$, we denote by $\bdg^X$ the bidirected subgraph of $\bdg$ obtained by removing all the arrows in $E(\bdg)\setminus X$. Note that $q_{\bdg^X}=q_{\bdg}^X$, where $q_{\bdg}^X$ is the restriction of $q_\bdg$ to $X$, cf.~\eqref{eq:restriction} and Lemma \ref{L:sub}. 

\begin{proposition}\label{prop:pr2}
Let $q:\Z^n \to \Z$ be a connected, irreducible and non-negative quadratic form of Dynkin type $\CC_r$ with $r\geq 2$. Then there exists a connected restriction $q'$ of $q$ such that $q'$ is a positive quadratic form of Dynkin type $\CC_r$.
\end{proposition}
\begin{proof}
By Theorem B, there is a connected bidirected graph $\bdg$ such that $q=q_{\bdg}$ with $\cyccond_\bdg=0$, at least one bidirected loop, say $i_0\in\edgs{\bdg}$, and  $r=|\verts{\bdg}|$. Since $q$ is connected and irreducible, $\bdg$ contains no directed loop.   Consider a spanning subtree $\bdg'$ of $\bdg$, let $X'=E(\bdg')$ be the set of arrows in $\bdg'$, and take $X:=X' \cup \{i_0\}$. Then $\bdg^X$ is connected, has exactly one bidirected loop and exactly $r \geq 2$ vertices. Then $q':=q^X$ is a connected irreducible restriction of $q$ (apply Lemma \ref{L:irr} and Table \ref{TA:des}) with Dynkin type $\CC_r$. Moreover, its corank is $\CRnk(q')=|E(\bdg^X)|-|V(\bdg^X)|=0$, thus $q'$ is a positive quadratic form. This completes the proof.
\end{proof}

\section{Roots of incidence quadratic forms}\label{sec:roots}
In this section we show that  walks in a bidirected graph  provide an adequate graphical model for certain roots of its incidence form. In particular,  we prove Theorem C in subsection \ref{subsec:proofB}.

Fix a bidirected graph $\bdg=(\verts{\bdg},\edgs{\bdg},\edges)$ with the set of all directed loops  $\shposloops:=\posloops{\bdg}\subseteq E:=\edgs{\bdg}$, cf.~Definition \ref{D:bdg}.    By $\extwalks(\bdg)$ we denote the set of all walks of $\bdg$, as defined in \eqref{eq:walk} in terms of the extended bidirected graph $\extbdg{\bdg}$. We define a mapping $\inc:\extwalks(\bdg)\to\ZZ^n$ for $n=|E|$ as follows.
For any walk $\wlk=(v_0,i_1,v_1,i_2,v_2,\ldots,v_{\ell-1},i_{\ell},v_{\ell})$ of ${\bdg}$ consider the vector $\inc(\wlk)$ in $\Z^n$ given by
\begin{equation}\label{eq:inc}
\inc(\wlk)=\inc_\bdg(\wlk)=\sum_{t=1}^{\ell}\sigma(\wlk^{[t-1]})d(v_{t-1},i_t)\bas_{i_t},
\end{equation}
where $\wlk^{[t]}=(v_0,i_1,v_1,i_2,v_2,\ldots,v_{t-1},i_{t},v_{t})$ for $t=1,\ldots,\ell$, and $\wlk^{[0]}=(v_0,v_0)$, and where we identify $\bas_{i}$ and $\bas_{i^{-1}}$ for $i \in \shposloops$.
Moreover, $d=d_\bdg$ denotes the function  $d_\bdg:\verts{\bdg} \times \edgs{\extbdg{\bdg}} \to \{\pm 1\}$ given by
\begin{equation*}
d(v,i) = \left\{
\begin{array}{l l}
d, & \text{if $(v,d) \in \edges(i)$ and $(v,-d) \notin \edges(i)$},\\
+1, & \text{if both $(v,d)$ and $(v,-d)$ belong to $\edges(i)$, and {$i \in \shposloops$}},\\
-1, & \text{if both $(v,d)$ and $(v,-d)$ belong to $\edges(i)$, and {$i \in (\shposloops)^{-1}$}},\\
0,& \text{otherwise.}
\end{array} \right.
\end{equation*}
For any walk $\wlk$ as above, define the \textbf{reversed walk} as the sequence $\omega^{-1}=(w_0,j_1,w_1,j_2,$ $w_2,\ldots,w_{\ell-1},j_{\ell},w_{\ell})$ where $w_t=v_{\ell-t}$ for $t=0,\ldots,\ell$, and $j_t=i_{\ell-t+1}$ if $i_{\ell-t+1}\in E\setminus \shposloops$, and $j_t=i_{\ell-t+1}^{-1}$ otherwise, for $t=1,\ldots,\ell$, where we take $(i^{-1})^{-1}=i$ for $i\in\shposloops$ (recall that the formal inverse is defined only for the elements of $E^{+\ell}$). {Note that $\sigma(\wlk)=\sigma(\wlk^{-1})$.} If $\wlk$ is closed, then we set $\wlk^0:=(v_0)$ and $\wlk^{k+1}:=\wlk^k\wlk$ for $k\geq 0$, and $\wlk^{-k}:=(\wlk^k)^{-1}$.

\begin{lemma}\label{L:inc}
Let $\bdg$ be a bidirected graph, and $\omega\in\extwalks(\bdg)$ a  walk  of $\bdg$.
\begin{enumerate}[label={\textnormal{(\alph*)}},topsep=3px,parsep=1px]
 \item If $\wlk$ has length one, say $\wlk=(v_0,i_1,v_1)$ for $i_1 \in \edgs{\extbdg{\bdg}}$, then $\inc(\wlk)=d(v_0,i_1)\bas_{i_1}$ and $\inc(\wlk^{-1})=-\sigma(\wlk)\inc(\wlk)$.
 \item If $\wlk=\wlk'\wlk''$ for walks $\wlk'$ and $\wlk''$, then $\inc(\wlk)=\inc(\wlk')+\sigma(\wlk')\inc(\wlk'')$.
 \item We have $\inc(\wlk^{-1})=-\sigma(\wlk)\inc(\wlk)$.
 \item If $\wlk'$ is a closed positive walk such that $\wlk\wlk'\wlk^{-1}$ is a walk, then $\inc(\wlk\wlk'\wlk^{-1})=\sigma(\wlk)\inc(\wlk')$.
 \item If $\wlk', \wlk''$ are walks such that $\wlk'\wlk\wlk^{-1}\wlk''$ is defined, then $\inc(\wlk'\wlk\wlk^{-1}\wlk'')=\inc(\wlk'\wlk'')$.
\item If $\wlk$ is a positive closed walk, then $\inc(\wlk^k)=k\inc(\wlk)$ for each $k\in\ZZ$.
\item If $\wlk$ is a negative closed walk, then $\inc(\wlk^k)=\inc(\wlk)$ for each odd $k\in\ZZ$ and
$\inc(\wlk^k)=0$ for each even $k\in\ZZ$.
\end{enumerate}
\end{lemma}
\begin{proof}
The first claim in (a) follows from the definitions of $\inc$ and $d(-,-)$, since $\sigma(\wlk^{[0]})=1$ for $\wlk^{[0]}$ is a trivial walk. Also notice that if $\wlk^{-1}=(v_0',i_1',v_1')$, then
\[
\inc(\wlk^{-1})=d(v_0',i_1')\bas_{i_1}=d(v_1,i_1')\bas_{i_1}=-\sigma(i_1)d(v_0,i_1)\bas_{i_1}=-\sigma(\wlk)\inc(\wlk)
\]
Observe that this equation holds even when $i_1$ is a directed loop.

Now, assume that $\wlk=(v_0,i_1,v_1,\ldots,v_{\ell-1},i_{\ell},v_{\ell})$. To show (b), take $\ell'$ and $\ell''$ the lengths of $\wlk'$ and $\wlk''$ (so that $\ell=\ell'+\ell''$). A direct calculation yields,
\begin{eqnarray*}
\inc(\wlk'\wlk'') & = & \sum_{t=1}^{\ell}\sigma(\wlk^{[t-1]})d(v_{t-1},i_t)\bas_{i_t} \\
& = & \inc(\wlk')+\sum_{t=\ell'+1}^{\ell'+\ell''}\sigma(\wlk^{[t-1]})d(v_{t-1},i_t)\bas_{i_t} \\
& = & \inc(\wlk')+\sigma(\wlk')\sum_{t=\ell'+1}^{\ell'+\ell''}\sigma((\wlk'')^{[t-\ell'-1]})d(v_{t-1},i_t)\bas_{i_t} \\
& = & \inc(\wlk')+\sigma(\wlk')\inc(\wlk'').
\end{eqnarray*}

Claim (c) follows by induction on the length of $\wlk$ (the base case shown in (a)). For a composite walk $\wlk'\wlk''$ with both $\wlk'$ and $\wlk''$ non trivial, using (b) we have
\begin{eqnarray*}
\inc[(\wlk'\wlk'')^{-1}] & = & \inc[(\wlk'')^{-1}(\wlk')^{-1}] =\inc[(\wlk'')^{-1}]+\sigma(\wlk'')\inc[(\wlk')^{-1}] \\
& = & -\sigma(\wlk'')\inc(\wlk'')-\sigma(\wlk'')\sigma(\wlk')\inc(\wlk') \\
& = & -\sigma(\wlk')\sigma(\wlk'')[\inc(\wlk')+\sigma(\wlk')\inc(\wlk'')] \\
& = & -\sigma(\wlk'\wlk'')\inc(\wlk'\wlk''),
\end{eqnarray*}
since clearly $\sigma(\wlk)=\sigma(\wlk^{-1})$ for any walk $\wlk$.

The remaining assertions (d)-(g) follow by a straightforward calculation from (b) and (c), and a simple induction in the cases (f)-(g).
\end{proof}

\begin{lemma}\label{L:Incinc}
Fix    a bidirected graph $\bdg$  and a walk $\wlk=(v_0,i_1,\ldots,i_{\ell},v_{\ell})\in\extwalks(\bdg)$. Then
\begin{enumerate}[label={\textnormal{(\alph*)}},topsep=4px,parsep=0px]
\item $I(\bdg)^{\tr}\inc(\wlk) = \bas_{v_0}-\sigma(\omega)\bas_{v_{\ell}}$,
 \item the walk $\wlk$ is closed and positive if and only if  $I(\bdg)^{\tr}\inc(\wlk)=0$,
 \item the walk $\wlk$ is closed and negative if and only if  $I(\bdg)^{\tr}\inc(\wlk)=2\bas_{v_0}$.
\end{enumerate}
\end{lemma}

\begin{proof} Applying \eqref{eq:inc} we have
\begin{eqnarray*}
\label{EqG:inc}
I(\bdg)^{\tr}\inc(\wlk) & = & \sum_{t=1}^{\ell}\sigma(\wlk^{[t-1]})d(v_{t-1},i_t)I(\bdg)^{\tr}\bas_{i_t} \nonumber\\
& = & \sum_{t=1}^{\ell}\sigma(\wlk^{[t-1]})d(v_{t-1},i_t)[d(v_{t-1},i_t)\bas_{v_{t-1}}+\xi(i_t)d(v_t,i_t)\bas_{v_{t}}] \nonumber\\
& = & \sum_{t=1}^{\ell}\sigma(\wlk^{[t-1]})[\bas_{v_{t-1}}-\sigma(i_t)\bas_{v_{t}}]
 \,=\,  \sum_{t=1}^{\ell}[\sigma(\wlk^{[t-1]})\bas_{v_{t-1}}-\sigma(\wlk^{[t]})\bas_{v_{t}}]
 \,=\,  \bas_{v_0}-\sigma(\wlk)\bas_{v_{\ell}},
\end{eqnarray*}
where $\xi(i):=-1$ for $i\in \posloops{\bdg}\sqcup(\posloops{\bdg})^{-1}$, and $\xi(i):=+1$ for all other arrows.

The claims $(b)$ and (c) follow directly from (a).
\end{proof}

The following is the converse of Lemma \ref{L:Incinc}(a) in some sense, and a crucial ingredient in the proof of Theorem~C.
\begin{proposition}\label{prop:incinc}
Let $\bdg$ be a connected bidirected graph with $n \geq 1$ arrows, and $x \in \Z^n$. If $I(\bdg)^{\tr}x=\bas_s + \T\bas_t$ for vertices $s,t \in V(\bdg)$ and a sign $\T \in \{\pm 1\}$, then there is a walk $\wlk$ from $s$ to $t$ such that $x=\inc(\wlk)$.
\end{proposition}
\begin{proof} 
We use the notation $\edges(i)=\{(v_0^i,\epsilon_0^i),(v_1^i,\epsilon_1^i)\}$ for an arrow $i \in \edgs{\bdg}$ and we identify $\edgs{\bdg}$ with the set $\un=\{1,\ldots,n\}$.
We proceed by induction on the length $|x|=\sum_{i=1}^n|x_i|$ of $x=[x_1,\ldots,x_n]^{\tr}$. If $x=0$, then we have $s=t$ and $\T=-1$, and the claim holds for the trivial walk $\wlk:=(s,s)$. If $|x|=1$, then $x=\eta\bas_i$ for some $i\in \edgs{\bdg}$ and $\eta\in\{\pm1\}$, so $\bas_s + \T\bas_t=I(\bdg)^{\tr}x=\eta\epsilon_0^i\bas_{v_0^i}+\eta\epsilon_1^i\bas_{v_1^i}$ and we may assume that $\{s,t\}=\{v_0^i,v_1^i\}$. Then applying \eqref{eq:inc} we verify that $\inc(\wlk)=x$ for $\wlk:=(s, i, t)$ or $\wlk:=(s, i^{-1}, s=t)$ in case $i\in\posloops{\bdg}$ and $\eta=-1$.

Now assume that $|x|>1$, and consider the following set
\[
E_{s}^x:=\{i \in \edgs{\bdg} \colon \text{$x_i \neq 0$ and $(s,\sgn(x_i)) \in \edges(i)$}\}.
\]
Observe first that if $I(\bdg)^{\tr}x \neq 0$, then $E_{s}^x \neq \emptyset$. Indeed, we have
\begin{equation}\label{EQ:proof5.4}
I(\bdg)^{\tr}x=\sum_{i=1}^nx_iI(\bdg)^{\tr}\bas_i=\sum_{i=1}^nx_i[\epsilon_0^i\bas_{v_0^i}+\epsilon_1^i\bas_{v_1^i}]=\bas_s+\T\bas_t.
\end{equation}
Therefore, the set $E_s^x$ must be non-empty. Fix an arbitrary $i \in E_s^x$, and take $\epsilon:=\sgn(x_{i})$. If $\edges(i)=\{(s,\epsilon),(s',\epsilon')\}$, then consider the vector $x':=x-\epsilon\bas_i$, which clearly satisfies $|x'|<|x|$ and
\begin{equation}\label{EQ:proof5.4-2}
I(\bdg)^{\tr}x'=I(\bdg)^{\tr}x-\epsilon I(\bdg)^{\tr}\bas_i=(\bas_s+\T\bas_t)-\epsilon(\epsilon\bas_s+\epsilon'\bas_{s'})=-\epsilon\epsilon'\bas_{s'}+\T\bas_t.
\end{equation}
Then $x'':=(-\epsilon\epsilon')x'$ satisfies $|x''|<|x|$ and $I(\bdg)^{\tr}x''=\bas_{s'}-\epsilon\epsilon'\T\bas_t$, and by induction hypothesis there is a walk $\wlk''$ from $s'$ to $t$ such that $x''=\inc(\wlk'')$. Take $\wlk':=(s,i,s')$ if $i\notin \posloops{\bdg}$ or $\wlk':=(s,i^\epsilon,s)$ if $i\in \posloops{\bdg}$, and let $\wlk:=\wlk'\wlk''$, which is a walk from $s$ to $t$. Using Lemma~\ref{L:inc}(b), we get
\[
\inc(\wlk)=\inc(\wlk'\wlk'')=\inc(\wlk')+\sigma(\wlk')\inc(\wlk'')=\epsilon\bas_i-\epsilon\epsilon'x''=\epsilon\bas_i+x'=x,
\]
since $\inc(\wlk')=\epsilon\bas_i$ and $\sigma(\wlk')=-\epsilon\epsilon'$. This shows the induction step when $I(\bdg)^{\tr}x \neq 0$.

In case $I(\bdg)^{\tr}x=0$, that is, when $s=t$ and $\T=-1$, there is a vertex $v \in V(\bdg)$ such that $E_v^x \neq \emptyset$. This also follows from the equation \eqref{EQ:proof5.4}, since $x \neq 0$. In this case  the equation \eqref{EQ:proof5.4-2} for a fixed $i \in E_v^x$ takes the form
\[
I(\bdg)^{\tr}(x-\epsilon\bas_i)=-\epsilon I(\bdg)^{\tr}\bas_i=-\bas_{v}-\epsilon\epsilon'\bas_{v'},
\]
where $\epsilon=\sgn(x_{i})$ and  $\edges(i)=\{(v,\epsilon),(v',\epsilon')\}$.
Proceeding the induction step similarly as above and using Lemma \ref{L:Incinc}(a)
 we get a positive closed  walk $\wlk'$ starting and ending at vertex $v$, and such that $\inc(\wlk')=x$.   Since $\bdg$ is connected, there is a walk $\widetilde{\wlk}$ from $s$ to $v$. If $\sigma(\widetilde{\wlk})=1$, then take $\wlk:=\widetilde{\wlk}\wlk'\widetilde{\wlk}^{-1}$, which satisfies $\inc(\wlk)=\sigma(\widetilde{\wlk})\inc(\wlk')=x$ by Lemma~\ref{L:inc}(d). If $\sigma(\widetilde{\wlk})=-1$, then take $\wlk:=\widetilde{\wlk}(\wlk')^{-1}\widetilde{\wlk}^{-1}$, which satisfies $\inc(\wlk)=\sigma(\widetilde{\wlk})\inc((\wlk')^{-1})=x$ by Lemma~\ref{L:inc}(c,d), since $\sigma(\wlk')=1$. This completes the proof.
\end{proof}

The above proposition inspires to consider the following notion. Fix a bidirected graph $\bdg$ with $m=|\verts{\bdg}|, n=|\edgs{\bdg}|\geq 1$ and the incidence matrix $I:=I(\bdg)\in\M_{n,m}(\Z)$. A (column) vector $x$ in $\Z^n$ together with a multi-set $\{(s,\S),(t,\T)\}$ with exactly two elements in $\{1,\ldots,m\}\times \{\pm 1\}$, is called an \textbf{incidence vector}  of $\bdg$ (or of $I$) if
\begin{equation}\label{eq:incvec}
I(\bdg)^{\tr}x=\S\bas_s+\T\bas_t,
\end{equation}
where $\bas_1,\ldots,\bas_m$ are the canonical basis vectors of $\Z^m$. 
Note that the multiset $\{(s,\S),(t,\T)\}$  is determined by $x$ unless $I^{\tr}x=0$. Observe that  all canonical vectors $\bas_1,\ldots,\bas_n$ of $\Z^n$ are incidence vectors of $I$, cf.~\eqref{eq:IofB}.

\begin{corollary}\label{cor:incinc}
Let $\bdg$ be a connected bidirected graph with $n\geq 1$ arrows. A vector $x$ in $\Z^n$ is an incidence vector of $I(\bdg)$ if and only if there is a walk $\wlk\in\extwalks(\bdg)$ of $\bdg$ such that
$
x=\pm \inc(\wlk).
$
\end{corollary}
\begin{proof}
The fact that $\inc(\wlk)$ and $-\inc(\wlk)$ are incidence vectors of $I(\bdg)$ follows from Lemma \ref{L:Incinc}(a), see \eqref{eq:incvec}.
To prove the converse assume that $I(\bdg)^{\tr}x=\S\bas_s+\T\bas_t$ for $\S,\T\in\{\pm1\}$ and $s, t\in\verts{\bdg}$. Then by Proposition \ref{prop:incinc} there is a walk $\wlk$ from $s$ to $t$ such that $\S x=\inc(\wlk)$.
\end{proof}

\subsection{{\it Proof of Theorem C from Section \ref{sec:intro}.}} \label{subsec:proofB} 
Fix a bidirected graph $\bdg$ with $m=|\verts{\bdg}|\geq 1$ and $n=|\edgs{\bdg}|\geq 1$. Let
\[\CRinc_\bdg:=\{\inc(\wlk) \colon \ \wlk\in\extwalks(\bdg)\}\,\cup\, \{-\inc(\wlk) \colon \ \wlk\in\extwalks(\bdg)\}\,\subseteq\, \Z^{n}\]
be the set of vectors as defined in the statement of Theorem C. The elements of $\CRinc_\bdg$ are called the {\bf incidence roots} of $\bdg$. By Corollary \ref{cor:incinc}, the incidence roots of $\bdg$  are precisely the  incidence vectors of $I(\bdg)$.

To show the inclusions \eqref{eq:B}, observe first that $q_\bdg(x)\in \{0,1\}$ if and only if $||I(\bdg)^{\tr}x||^{2}\in \{0,2\}$, that is, exactly when there are signs $\S,\T \in \{\pm 1\}$ and indices $s,t \in \{1,\ldots,m\}$ with
\[
I(\bdg)^{\tr}x=\S\bas_s +\T\bas_t,
\]
satisfying $\S=-\T$ when $s=t$, and where $s=t$ is arbitrary if $q_\bdg(x)=0$ (cf.~Definition~\ref{D:inc}). To show the second inclusion in  \eqref{eq:B} observe that in the remaining case, that is, with $\S=\T$ and $s=t$, we have $q_\bdg(x)=2$.

Now having the inclusions \eqref{eq:B}, the assertions (a) and (c) of Theorem C follow from Lemma \ref{L:Incinc}. The remaining assertions  (b) and (d) follow obviously from (a) and (c) (for (d) recall that $\bdg$ being balanced means  that there is no negative closed walk, that is, the set $\CRinc_\bdg\cap\CR_q(2)$ is empty, cf.~Proposition \ref{P:null}).
\hfill$\Box$

\medskip

\begin{example}\label{ex:onevertexbdgbasis}
Take a one-vertex bidirected graph $\bdg=\onevertexbdg^{p,s,t}$ as defined in
Remark \ref{R:loo}, for some $p+s+t= n\geq 1$ (here $\verts{\bdg}=\{u\}$). Assume that $q:=q_\bdg$ is connected, equivalently, that $\bdg$ has no directed loops, that is, $p=0$ (cf.~Lemma \ref{L:irr}(a)). In this case $\CRnk(q_\bdg)=n-1$ (see Example \ref{ex:onevertexbdgcrk}). Moreover, it follows that:
\begin{enumerate}[label={\textnormal{(\alph*)}},topsep=3px,parsep=0px]
\item
By applying  respective sign inversions we have $q_\bdg\Gweak q_{\bdg'}$ for $\bdg'=\onevertexbdg^{0,1,n-1}$, see Lemma \ref{L:tra}(b). It can be  verified that $q_{\bdg'}=2\hat{q}$ where $\hat{q}=q_{\widehat{\A}_1^{(n-1)}}=q_{\bdgA_1^{(n-1)}}$ is the canonical $(n-1)$-extension of the Dynkin graph $\AA_1$ as in Remark \ref{R:canAD}. In particular, $\hat{q}$ has Dynkin type $\Dyn(\hat{q})=\AA_1$.
\item By  Theorem C(a) in this case the radical $\CR_q(0)$ of $q=q_{\bdg'}$ consists of vectors $\pm \inc(\wlk)$ for all walks $\wlk$ of $\bdg'$ made up of even number of loops. In particular, if $\edgs{\bdg'}=\{i_1, \ldots, i_n\}$ and $i_1$ is the unique two-tail arrow, then $y_2:=\inc(u,i_1,u,i_2,u)=\bas_1+\bas_2$, $y_3:=\inc(u,i_1,u,i_3,u)=\bas_1+\bas_3$, \ldots, $y_n:=\inc(u,i_1,u,i_n,u)=\bas_1+\bas_n$ belong to $\CR_q(0)$. By the above description of $\CR_q(0)$  it is easy to check that the set $Y:=\{y_2, \ldots, y_n\}$ form a $\ZZ$-basis of $\CR_q(0)$. Moreover, $Y$ is a so-called $(i_2,\ldots, i_n)$-{\bf special $\ZZ$-basis} of $\CR_q(0)$ in the sense of \cite{dS16a}. 
\end{enumerate}
\end{example}

\begin{example}\label{ex:incroots} 
Take $\bdg$ and $\bdg'$ with $q:=q_\bdg=q_{\bdg'}=q_{\AA_3}$ and $\cyccond_\bdg=0$, $\cyccond_{\bdg'}=1$ as in Examples \ref{ex:twobdg} and \ref{ex:twobdgcrks}. Since $q$ is positive then $\CR_q(0)=\{0\}$; note that $\bdg$ and $\bdg'$ do not have non-trivial positive closed walks, compare with Theorem C(a). Observe that by Theorem C(b), $\bdg$ and $\bdg'$ yield two different graphical descriptions of the set
\[
\CR_q(1) =\{\,\pm[1, 0, 0], \pm[0, 1, 0], \pm[0, 0, 1], \pm[1, 1, 0], \pm[0, 1, 1],   \pm[1, 1, 1]\,\}
\]
in terms of open walks. For example,
\[
\inc_{\bdg}(u_3, 2, u_2, 3, u_1, 1, u_2) = -[1, 1, 1] = \inc_{\bdg'}(u_4, 3, u_3, 2, u_2, 1, u_1).
\]
Since $\cyccond_{\bdg'}=1$ then $\CRinc_{\bdg'}\cap\CR_q(2)=\emptyset$ by Theorem C(d). But in case of $\bdg$ it appears that
\[
\CRinc_{\bdg}\cap\CR_q(2)=\CR_q(2) =
\{\,   \pm[1, 0, 1], \pm[1, 0, -1], \pm[1, 2, 1]\,\}.
\]
For instance, take the closed walk $\wlk:=(u_3, 2, u_2, 1, u_1, 3, u_2, 2, u_3)$. Then we check that $\inc_{\bdg}(\wlk^k)=-[1,2,1]$ for $k\in\ZZ$ odd and $\inc_{\bdg}(\wlk^k)=0$ for $k$ even (cf.~Lemma \ref{L:inc}(g)).
\end{example}

\section{Applications, conclusions and plans}\label{sec:appl}

In this last section we exhibit some consequences of our main results in number theory, graph theory, representations of algebras and computer science, in part showing the potential of our techniques as well as outlining some of our future work.

\subsection{Positive incidence forms}\label{subsec:positive}

A walk $\wlk$ in a bidirected graph $\bdg$ is called  a {\bf reduced walk} if it is not of the form  $\wlk=\wlk'\wlk_1\wlk_1^{-1}\wlk''$ for any walks $\wlk_1,\wlk',\wlk''$ with $\wlk_1$ non-trivial. The set of all reduced walks of $\bdg$ is denoted by $\walksred(\bdg)\subseteq\walks(\bdg)$.

\begin{remark}\label{rem:reduced}
We can restrict the study of incidence roots to these induced by reduced walks. More precisely, by Lemma \ref{L:inc}(e) we have $\CRinc_\bdg=\{\pm\inc(\wlk)\,| \ \wlk\in\walksred(\bdg)\}$.
\end{remark}

The following may be viewed as an analog of a criterion for unit forms of Drozd \cite[1.1(2)]{cmR} and Happel \cite{Happel95}.
\begin{proposition}\label{prop:posfinite}
Let $q$ be a connected irreducible incidence form $q=q_\bdg$ of a  bidirected graph $\bdg$. Then the following conditions are equivalent:
\begin{enumerate}[label={\textnormal{(\alph*)}},topsep=3px,parsep=0px]
\item $q$ is positive,
\item $\CR_q(0)=\{0\}$,
\item the set $\CR_q(1)$ is finite,
\item the set $\CRinc_\bdg$ is finite.
\end{enumerate}
\end{proposition}
\begin{proof} Since $q$ is assumed  to be connected then so is $\bdg$ (cf.~Remark \ref{R:con}(b)).
So the assertion (d) implies that the sets  $\CR_q(0)$ and $\CR_q(1)$ are finite by Theorem C.
Assume that  the set $\CR_q(0)$ contains a non-zero element $h\in\CR_q(0)$. Then $q(ah)=a^2q(h)=0$ for each $a\in\ZZ$, what contradicts with the finiteness of $\CR_q(0)$. This proves  that (d) implies (b) and (c).

 To show that (b) (resp.~(c)) implies (a) assume to the contrary that $q$ is not positive. Then by Lemma \ref{L:ful} $q$ is non-negative with $\crk(q)\geq 1$. This means that $\CR_q(0)\neq \{0\}$. Moreover, note that  $\bas_{i_0}\in\CR_q(1)$ for some index $i_0\in\edgs{\bdg}$ with $q_{i_0}=1$ (it exists since $q$ is irreducible,  cf.~Lemma \ref{L:irr}(b)).  So  $q(ah+\bas_{i_0})=q(ah)+q(ah,\bas_{i_0})+q(e_{i_0})=1$ for each $a\in\ZZ$ and $h\in\CR_q(0)$ (cf.~\eqref{eq:rad}). This shows that the set $\CR_q(1)$ is infinite.

Now it remains to show the implication (a) $\Rightarrow$ (d). If $q$ is positive, that is, $\crk(q)=0$, then by Theorems A and B we have $0=|\edgs{\bdg}|-|\verts{\bdg}|+\cyccond_\bdg$. We consider two cases.

\medskip
\noindent \textit{Case 1:} $\cyccond_\bdg=1$. Then $|\verts{\bdg}|=|\edgs{\bdg}|+1$ which means that $\bdg$ is a tree. Hence, for any pair of vertices $s,t \in \verts{\bdg}$ there is exactly one reduced walk $\wlk_{s,t}$ from $s$ to $t$, which is trivial if and only if $s=t$. By Remark~\ref{rem:reduced} and Lemma \ref{L:inc}(c), we have
\[
\CRinc_{\bdg}=\{\inc(\wlk_{s,t})\}_{s<t \in \verts{\bdg}} \cup \{0\} \cup \{-\inc(\wlk_{s,t})\}_{s<t \in \verts{\bdg}}.
\]
In particular, the set $\CRinc_{\bdg}$ is finite. To be precise, if $\inc(\wlk_{s,t})=\epsilon \inc(\wlk_{s',t'})$ for some $\epsilon=\pm 1$ and $s<t$, $s'<t'$, then $s=s'$, $t=t'$ and $\epsilon=1$, see Lemma~\ref{L:Incinc}(a). This shows that $|\CRinc_{\bdg}|=n^2+n+1$ for $n:=|\edgs{\bdg}|$, and $q_{\bdg}(x)=1$ for any non-zero $x \in \CRinc_{\bdg}$. 

\medskip
\noindent \textit{Case 2:} $\cyccond_\bdg=0$. Then $|\verts{\bdg}|=|\edgs{\bdg}|$ (we call such graph a {\bf 1-tree}). In this case there exists a unique (up to rotation) closed walk $\wlk=(v_0,i_1,v_1,\ldots,v_{\ell-1}, i_{\ell},v_{\ell})$ with $\ell\geq 1$ and pairwise different vertices $v_1,\ldots, v_{\ell-1}$ and arrows $i_1,\ldots, i_\ell$. Moreover, $\wlk$ is negative since $\cyccond_\bdg=0$. For any vertex $s$ fix a walk $\wlk_s$ from $s$ to $v_0$ of minimal length (in particular, reduced). For any pair of vertices $s\leq t$ and any $k \in \Z$, take $\wlk_{s,t}^k:=\wlk_s\wlk^k\wlk_t^{-1}$. It is easy to check that for every reduced walk $\wt{\wlk}$ from $s$ to $t$ in $\bdg$, there is $k \in \Z$ such that $\inc(\wt{\wlk})=\inc(\wlk_{s,t}^k)$. Moreover, using Lemma~\ref{L:inc}(b),(g) we get
\begin{equation}\label{eq:summands}
\inc(\wlk^k_{s,t}) = \left\{
\begin{array}{l l}
\inc(\wlk_s)+\epsilon\inc(\wlk_t^{-1}), & \text{if $k$ is even}, \\
\inc(\wlk_s)+\epsilon(\inc(\wlk)-\inc(\wlk_t^{-1})), & \text{if $k$ is odd},
\end{array} \right.
\end{equation}
where $\epsilon:=\sigma(\wlk_s)$. Using Remark \ref{rem:reduced}, this shows that
\begin{equation}\label{eq:Rinc1tree}
\CRinc_{\bdg}=\{\inc(\wlk^k_{s,t})\}^{k \in \{1,\delta_{s,t}\}}_{s\leq t \in \verts{\bdg}} \cup \{0\} \cup \{-\inc(\wlk^k_{s,t})\}^{k \in \{1,\delta_{s,t}\}}_{s\leq t \in \verts{\bdg}}.
\end{equation}
Then the set $\CRinc_\bdg$ is finite, which finishes the proof of the remaining implication (a) $\Rightarrow$ (d). As in Case~1, one can show that the expression \eqref{eq:Rinc1tree}  is a disjoint partition of sets without repeated elements, and therefore $|\CRinc_{\bdg}|=2n^2+1$, having exactly $2n$ elements $x$ with $q_{\bdg}(x)=2$: those given by $\pm \inc(\wlk^1_{s,s})$ for some $s \in \verts{\bdg}$.
\end{proof}

As an immediate consequence of the proof  we obtain the following useful criterion (compare with \cite[Theorem 5.1]{tZ08}).

\begin{corollary}\label{cor:posfinite}
A connected irreducible incidence form $q=q_\bdg$ is positive if and only if $\bdg$ is a tree or an unbalanced 1-tree.
\end{corollary}

\subsection{Diophantine equations and universality}\label{subsec:universal}

Recall that a non-negative  integral quadratic form $q:\ZZ^n\to\ZZ$ is called universal if for any integer $d\geq 0$ there exists $x\in\ZZ^n$ such that $q(x)=d$ (cf.~Section \ref{sec:intro}).  Simson proved in  \cite[Proposition 4.1]{dS11a} and \cite[Theorems 3.1, 3.3]{dS13c} the following fact by a clever reduction to Lagrange's theorem and the results of Ramanujan \cite{Ram}.

\begin{theorem}[{{\cite{dS11a,dS13c}}}]\label{thm:Simuniv}
Let $q:\ZZ^n\to\ZZ$ be a non-negative connected unit form of  corank $c=\CRnk(q)$ equal 0 or 1. Then $q$ is universal, provided $n-c\geq 4$.
\end{theorem}

By applying our results we are able to generalize Simson's theorem to a larger class of quadratic forms containing many non-unitary ones.
\begin{theorem}\label{thm:univ}
Let $q:\ZZ^n\to\ZZ$ be a non-negative connected irreducible integral form of arbitrary corank $c=\CRnk(q)\geq 0$. Assume that $q$ is unitary or that $q$ has Dynkin type $\CC$. Then $q$ is universal, provided $n-c\geq 4$.
\end{theorem}

\begin{proof}
 Let $r:=\Rnk(q)=n-c$. Assume first that $q$ is unitary. Then $q\weak^T q_{\Delta}$ for $\Delta=\widehat{D}_r^{(c)}$,  $D_r \in \{\A_r,\D_r,\E_6,\E_7,$ $\E_8\}$ and some $T\in\GlnZ$ by Theorem \ref{T:wea}(b). Observe that since $r\geq 4$, then $\AA_4$ or $\DD_4$ is a full subbigraph of $\Delta$, cf.~Remark \ref{R:canAD}, Table \ref{TA:Dynkins} and \cite[Definition 2.2]{dS16a}. This means that $q_{\Delta}$ is universal since so are $q_{\AA_4}$ and $q_{\DD_4}$ by Theorem \ref{thm:Simuniv}, cf.~the definition of the restriction \eqref{eq:restriction}. So given an arbitrary integer $d\geq 0$, there exists $x\in\ZZ^n$ such that $d=q_{\Delta}(x)=q(T(x))$. This shows that also $q$ is universal.

At this point we should note that the proof of \cite[Theorem 3.1]{dS13c} for the case $\AA_4$ seems to contain a small gap. Namely, the author omits the argument that the reduction from $q_{\AA_4}(x)=d$ to Ramanujan's equation
$y^2_1+ 2y^2_2+ 2y^2_3+ 5y^2_4 = d'$ can be indeed done over the integers (see \cite[p.~21]{dS13c} for the details). We propose the following alternative argument to show that $q_{\AA_4}$ is universal.
By Bhargava-Hanke \lq\lq 290-Theorem''  \cite{BH290} it is enough to show that the quadratic form $q_{\AA_4}(x_1,x_2,x_3,x_4)=x_1^2+x_2^2+x_3^2+x_4^2-x_1x_2-x_2x_3-x_3x_4$  represents  29 \lq\lq critical'' integers $d$ from the range $1\leq d\leq 290$. We verify this by hand or a simple computer search (for example, two largest critical integers are
$203=5^2+1^2+14^2-5-14=q_{\AA_4}(0,5,1,14)$ and $290=q_{\AA_4}(1,0,0,17)$). Note that the same can be done for $q_{\DD_4}$.

In the remaining case, that is, when $q$ has Dynkin type $\CC_r$, we apply Theorem \ref{T:pr1}(b) to get the equivalence $q \weak q_{\mathcal{\DD}_r}\oplus \zeta^c$. On the other hand, by Theorem \ref{T:wea} we have
$q_{\widehat{\DD}_r^{(c)}}\weak q_{\mathcal{\DD}_r}\oplus \zeta^c\weak q$. This means that $q$ is universal since so is $q_{\widehat{\DD}_r^{(c)}}$ by the arguments for the unitary case.
\end{proof}

Note that Theorem \ref{thm:univ} embraces all connected irreducible incidence forms $q_\bdg$ (with $n-c\geq 4$) of bidirected graphs $\bdg$, see Theorems A and B.
It would be useful to have a direct method to solve  the Diophantine equation $q_\bdg(x)=d$ for arbitrary $d\geq 0$, at least to find a single solution $x\in\ZZ^n$. The following observations can be helpful. Note that by Theorem C the value of $q_\bdg(\inc(\wlk))$ for a walk $\wlk$ in $\bdg$ depends only on (interrelation between) the ending points of $\wlk$ and its sign $\sigma(\wlk)$. It appears that similar fact holds also for the associated bilinear form $q(x,y)=q(x+y)-q(x)-q(y)$ (cf.~Section \ref{sec:basics}).
\begin{lemma}\label{lem:univ.polar}
Let $q=q_\bdg$ be the incidence form of a bidirected graph $\bdg$. Fix two walks $\wlk=(v_0,i_1,v_1\ldots,i_{\ell},v_{\ell})$ and $\wlk'=(w_0,i_1,w_1\ldots,i_{\ell'},w_{\ell'})$ in $\bdg$ with signs $\sigma:=\sigma(\wlk)$, $\sigma':=\sigma(\wlk')$ and the corresponding incidence roots $x:=\inc(\wlk)$ and $y:=\inc(\wlk')$.  Then the value of $q(x,y)$  depends on the vertices $s:=v_0$, $t:=v_\ell$,
$s':=w_0$ and $t':=w_{\ell'}$ as follows:
\begin{enumerate}[label={\textnormal{(\alph*)}},topsep=3px,parsep=0px]
\item if $\wlk$ and $\wlk'$ are open walks, that is, if $s\neq t$ and $s'\neq t'$, then
\begin{equation}\label{tab:univ.polar}\begin{array}{rcccccccc}
q(x,y)=& 1+\sigma\sigma'&-\sigma-\sigma'&\sigma\sigma'&1&-\sigma&-\sigma'&0\\
\hline
   \text{if}& s=s'&s=t'&s\neq s'&s=s'    &s\neq t'    &s=t'    &s'\neq s\neq t'\\
  \text{and}& t=t'&t=s'&t=t'    &t\neq t'&t=s'        &t\neq s'&t' \neq t\neq s'\\
\end{array}
\end{equation}
\item if $\wlk$ is a closed walk, that is, if $s= t$,  then
\begin{equation}\label{tab:univ.polarclosed}\begin{array}{rcccccccc}
q(x,y)=& 1-\sigma & (\sigma-1)\sigma' &(1-\sigma)(1-\sigma')&0\\
\hline
   \text{if}& s=s'     &s=t'     &s=s'     &s\neq s'    \\
  \text{and}& s'\neq t'&s'\neq t'&s'=t'    &s\neq t'\\
\end{array}
\end{equation}
\end{enumerate}
\end{lemma} 
\begin{proof} Since $q(x,y)=x^\tr G_q y=x^\tr I(\bdg)I(\bdg)^\tr y$ (cf.~Remark \ref{R:con}(a)), by Lemma \ref{L:Incinc}(a) we have
 $$q(x,y)=[I(\bdg)^\tr\inc(\wlk)]^\tr[I(\bdg)^{\tr}\inc(\wlk')] = (\bas_{s}-\sigma\bas_{t})^\tr(\bas_{s'}-\sigma'\bas_{t'}).$$
 Now tables \eqref{tab:univ.polar} and \eqref{tab:univ.polarclosed} follow by a straightforward case by case check.
\end{proof}

\begin{corollary}\label{cor:univ.polar}
Let $q=q_\bdg$ be the incidence form of a connected bidirected graph $\bdg$. Then for any $k\geq 1$ and a collection  of 1-roots $x^{(1)},x^{(2)},\ldots, x^{(k)}\in\CR_q(1)$,  we have $q(\sum_{i=1}^kx^{(i)})\leq k^2=q(kx^{(1)})$.
\end{corollary}
\begin{proof}
Fix $x^{(1)},\ldots, x^{(k)}\in\CR_q(1)$. By Theorem C(b) there exist open walks $\wlk^{(i)}$ in $\bdg$ and signs $\epsilon_i\in\{\pm1\}$ such that $x^{(i)}=\epsilon_i\inc(\wlk^{(i)})$ for each $i=1,\ldots,k$. Then
\begin{equation}\label{eq:cor.univ}
q(\sum\limits_{i=1}^kx^{(i)})=\sum\limits_{i=1}^kq(x^{(i)})+\sum\limits_{i<j}^kq(x^{(i)},x^{(j)})\leq k + 2{{k}\choose{2}}=k^2
\end{equation}
since $q(x^{(i)},x^{(j)})\leq 2$ by Lemma \ref{lem:univ.polar}. The final equality is obvious: $q(kx^{(1)})=k^2q(x^{(1)})=k^2$.
\end{proof}

Lemma \ref{lem:univ.polar} and the proof above may lead to an efficient  method to find solutions of the equations $q_\bdg(x)=d$ for some $d$. Namely, experimental results show that for a proper choice of incidence roots
$x^{(1)}=\inc(\wlk^{(1)}),\ldots,$ $x^{(k)}=\inc(\wlk^{(k)})\in\CR_q(1)$, for any $l$
 we may represent almost all integers $\leq l^2$ by taking $q(\sum_{t=1}^{l}x^{(i_t)})$ for some $1\leq i_1,\ldots,i_{l}\leq k$. We illustrate the method on the following example.

 \begin{example}\label{eq:univ} 
 Take the incidence form $q=q_\bdg:\ZZ^4\to\ZZ$ of the bidirected graph $\bdg:=\bdgC_4^{0,0}$ with the incidence bigraph $\Delta=\Delta(\bdg)=\mathcal{C}_4$:
 \[
\bdg =  \xymatrix@C=1.6pc{ {\bulito}_{u_1} \ar@{<->}@(lu,ld)_{1} \ar[r]^-{2} & {\bulito}_{u_2} \ar[r]^-{3} & {\bulito}_{u_3} \ar[r]^-{4} & {\bulito}_{u_{4}} }\qquad \Delta=\ \
\begin{minipage}{4.8cm}
			{\scriptsize
				\begin{tikzpicture}[auto]
				\node (q1) at (0  , 0  ) {$1$};
				\node (q2) at (1, 0  ) {$2$};
				\node (q3) at (2, 0  ) {$3$};
				\node (q4) at (3, 0  ) {$4$};
				\foreach \x/\y in {1/2, 2/1}
				\draw[-] (q\x) edge [bend left=8] node{}(q\y);
				\foreach \x/\y in {2/3, 3/4}
				\draw[-] (q\x) edge node{}(q\y);
				\draw[dashed] ([yshift=4pt,xshift=4pt]q1) arc (-50:240:2.2mm);
				\end{tikzpicture}
		} \end{minipage}
\]
 (cf.~Definition \ref{D:canC}). Then $q$ is positive of Dynkin type $\CC_4$ hence it is universal by Theorem \ref{thm:univ}, cf.~Lemma \ref{L:canC}. We are going to find chosen solutions of the Diophantine equation $q(x)=d$, that is,
 $$
 2x_{{1}}^{2}+x_{{2}}^{2}+x_{{3
}}^{2}+x_{{4}}^{2}-2x_{{1}}x_{{2}}-x_{{2}}x_{{3}}-x_{{3}}x_{{4}}=d,
 $$
 for each $d\leq 16=4^2$. Inspired by Lemma \ref{lem:univ.polar} we choose four walks in $\bdg$ with common starting point: $\wlk^{(1)}=(u_1,2,u_2)$, $\wlk^{(2)}=(u_1,2,u_2,3,u_3)$, $\wlk^{(3)}=(u_1,2,u_2,3,u_3,4,u_4)$ and $\wlk^{(4)}=(u_1,1,u_1,2,u_2)$. We set $x^{(i)}:=\inc(\wlk^{(i)})$ for $i=1,2,3,4$. Then $$x^{(1)}=[0,1,0,0]^\tr,\ \, x^{(2)}=[0,1,1,0]^\tr,\ \, x^{(3)}=[0,1,1,1]^\tr, \ \,x^{(4)}=[-1,-1,0,0]^\tr$$ (see \eqref{eq:inc}). Moreover, by the table \eqref{tab:univ.polar} (or by a direct check) it follows that $q(x^{(i)},x^{(i)})=2q(x^{(i)})=2$ for $1\leq i\leq 4$, $q(x^{(1)},x^{(4)})=0$, and $q(x^{(i)},x^{(j)})=1$ for the remaining pairs $(i,j)$. Then with the help of the formula \eqref{eq:cor.univ} we deduce that:
 $$\begin{array}{ll}q(x^{(1)} ) = q(0, 1, 0, 0) = 1&q(x^{(1)}+x^{(1)}+x^{(1)} ) = q(0, 3, 0, 0) = 9\\
q(x^{(1)}+x^{(4)} ) = q(-1, 0, 0, 0) = 2&q(x^{(1)}+x^{(1)}+x^{(1)}+x^{(4)} ) = q(-1, 2, 0, 0) = 10\\
q(x^{(1)}+x^{(2)} ) = q(0, 2, 1, 0) = 3&q(x^{(1)}+x^{(1)}+x^{(2)}+x^{(3)} ) = q(0, 4, 2, 1) = 11\\
q(x^{(1)}+x^{(1)} ) = q(0, 2, 0, 0) = 4&q(x^{(1)}+x^{(1)}+x^{(2)}+x^{(2)} ) = q(0, 4, 2, 0) = 12\\
q(x^{(1)}+x^{(1)}+x^{(4)} ) = q(-1, 1, 0, 0) = 5&q(x^{(1)}+x^{(1)}+x^{(1)}+x^{(2)} ) = q(0, 4, 1, 0) = 13\\
q(x^{(1)}+x^{(2)}+x^{(3)} ) = q(0, 3, 2, 1) = 6&q(x^{(1)}+ x^{(1)}+ x^{(2)}+ x^{(3)}+ x^{(4)} ) = q(-1, 3, 2, 1) = 14\\
q(x^{(1)}+x^{(1)}+x^{(2)} ) = q(0, 3, 1, 0) = 7&q(x^{(1)}+ x^{(1)}+ x^{(1)}+ x^{(2)} +x^{(4)} ) = q(-1, 3, 1, 0) = 15\\
q(x^{(1)}+x^{(1)}+x^{(4)}+x^{(4)} ) = q(-2, 0, 0, 0) = 8&q(x^{(1)}+ x^{(1)}+ x^{(1)} +x^{(1)} ) = q(0, 4, 0, 0) = 16.
\end{array}$$
As a side effect we obtain the following interesting application. Take the matrix
$$T:=\left[ \begin {array}{rrrr} -1&0&0&1\\ -1&-1&0&2\\ 0&0&0&2\\ 0&0&-1&1\end {array}
 \right].$$
 We check directly that $q\circ T=q_{\rm Lag}$, where $q_{\rm Lag}(x_1,x_2,x_3,x_4)=\sum_{i=1}^4x_i^2$ is the Lagrange form. Note that this is not a $\ZZ$-equivalence since ${\rm det}(T)=2$. However, we obtain a quick method to present integers as a sum of squares of 4 half-integers. For instance,
 $$
 \begin{array}{l}
  14 = q(-1, 3, 2, 1)=q_{\rm Lag}(T^{-1}[-1, 3, 2, 1]^\tr)=q_{\rm Lag}(2, -3, 0, 1)=2^2+3^2+0^2+1^2\\
15 = q(-1, 3, 1, 0)=q_{\rm Lag}(T^{-1}[-1, 3, 1, 0]^\tr)= q_{\rm Lag}(\frac{3}{2}, -\frac{7}{2}, \frac{1}{2}, \frac{1}{2})=(\frac{3}{2})^2+(\frac{7}{2})^2+(\frac{1}{2})^2+(\frac{1}{2})^2. \\
\end{array}
 $$
 \end{example}

\subsection{Reflections and root systems}\label{subsec:rs}

In this section we show that reflections through (non-radical) incidence roots of a bidirected graph $\bdg$ can be realized as certain switching of $\bdg$ (Lemma~\ref{L:reflection}). This is a simple graph theoretical counterpart of the isometry of reflections, which has some interesting consequences. For instance, we show that in the finite case studied in~\ref{subsec:positive}, the set $\CRinc_{\bdg}\setminus\{0\}$ is a finite root system in the sense Bourbaki, cf. \cite[9.2]{jeH72} and~\cite{CGSS76,Kac80} (Proposition~\ref{P:rootSys}).

Recall that given an integral quadratic form $q:\ZZ^n\to\ZZ$ and a vector $x \in \Z^n$ with $q(x)\neq 0$, the (real) \textbf{reflection} $\refl_x=\refl_x^q:\R^n \to \R^n$ at $x$ with respect to $q$ is the linear mapping  given by (see for instance~\cite{MZ22})
\[
\refl_x(y)=y-\frac{2q(y,x)}{q(x,x)}x, \qquad \text{for $y \in \R^n$.}
\]
Note that if $q$ is Cox-regular, then the reflections $\refl_{\bas_i}$ through canonical vectors $\bas_i$ are defined over the integer numbers, that is, simple reflection for Cox-regular quadratic forms are integral, cf.~\cite[Remark~4.4]{MZ22}.

\begin{lemma}\label{L:reflection}
Let $\bdg$ be a connected bidirected graph with $m$ vertices and $n$ arrows. For any incidence root $x$ of $\bdg$ with $q_{\bdg}(x)\neq 0$, the reflection $\refl_x:=\refl_x^{q_{\bdg}}$ restricts to an isomorphism $\Z^n \to \Z^n$, also denoted by $\refl_x$. Moreover, there is an orthogonal  matrix $O^x\in\MM_m(\ZZ)$ such that
\[
I(\bdg)^{\tr}\refl_x=O^xI(\bdg)^{\tr}.
\]
\end{lemma}
\begin{proof}
Take $q:=q_{\bdg}$ and consider a walk $\wlk$ such that $\inc(\wlk)=\pm x$. Take an arbitrary $x' \in \CRinc_{\bdg}$ and a walk $\wlk'$ with $\inc(\wlk')=\pm x'$. If $\wlk$ is an open walk, then by Theorem C(b) we have $q(x,x)=2q(x)=2$, and thus $2q(x,x')/q(x,x)$ is an integer. If $\wlk$ is a closed walk, since $q(x)\neq 0$, then $\wlk$ is  negative  and $q(x,x)=4$, see Theorem C. Using Lemma~\ref{lem:univ.polar}(b), note that $q(x,x')$ is even, hence $2q(x,x')/q(x,x)$ is also an integer. Since the canonical vectors belong to $\CRinc_{\bdg}$ and $\refl_x(\refl_x(y))=y$, this shows that $\refl_x$ restricts to a $\ZZ$-isomorphism of $\Z^n$.

Take $\alpha:=I(\bdg)^{\tr}x=\S\bas_s+\T\bas_t\neq 0$ (cf.~Lemma \ref{L:Incinc}), and define $O^x:=\left(\Id_m-\frac{2\alpha \alpha^{\tr}}{\alpha^{\tr}\alpha}\right)$. Verify directly that $O^x(O^x)^{\tr}=(O^x)^2=\Id_m$, and that for any vector $y \in \Z^n$, taking $\beta:=I(\bdg)^{\tr}y$ we get
\[
I(\bdg)^{\tr}\refl_x(y)=\beta-\frac{2\beta^{\tr}\alpha}{\alpha^{\tr}\alpha}\alpha=\beta-\frac{2\alpha \alpha^{\tr}}{\alpha^{\tr}\alpha}\beta=O^x\beta=O^xI(\bdg)^{\tr}y.
\]
This shows the second claim of the lemma.
\end{proof}

In other words, reflections through (non-radical) incidence roots are realized as switching of bidirected graphs, cf.~Lemma \ref{L:tra}(d). This property potentially sheds new light in the study of reflections on root systems and related algebraic structures.

\begin{proposition}\label{P:rootSys}
Let $\bdg$ be a bidirected graph which is either a tree with $n\geq 1$ arrows or an unbalanced $1$-tree with $n \geq 2$ arrows. Then $\CRinc_{\bdg}\setminus\{0\}$ is an irreducible finite root system $($in the sense of \cite{jeH72}$)$ in the Euclidean space $(\mathbb{R}^n, q(-,-))$  of Dynkin type $\A_n$ or $\CC_n$, respectively.
\end{proposition}
\begin{proof} 
Take $q:=q_{\bdg}$. By Corollary~\ref{cor:posfinite}, $q$ is a positive incidence form, and $\CRinc_{\bdg}$ is a finite set by Proposition~\ref{prop:posfinite}. Since the canonical basis is contained in $\CRinc_{\bdg}$, then $\roots:=\CRinc_{\bdg}\setminus\{0\}$ is a finite set of non-zero vectors that generates $\mathbb{R}^n$. By Theorem C we have $q_{\bdg}(x)\in \{1,2\}$ for any $x \in \roots$. In particular, if $ax \in \roots$ for some $a \in \Z$, then $a=\pm 1$. That $2q(x,y)/q(x,x)$ is an integer for any $x,y \in \roots$ was shown in the proof of Lemma~\ref{L:reflection}.  Therefore, to verify that $\roots$ is a finite root system as in~\cite[9.2]{jeH72}, it remains to show that $\roots$ is closed under reflection at any of its elements. This is consequence of Lemma~\ref{L:reflection} as follows: if $x,y \in \roots$, then by Corollary~\ref{cor:incinc} we have $I(\bdg)^{\tr}y=\S\bas_s+\T\bas_t$ for some signs $\S,\T$ and vertices $s,t$. Let $O^x$ be the integer orthogonal matrix given in Lemma~\ref{L:reflection}, and take signs $\S',\T'$ and vertices $s',t'$ such that $O^x(\S\bas_s)=\S'\bas_{s'}$ and $O^x(\T\bas_t)=\T'\bas_{t'}$ (cf. Lemma~\ref{L:ortSimson}). Then
\[
I(\bdg)^{\tr}\refl_x(y)=O^xI(\bdg)^{\tr}y=O^x(\S\bas_s+\T\bas_t)=\S'\bas_{s'}+\T'\bas_{t'},
\]
and again by Corollary~\ref{cor:incinc} we have $\refl_x(y) \in \CRinc_{\bdg}$. Since $\refl_x$ is an isomorphism, 
 the reflected vector $\refl_x(y)$ is non-zero, that is, $\refl_x(y) \in \roots$.

To show that the root system $\roots$ is irreducible, consider two incidence roots $x,x' \in \roots$ such that $q(x,x')=0$. Take reduced walks $\wlk$ from vertex $s$ to $t$, and $\wlk'$ from vertex $s'$ to $t'$, such that $\inc(\wlk)=\pm x$ and $\inc(\wlk')=\pm x'$. By Lemma~\ref{lem:univ.polar}, since $\bdg$ has no non-trivial positive reduced closed walks, we have either $\{s,t\} \cap \{s',t'\}=\emptyset$, or $\{s,t\} = \{s',t'\}$ and $\sigma(\wlk)\sigma(\wlk')=-1$. In the first case, since $\bdg$ is connected, there is a walk $\wlk''$ from vertex $s$ to $s'$. By Lemma~\ref{lem:univ.polar}, then $x'':=\inc(\wlk'') \in \roots$ satisfies $q(x,x'')\neq 0$ and $q(x',x'') \neq 0$. In the second case, one of the walks $\wlk'':=\wlk \wlk'$ or $\wlk'':=\wlk (\wlk')^{-1}$ is defined, and therefore closed and negative. In particular, $x'':=\inc(\wlk'')$ belongs to $\roots$ and satisfies $q(x,x'')\neq 0$ and $q(x',x'') \neq 0$ also by Lemma~\ref{lem:univ.polar}. This shows that $\roots$ is irreducible.

To determine the Dynkin type of $\roots$, assume first that $\bdg$ is a tree. By the above, and as indicated in the proof of Proposition~\ref{prop:posfinite} (Case~1), $\roots$ is an irreducible finite root system of rank $n$ with $n^2+n$ elements $x$, all of them satisfying $q(x)=1$, see Theorem C(b,d). In particular, $\roots$ has Dynkin type $\A_n$ (cf.~\cite[12.1-2]{jeH72}). Assume now that $\bdg$ is a $1$-tree with at least $2$ arrows. By the proof of Proposition~\ref{prop:posfinite} (Case~2), $\roots$ is an irreducible finite root system of rank $n$ with a total of $2n^2$ roots, exactly $2n$ of them satisfying $q_{\bdg}(x)=2$. Then $\roots$ has Dynkin type $\CC_n$ (cf.~\cite[12.1-2]{jeH72}).
\end{proof}

\begin{remark}\label{rem:rootsysD} 
If $q=q_\bdg$ is an incidence form of a (connected) unbalanced 1-tree $\bdg$ with $n:=|\verts{\bdg}|=|\edgs{\bdg}|\geq 2$, then $q$ is positive and Theorems A and B imply the following (see also Lemmas \ref{L:ful} and \ref{L:irr}). If $\bdg$ has a bidirected loop, then $q$ is connected of Dynkin type $\CC_n$. In case $\bdg$ has no loops, if $n=2$ or $n=3$, resp.~$n\geq 4$, then $q_\bdg=q_{\AA_1}\oplus q_{\AA_1}$ or $q_\bdg$ is connected of Dynkin type $\AA_3$, resp.~$\DD_n$. In all these cases the set $\CRinc_{\bdg}\setminus\{0\}$ is an irreducible root system of type $\CC_n$ by Proposition \ref{P:rootSys}. However, by applying the formula \eqref{eq:Rinc1tree} and the results of \cite{MM19,MZ22} (cf.~Theorem \ref{T:pr1}) one can show that for $n\geq 4$ the set $\CR_q(1)\subsetneq \CRinc_{\bdg}\setminus\{0\}$ has $2(n^2-n)$ elements and it is an irreducible root system of Dynkin type $\DD_n$.
\end{remark}
In general, there is much to say about root systems, Weyl groups, reflections and Coxeter transformations in the incidence setting, which we leave for the future work. For some special cases, we refer the reader to~\cite{MZ22,jaJ2020b,jaJ2020c}.

\subsection{Whitney's problem on line graphs}\label{subsec:W}

Given an integral quadratic form $q:\ZZ^n\to\ZZ$ its Gram matrix has the shape
$G_q={\bf Ad}(\Delta)+2\Id=-{\bf Ad}(-\Delta)+2\Id$, where $\Delta=\Delta_q$ is the associated bigraph,  ${\bf Ad}(\Delta)$ its (signed) adjacency
matrix, and $-\Delta$ is the bigraph obtained from $\Delta$ by inverting all signs, cf.~\cite{tZ08}, Section \ref{sec:basics} and the footnote on p.~\pageref{p:convsigns}. In particular, $q$ is non-negative if and only if the (real) eigenvalues of ${\bf Ad}(-\Delta)$ are bounded by 2. Moreover, in case $q=q_\bdg$ for a bidirected graph $\bdg$, the bigraph $-\Delta$ coincides with the {\em line $($signed$)$ graph} of $\bdg$ in the sense of \cite{tZ08}. These simple observations provide a bridge between our results and the study of graphs, their spectra, and related root systems in \cite{CGSS76, tZ08,BH12}.
Note that usually the results in the classical spectral graph theory refer to (unsigned) graphs or  simple signed or bidirected graphs. We consider a more general setting: we study signed and bidirected graphs  admitting multi-edges and loops. The proper understanding of the consequences of the above bridge requires an  extensive study. Here we focus on the following interesting aspects.

We say that two bidirected graphs $\bdg$ and $\bdg'$ are {\bf switching equivalent} if $\bdg'$ is a switching $\bdg'=\bdg^O$ of $\bdg$ for some orthogonal integer matrix $O$ (cf.~Definition \ref{D:tra}(d) and \cite{tZ08}). We write $\bdg\switch \bdg'$ (note that $\switch$ is indeed an equivalence relation on bidirected graphs). Recall that then $q_\bdg=q_{\bdg'}$, see Lemma \ref{L:tra}(d). On the other hand, Example \ref{ex:twobdg} shows that there exist $\bdg$ and $\bdg'$ which are not switching equivalent but $q_\bdg=q_{\bdg'}$. This motivates to state the following natural problem, extending a result on line graphs by Whitney dating back to 1932~\cite{hW32} (see~\cite{hC21} and \cite{tZ08,BH12}).

\begin{problem}\label{prob:switching}
 Given an incidence quadratic form $q:\ZZ^n\to\ZZ$, describe all switching equivalence classes of bidirected graphs $\bdg$ such that $q_\bdg=q$. In other words, describe all switching classes of bidirected graphs $\bdg$ having the same line graph $-\Delta(\bdg)$.
 \end{problem}

Using Proposition~\ref{P:null} and one of the main results of~\cite{jaJ2018}, here we give a partial solution to Problem~\ref{prob:switching} as follows. Recall that, as general convention, bidirected graphs have no isolated vertices.

\begin{lemma}\label{L:swiBal}
Let $\bdg$ and $\bdg'$ be balanced connected loop-less bidirected graphs with $q_{\bdg}=q_{\bdg'}$. Then $\bdg$ and $\bdg'$ are switching equivalent.
\end{lemma}
\begin{proof}
By Proposition~\ref{P:null}(b), the balanced bidirected graphs $\bdg$ and $\bdg'$ are switching equivalent to quivers $\overrightarrow{\bdg}$ and $\overrightarrow{\bdg}'$, respectively. Using Lemma~\ref{L:tra}(d) we have $q_{\overrightarrow{\bdg}}=q_{\overrightarrow{\bdg}'}$, and by Theorem~\ref{T:mainA} we get $\overrightarrow{\bdg} \switch \overrightarrow{\bdg}'$. Then $\bdg \switch \bdg'$, as wanted.
\end{proof}

Let $\bdg$ be a loop-less bidirected graph with $n$ arrows and $m$ vertices. Recall that we identify the set of arrows of $\bdg$ with $\edgs{\bdg}=\un$, and its set of vertices with the labels $\verts{\bdg}=\{u_1,\ldots,u_m\}$, with respective natural orders. With this identification, the collection $\bdg(q):=\{\text{$\bdg$ a bidirected graph} \colon q_{\bdg}=q\}$ is a finite set for any quadratic form $q$. 
By Lemma~\ref{L:tra}(d), the set $\bdg(q)$ is closed under switching. Let $\swi(q)$ denote the number of switching classes of bidirected graphs $\bdg$ such that $q=q_{\bdg}$ (that is, $\swi(q)=|\bdg(q)_{/\switch}|$). 

\begin{corollary}\label{C:swiAr}
If $q$ is a connected non-negative unit form of Dynkin type $\Dyn(q)=\A_r$ for some $r \geq 4$, then $\swi(q)=1$.
\end{corollary}
\begin{proof}
For such $q$, by Theorem A we have $\swi(q)>0$. Let $\bdg$ and $\bdg'$ be bidirected graphs such that $q=q_{\bdg}=q_{\bdg'}$.  By Theorem A (cf.~Lemma~\ref{L:ful} and Remark~\ref{R:con}(b)), both $\bdg$ and $\bdg'$ are loop-less and connected. Since $r \geq 4$, by Lemma~\ref{L:typeA} both $\bdg$ and $\bdg'$ are balanced. Therefore, Lemma~\ref{L:swiBal} implies that $\bdg\switch \bdg'$, that is, $\swi(q)=1$.
\end{proof}

The problem for incidence forms $q$ not satisfying the assumptions of Corollary \ref{C:swiAr} seems to be more difficult. For instance, by Example  \ref{ex:twobdg} we have $\swi(q_{\AA_3})\geq 2$ (one can show that actually $\swi(q_{\AA_3})= 2$). We believe that our techniques allow to solve  Problem \ref{prob:switching} in full generality.

\subsection{Euler form of a gentle algebra}\label{subsec:gentle}

 In this subsection we freely use basic notions and  facts of representation theory of finite-dimensional algebras, including path algebras of (bound) quivers. For all the details we refer to \cite{ASS06}.

 Gentle algebras were introduced in \cite{AS87} in the context of study of algebras derived equivalent to hereditary algebras of type $\widetilde{\AA}$. Recall that a finite-dimensional algebra $\La$  over a field $\kk$ is called a {\bf gentle algebra} if it is a quotient $\La=\kk \quiv/\ideal$ of the path algebra $\kk \quiv$ of a connected quiver $\quiv$ modulo a two-sided ideal $\ideal\subseteq \kk \quiv$ such that: the indegree and the outdegree of every vertex in $\quiv$ is at most 2, the ideal $\ideal$ is generated by paths in $\kk \quiv$ of length 2, and for each arrow $\alpha\in \quiv$, there is at most one arrow $\beta$ such that $\alpha\beta\in\ideal$, there is at
most one arrow $\beta'$ such that $\beta'\alpha\in\ideal$, there is at most one (composable) arrow $\gamma$ such that $\alpha\gamma\notin\ideal$, and there is at
most one (composable)  arrow $\gamma'$ such that $\gamma'\alpha\notin\ideal$. Although this quite technical nature gentle algebras  appear in many different contexts, and in recent years they attracted special attention for their deep interrelations with symplectic geometry and mathematical physics, see \cite{LP20, OPS18} and references therein.

Fix a gentle algebra $\La=\kk \quiv/\ideal$. For any vertex $i$ in $\quiv$ we denote by $\mathbbm{1}_i$ the trivial path in $i$ and by $P_i=\mathbbm{1}_i\La$ the corresponding indecomposable projective (right) $\La$-module, cf.~\cite[III.2]{ASS06}. From now we assume that $\La$ has finite global dimension. In particular this means that $\quiv$ has no loops (cf.~\cite{gB17}) and that the Cartan matrix $\Car_\La\in\MM_n(\ZZ)$ of $\La$, whose columns $\Car_\La\bas_i=\pp_i:=\dim(P_i)$ are the dimension vectors of $P_i$'s, is $\ZZ$-invertible, where $n\geq 1$ is the number of vertices in $\quiv$ (see \cite[III.3]{ASS06}). Moreover,  the {\bf Euler form} $q_\La:\ZZ^n\to\ZZ$ of $\La$ given by
\begin{equation}\label{eq:eulf}
q_\La(x):=x^\tr(\Car_\La^{-1})^\tr x
\end{equation}
 is a well-defined integral quadratic form (cf.~Example \ref{ex:gentle}).

\def\fp{\phi}

Following \cite{AG08}, a  path in $\quiv$ not contained in the ideal  $\ideal$ is called a \textbf{permitted path}, and such paths which are maximal are called \textbf{non-trivial permitted threads}.  A path $c=\alpha_1\alpha_2\cdots \alpha_\ell$ such that $\alpha_{t}\alpha_{t+1} \in \ideal$ for each $t=1\ldots \ell-1$ is called a \textbf{forbidden path}, and such paths which are maximal are called \textbf{non-trivial forbidden threads}. If $i$ is a vertex in $\quiv$ with at most one arrow $\alpha$ ending at $i$, and at most one arrow $\beta$ starting at $i$, then the trivial path $\mathbbm{1}_i$ is a called a \textbf{trivial permitted thread} (resp.~\textbf{trivial forbidden thread}) if $\alpha\beta \notin \ideal$ (resp.~$\alpha\beta \in \ideal$), see \cite{AG08} for more details.
We denote by $\perm$ (resp.~$\forb$) the set of all permitted  (resp.~forbidden) threads. By the definition of a gentle algebra it follows that there is a well-defined bijection $\fp:\forb\to\perm$ such that $\theta\in\forb$ and $\eta:=\fp(\theta)$ have the same starting vertex and if $\theta\in\forb$ is non-trivial, then  $\eta$ is either trivial or the starting arrows of $\theta$ and $\eta$ differ (see \cite[2.2]{gB17} for the details).
Given a path $\wlk=\alpha_1\cdots \alpha_\ell$ in $\quiv$ with arrows $\alpha_t:i_{t-1} \to i_t$, for $t=1,\ldots \ell$, 
 we consider the vectors
$
\lceil \wlk \rceil :=\sum_{t=0}^\ell\bas_{i_t}$ and $\lfloor \wlk \rfloor:=\sum_{t=0}^\ell(-1)^t\bas_{i_t}
$ in $\ZZ^n$.

\begin{lemma}\label{lem:forbperm}
For any forbidden thread $\theta\in\forb$  we have
 $
 \Car_\La\lfloor \theta \rfloor=\lceil \fp(\theta) \rceil.
 $
\end{lemma} 

\begin{proof} Take $\theta=\alpha_1\cdots \alpha_\ell\in\forb$ with $\alpha_t:i_{t-1} \to i_t$, for $t=1,\ldots \ell$, and $\eta:=\fp(\theta)\in\perm$. Since $\eta$ does not belong to the ideal $\ideal$, we can define an indecomposable $\La$-module $M_\eta$ with $\dim(M_\eta)=\lceil \eta \rceil$ (so-called string module, see~\cite{BR87}). By the definitions of a gentle algebra and of forbidden (resp.~permitted) threads we verify that the following sequence of $\La$-modules and $\La$-homomorphisms

\begin{equation}\label{eq:resolution}
\xymatrix{0 \ar[r]& P_{i_{\ell}} \ar[r]^-{\hat{\alpha}_{\ell}} & P_{i_{\ell-1}} \ar[r]^-{\hat{\alpha}_{\ell-1}} & \cdots \ar[r]^-{\hat{\alpha}_2} & P_{i_1} \ar[r]^-{\hat{\alpha}_1} & P_{i_0}\ar[r]^-{p}&M_\eta\ar[r]&0,}
\end{equation}
is a projective resolution of $M_\eta$, where $\hat{\alpha}_t$ denotes the homomorphism given by respective composition with the arrow $\alpha_t$ and $p$ is the projective cover of $M_\eta$,
cf.~\cite[I.5]{ASS06}, and \cite[Lemma 2.4]{gB17} and its proof. Thus, since the sequence \eqref{eq:resolution} is exact, we have
$$\lceil \eta \rceil = \dim(\stringmod{\eta})=\sum_{t=0}^\ell(-1)^t\dim(P_{i_t})=\sum_{t=0}^\ell(-1)^t\Car_\La\bas_{i_t}=\Car_\La\lfloor \theta \rfloor.$$
\end{proof}

\begin{theorem}\label{thm:gentle}
Let $\La=\kk \quiv/\ideal$ be a gentle algebra of finite global dimension. Then there exists a bidirected graph $\bdg$ such that the Euler form $q_\La=q_\bdg$ is the incidence form of $\bdg$. In particular, $q_\La$ is non-negative, and if $q_\La$ is connected, then $q_\La$ has Dynkin type $\AA$, $\DD$ or $\CC$.
\end{theorem}
\begin{proof}
If $\La=\kk$, then $q_\La=q_{\AA_1}=q_{\bdgA_1^0}$. So we may assume that $\quiv$ has $n\geq  2$ vertices. We fix a numbering of the elements in $\forb=\{\theta_1,\ldots,\theta_m\}$ and $\perm=\{\eta_1,\ldots,\eta_m\}$ such that $\eta_u=\fp(\theta_u)$ for $u=1,\ldots,m$. We define a matrix $I=I_\La\in\MM_{n\times m}(\ZZ)$ by setting its columns to be
\begin{equation}\label{eq:IofLa}
I_\La\bas_u=\lfloor \theta_u \rfloor
\end{equation} for $u=1,\ldots,m$. Since $\La$ has finite global dimension it is known that every vertex of $\quiv$ appears exactly two times among forbidden threads, see \cite[2.2]{AG08}. Thus every transposed row of $I$ has the shape $I^{\tr}\bas^{(n)}_i=\epsilon\bas^{(m)}_{u}+\epsilon'\bas^{(m)}_{u'}$ for some $u,u'\in\underline{m}$ and $\epsilon, \epsilon'\in\{\pm1\}$. In particular, there exists a bidirected graph $\bdg$ with $|\edgs{\bdg}|=n$ and $|\verts{\bdg}|=m$ such that $I=I(\bdg)$ is its incidence matrix, cf.~\eqref{eq:IofB}. Note that $\bdg$ is determined uniquely (up to isomorphism depending on numbering of $\forb$) by $\La$ unless $I$ has zero rows. 

Now we show that $q_\La=q_\bdg$. First recall that the rows of the Cartan matrix  $\Car_\La^\tr\bas_i=\qq_i:=\dim(Q_i)$ are the dimension vectors of the indecomposable injective $\La$-modules $Q_i={\bf Hom}_K(\La\mathbbm{1}_i,K)$ corresponding to vertices $i$ of $\quiv$, see~\cite[III.3]{ASS06}. Moreover, given a vertex $i$ in $\quiv$, since indegree and outdegree of $i$ is at most 2 we infer that
\begin{equation}\label{eq:ppqq}
 (\Car_\La+\Car_\La^{\tr})\bas_i=\pp_i+\qq_i=\lceil \eta_{s_i}\rceil + \lceil \eta_{t_i} \rceil
\end{equation}
for  indices $s_i,t_i\in\underline{m}$ (maybe equal) such that permitted threads $\eta_{s_i}, \eta_{t_i}\in\perm$ contain $i$, cf.~\cite[III.2]{ASS06} (note that every vertex of $\quiv$ appears exactly two times also among permitted threads, see \cite[2.2]{AG08}). 
In particular,  matrix $\hI\in\MM_{n\times m}(\ZZ)$ defined by
$\hI\bas^{(m)}_u=\lceil \eta_u \rceil$ for $u=1,\ldots,m$, has (transposed) rows of the form $\hI^{\tr}\bas^{(n)}_i=\bas^{(m)}_{s_i}+\bas^{(m)}_{t_i}$ for $i\in\un$. Note that $\Car_\La I=\hI$ by Lemma \ref{lem:forbperm}. Combining these observations with \eqref{eq:ppqq} and \eqref{eq:IofLa} we get:
\begin{equation*}\label{eq:CChIhI}
(\Car+\Car^{\tr})\bas_i=\lceil \eta_{s_i}\rceil + \lceil \eta_{t_i} \rceil=
\Car(\lfloor \theta_{s_i}\rfloor + \lfloor \theta_{t_i} \rfloor)
=\Car I (\bas_{s_i}+\bas_{t_i})
=\Car I \hI^\tr\bas_i
=\hI \hI^\tr\bas_i
\end{equation*}
for any $i$, where $\Car=\Car_\La$. Therefore $\Car+\Car^{\tr}=\hI \hI^\tr$ and in consequence,
$$G_{q_\La}=\Car^{-\tr}+\Car^{-1}=\Car^{-1}(\Car+\Car^\tr)\Car^{-\tr}
=\Car^{-1}(\hI \hI^\tr)\Car^{-\tr}
=\Car^{-1}\Car I (\Car I)^\tr\Car^{-\tr}=II^\tr,
$$
which means that $q_\La=q_\bdg$, cf.~\eqref{eq:eulf}, Definition \ref{D:inc} and Remark \ref{R:con}(a).

Since $q_\La$ is an incidence form then the last claim follows from Lemma \ref{L:ful} and Theorems A and B.
\end{proof}

Observe that by the above theorem and Theorem C all $0$-roots, 1-roots and some $2$-roots of the Euler form $q_\La=q_\bdg$ of a gentle algebra $\La$ are described by walks in $\bdg$. Moreover, the rank, corank and the Dynkin type of $q_\La$ can be read off from $\bdg$ by Theorems A and B. As we mentioned in Section \ref{sec:intro}, the properties like the non-negativity, the Dynkin type of $q_\La$ and its roots are of interest for the structural analysis in the representation theory of $\La$ (cf.~\cite{ASS06, cmR,Kac80,BPS11}). We observed that the associated bidirected graph $\bdg$ 
encodes more useful data on $\La$, including  derived invariants like the Coxeter polynomial, cf.~\cite{H88,JAP.AM2015}. In particular,  Theorem \ref{thm:gentle} implies that the Coxeter polynomial of $\La$ is cyclotomic, that is, its Mahler measure is 1, see \cite{Sato} and \cite{BP99,JAP.AM2015}.
By using our techniques we hope to obtain a graphical model for gentle algebras which may be an  alternative (in some aspects simpler) to the widely applied combinatorial invariant of Avella-Alaminos and Geiss \cite{AG08} and to the geometric model of  \cite{OPS18,LP20}.
We plan to investigate these problems in a subsequent paper, also for the infinite global dimension case.

\begin{example}\label{ex:gentle} Let $\La=\kk \quiv/\ideal$, where $\quiv$ is the  quiver
$
\quiv:\, \xymatrix{{}_1 \ar@/^7pt/[r]^-{\alpha} & {}_2 \ar@/^7pt/[l]^-{\beta} }
$
and $\ideal$ is the two-sided ideal generated by the path $\alpha\beta$.  Then $\La$ is  a gentle algebra and we have
$$\Car_{\La}=\left[{\begin{array}{rr}1&1\\1&2  \end{array}}\right],\quad  \Car_\La^{-\tr}=
\left[{\begin{array}{rr}\!2\!&\!-1\!\\\!-1\!&\!1 \! \end{array}}\right],\ \ \text{and}\ \
 q_\Lambda(x)=x^\tr \Car_\La^{-\tr}x=2x_1^2+x_2^2-2x_1x_2.$$
 \vspace{-0.1cm}
 Note that the  bigraph of $q_\La$ has the shape $\Delta_{q_\La}=\begin{minipage}{4.2em}
{\scriptsize
				\begin{tikzpicture}[auto]
				\node (q1) at (0  , 0  ) {1};
				\node (q2) at (1, 0  ) {2};
				\foreach \x/\y in {1/2, 2/1}
				\draw[-] (q\x) edge [bend left=8] node{}(q\y);
				\foreach \x in {1}
				\draw[dashed] ([yshift=4pt,xshift=5pt]q\x) arc (-50:240:2.7mm);
				\end{tikzpicture}
		}\end{minipage} = \mathcal{C}_2$ so $q_\La$ is positive and $\Dyn(q_\La)=\CC_2$ (cf.~Lemma \ref{L:cla}). We have two forbidden threads: $\theta_1:=\alpha\beta$ and the trivial one $\theta_2:=\mathbbm{1}_2$. So the associated matrix $I=I_\La\in\MM_2(\ZZ)$ (see \eqref{eq:IofLa}) and the bidirected graph $\bdg$ with $I(\bdg)=I$ have the shape
$$I=\left[{\begin{array}{rr}2&0\\-1&1  \end{array}}\right],\qquad  \bdg=\xymatrix@C=3pc{ {\bulito}_{\theta_1} \ar@{|-|}@(lu,ld)_-{1}  & {\bulito}_{\theta_2} \ar[l]^-{2} }$$
We check directly that $\Car_\La^{-1}+(\Car_{\La}^{-1})^\tr=II^\tr$, that is, $q_\La=q_\bdg$. Note that  $\bdg$ is switching equivalent to $\bdgC_{2}^{0,0}$.
\end{example}
\subsection{Computational aspects}\label{subsec:comp}

We emphasize that most of the proofs of our main results as well as the corresponding auxiliary facts are constructive and the arguments can be easily transformed into algorithmic procedures. We avoid presenting the explicit pseudocodes not to extend our discussion, but restrict to the following hints which can also be viewed as a brief summary of the computational part of the paper:
\begin{alg}\label{alg:1}
\noindent
{\bf Input:} a connected irreducible non-negative integral quadratic form $q:\Z^n \to \Z$,


\noindent
{\bf Output:} a bidirected graph $\bdg$ such that $q=q_\bdg$ or {\bf false} if $q$ is not an incidence form.


\begin{enumerate}[label={\textnormal{\arabic*.}},topsep=3px,parsep=-1px]
\item If $q$ is unitary, then apply the inflation algorithm \cite[Algorithm 3.18]{SZ17} (cf.~Theorem \ref{T:wea}(a)) to obtain an explicit $G$-transformation $T$ such that $q\circ T=q_{\widehat{D}_r^{(c)}}$ with $D_r=\Dyn(q)$. If $D_r=\EE_r$, then return {\bf false}.
\item Otherwise, that is, if $D_r=\A_r$ (resp.~$D_r=\DD_r$), then apply Lemma \ref{L:tra} (cf.~Corollary \ref{C:tra}) to $T^{-1}$ and $\bdg:=\bdgA_r^c$ (resp.~$\bdg:=\bdgD_r^c$)  to obtain $\bdg\mathcal{T}$ with $q=q_{\bdg\mathcal{T}}$ (cf.~the proof of \lq\lq(i) implies (ii)'' of Theorem A in \ref{subsec:proofA}). Return $\bdg\mathcal{T}$. 
\item If $q$ is not unitary, then check whether or not $q$ satisfies Definition~\ref{D:tyc}. If not, then return {\bf false}. Otherwise follow the proof of Theorem B in \ref{subsec:proofAp} to find a G-transformation $T$ such that $q\circ T$ satisfies the hypothesis of Lemma~\ref{L:techC} (recall that $T$ is composition of a permutation and a rigid G-transformation given by Lemma~\ref{L:piv}), and to constuct a bidirected graph $\bdg'$ with $q\circ T=q_{\bdg'}$. Finally, apply Lemma \ref{L:tra} to $T^{-1}$ and $\bdg'$ to obtain $\bdg=\bdg'\mathcal{T}$ with $q=q_\bdg$. Return $\bdg$.
\end{enumerate}
\end{alg}

\begin{example} \label{exa:algo}
We illustrate Algorithm~\ref{alg:1} (especially the proofs of Theorem B and Lemmas~\ref{L:piv}, \ref{L:techC}) with the quadratic form $q$ of Dynkin type $\CC$ given by the following bigraph:
\[
\xymatrix@!0@C=30pt@R=30pt{ 1 \ar@{.}@(lu,ld) \ar@{=}[r]  \ar@{:}[d] & 2 \ar@{.}[d] \ar@{-}[ld] \\
3 \ar@{-}[r]  & 4}
\]
Since $q_1=2$, we are ready to find a rigid G-transformation $T'$ as in Lemma~\ref{L:piv} such that $q'=q\circ T'$ satisfies $q'_{1,i}>\delta_{1,i}$ for $i=1,\ldots,n$. We get $T'=T_{3,4}T_2$
 \[
 \xymatrix@!0@C=30pt@R=30pt{ 1 \ar@{.}@(lu,ld) \ar@{=}[r]  \ar@{:}[d] & 2 \ar@{.}[d] \ar@{-}[ld] \\
3 \ar@{-}[r]  & 4}
\, \xymatrix@!0@C=25pt@R=20pt{{} \\ {} \ar@{|->}[r]^-{T_{3,4}} & {} } \quad
\xymatrix@!0@C=30pt@R=30pt{ 1 \ar@{.}@(lu,ld) \ar@{:}[rd] \ar@{=}[r]  \ar@{:}[d] & 2  \ar@{-}[ld] \\
3 \ar@{.}[r]  & 4}
\, \xymatrix@!0@C=25pt@R=20pt{{} \\ {} \ar@{|->}[r]^-{T_{2}} & {} } \quad
\xymatrix@!0@C=30pt@R=30pt{ 1 \ar@{.}@(lu,ld) \ar@{:}[rd] \ar@{:}[r]  \ar@{:}[d] & 2  \ar@{.}[ld] \\
3 \ar@{.}[r]  & 4}
 \]
Following the cases 1-4 of Lemma~\ref{L:techC}, we obtain a partition $U$ of $\{1,\ldots,n\}$ with $U^2_{1,-1}=\{1\}$, $U^1_{2,+1}=\{2\}$, $U^1_{2,-1}=\{4\}$, $U^1_{3,+1}=\{3\}$ and $U^1_{3,-1}=\emptyset$. As indicated in the proof of Theorem B, a bidirected graph $\bdg'$ with $q'=q_{\bdg'}$ is given below, which is taken with the iterated transformation $\mathcal{T}_2 \circ \mathcal{T}_{3,4}$ into a bidirected graph $\bdg$ with $q=q_\bdg$,
\[
\bdg'=\xymatrix@!0@C=25pt@R=20pt{ & \bulito_2 \\
\bulito_1 \ar@{<->}@(lu,ld)_(.2)1 \ar@{<-}[rd]_-3 \ar@<-.5ex>@{<->}[ru]_-4  \ar@<.5ex>@{<-}[ru]^-2
\\  & \bulito_3 }
\xymatrix@!0@C=25pt@R=20pt{{} \\ {} \ar@{|->}[r]^-{\mathcal{T}_{2}} & {} } \,
\xymatrix@!0@C=25pt@R=20pt{ & \bulito_2 \\
\bulito_1 \ar@{<->}@(lu,ld)_(.2)1 \ar@{<-}[rd]_-3 \ar@<-.5ex>@{<->}[ru]_-4  \ar@<.5ex>[ru]^-2
\\  & \bulito_3 }
\xymatrix@!0@C=25pt@R=20pt{{} \\ {} \ar@{|->}[r]^-{\mathcal{T}_{3,4}} & {} } \,
\xymatrix@!0@C=25pt@R=20pt{ & \bulito_2 \ar@{<->}[dd]^-4 \\
\bulito_1 \ar@{<->}@(lu,ld)_(.2)1 \ar@{<-}[rd]_-3   \ar[ru]^-2
\\  & \bulito_3 }=\bdg
\]
In particular, $\Rnk(q)=3$ and $\CRnk(q)=1$.
\end{example}

\begin{alg}\label{alg:2} 
\noindent
{\bf Input:} a non-negative integral quadratic form $q:\Z^n \to \Z$ of Dynkin type $\CC_r$,


\noindent
{\bf Output:} a matrix $T\in\GlnZ$ and the canonical $(c_1,c_2)$-extension $\Delta:=\widehat{\mathcal{C}}^{(c_1,c_2)}_r=\Delta(\bdgC_r^{c_1,c_2})$ of $\mathcal{C}_r$ such that $q\Gweak^T q_\Delta$, $c_1+c_2=\crk(q)$ and $c_2=\dl(q)-1\geq 0$ (cf.~Theorem \ref{T:chc}).


\begin{enumerate}[label={\textnormal{\arabic*.}},topsep=3px,parsep=-1px]
\item As in Step 3 of Algorithm \ref{alg:1}, find a G-transformation $\widehat{T}$ such that $q\circ \widehat{T}=q_{\bdg'}$ for a bidirected graph $\bdg'$ (note that this is also Step 1 of the proof of Theorem \ref{T:chc}).
\item Compose the explicit transformations from Steps 2-7 of the proof of Theorem \ref{T:chc} to get a G-transformation $\widetilde{T}$ such that $q\circ(\widehat{T}\widetilde{T})=q_\Delta$ for $\Delta=\widehat{\mathcal{C}}^{(c_1,c_2)}_m$ (note that $m=r$). Return $\widehat{T}\widetilde{T}$ and $\Delta$.
\end{enumerate}
\end{alg}

\begin{example} \label{exa:algo2}
We illustrate Algorithm~\ref{alg:2} (especially the proof of Theorem~\ref{T:chc}) with the form $q$ of Example~\ref{exa:algo}. Recall from that example that $q$ has rank $3$, corank $1$ and only one (dotted) loop, thus $(c_1,c_2)=(1,0)$. We want to find a G-transformation $T$ such that $q\circ T$ coincides with the canonical $(1,0)$-extension of $\mathcal{C}_3$. The seven steps of the proof of Theorem~\ref{T:chc} yield the following transformations (cf. Example~\ref{exa:algo} for the first step):
\[
\xymatrix@!0@C=40pt{{} \\ {q\quad} \ar@{|->}[r]^-{T_{3,4}T_2}_-{\text{Step~1}} & {} } \,
\xymatrix@!0@C=25pt@R=20pt{ & \bulito_2 \\
\bulito_1 \ar@{<->}@(lu,ld)_(.2)1 \ar@{<-}[rd]_-3 \ar@<-.5ex>@{<->}[ru]_-4  \ar@<.5ex>@{<-}[ru]^-2
\\  & \bulito_3 }
\xymatrix@!0@C=30pt{{} \\ {} \ar@{|->}[r]^-{\mathcal{T}_{1,4} \circ\mathcal{T}_{4}}_-{\text{Step~2}} & {} } \,
\xymatrix@!0@C=25pt@R=20pt{ & \bulito_2 \\
\bulito_1 \ar@{<->}@(lu,ld)_(.2)1 \ar@{<-}[rd]_-3 \ar@<-.5ex>@{<-}[ru]_-4  \ar@{<-}@<.5ex>[ru]^-2
\\  & \bulito_3 }
\xymatrix@!0@C=30pt{{} \\ {} \ar@{|->}[r]^-{\mathcal{P}^{(2,3)}}_-{\text{Step~4}} & {} } \,
\xymatrix@!0@C=25pt@R=20pt{ & \bulito_2 \\
\bulito_1 \ar@{<->}@(lu,ld)_(.2)1 \ar@{<-}[rd]_-2 \ar@<-.5ex>@{<-}[ru]_-4  \ar@{<-}@<.5ex>[ru]^-3
\\  & \bulito_3 }
\xymatrix@!0@C=30pt{{} \\ {} \ar@{|->}[r]^-{\mathcal{T}_{1,4}}_-{\text{Step~5}} & {} } \,
\xymatrix@!0@C=25pt@R=20pt{ & \bulito_2 \\
\bulito_1 \ar@{<->}@(lu,ld)_(.2)1 \ar@{<-}[rd]_-2 \ar@<-.5ex>@{|-|}[ru]_-4  \ar@{<-}@<.5ex>[ru]^-3
\\  & \bulito_3 }
\xymatrix@!0@C=40pt{{} \\ {} \ar@{|->}[r]^-{\mathcal{T}_{2,1} \circ\mathcal{T}_{3,2}}_-{\text{Step~6}} & {\, \bdgC^{1,0}_3} } \,
\]
In this example, Steps~3 and~7 are empty. The wanted G-transformation is $T=T_{3,4}T_2T_{1,4}T_4P^{(2,3)}T_{1,4}T_{2,1}T_{3,2}$.
\end{example}

\begin{alg}\label{alg:3} 
\noindent
{\bf Input:}  a connected positive irreducible incidence form $q=q_\bdg$ of a  bidirected graph $\bdg$,


\noindent
{\bf Output:} the (finite) set $\CR_q(1)\subseteq\ZZ^n$ of solutions of Diophantine equation $q(x)=1$.


\begin{enumerate}[label={\textnormal{\arabic*.}},topsep=3px,parsep=-1px]
\item Since $q$ is positive then $\bdg$ is a tree or a 1-tree (cf.~the proof of Proposition \ref{prop:posfinite}).
\item If $\bdg$ is a tree, then apply Breadth-First Search to compute the (finite) set $\walksred(\bdg)$ of all reduced walks in $\bdg$ (cf.~Remark \ref{rem:reduced}). Return $\pm\inc(\walksred(\bdg))\setminus\{0\}$ (cf.~\eqref{eq:inc} and Theorem C).
\item If $\bdg$ is a 1-tree, then apply Depth-First Search to detect a non-trivial closed walk $\wlk$ and apply Breadth-First Search to compute the (finite) set $W$ of all reduced walks in $\bdg$ of the form $\wlk_s\wlk^k\wlk_t^{-1}$, $s<t$,  as in Case 2 of the proof of Proposition \ref{prop:posfinite} for $k=0,1$ (cf.~the formula \eqref{eq:Rinc1tree}). Return $\pm\inc(W)$.
\end{enumerate}
\end{alg}

All three  above algorithms can be implemented with polynomial arithmetic complexity (with respect to $n$), cf.~\cite{kZ20, MM19, MM21qa}; moreover, main arithmetic operations are performed on small integers,
cf.~Lemma \ref{lem:DynCbounds} and \cite[Proposition 5.1]{MM21qa}.
In particular, Algorithm \ref{alg:3} provides an efficient polynomial time procedure for solving Diophantine equations \eqref{EQ:dio} with $d=1$, compare with general exponential time Simson's \cite[Algorithm 3.7]{dS11} and \cite[Algorithm 4.2]{dS11a}. It seems that Algorithm \ref{alg:3} can be extended to non-negative incidence forms $q$ to compute \lq\lq bases'' of the (usually infinite) sets $\CR_q(d)$ for $d=0,1,2$, cf.~Example \ref{ex:onevertexbdgbasis} and \cite[Algorithm 3.9]{dS11}.

Finally, observe that a single Gabrielov transformation on a quadratic form or a bigraph takes linear time (see the formula \eqref{EQ:coe}). These elementary transformations are the main ingredient in the known algorithms reducing a non-negative unit or Cox-regular form to its canonical form, cf.~Theorem \ref{T:wea} and \cite{kZ20, MM19, MM21qa, BJP19, BP99, PR19}. Our approach allows to substitute Gabrielov transformations with the constant time transformations of bidirected graphs, see Definition \ref{D:tra}(a), cf.~Lemma \ref{L:tra}. This may lead to more efficient reduction algorithms.

\bibliographystyle{plainnat}

\end{document}